\documentclass[preprint]{1-elsarticle}
\usepackage{amsmath}
\usepackage{amssymb}
\usepackage{amsthm}
\usepackage{mathrsfs}

\newtheorem{thm}{Theorem}[section]
\newtheorem{lem}[thm]{Lemma}
\newtheorem{cor}[thm]{Corollary}
\newtheorem{defi}[thm]{Definition}
\newtheorem{prop}[thm]{Proposition}
\newtheorem{rema}[thm]{Remark}
\usepackage{color}
\usepackage{mathtools}
\definecolor{cyan}{cmyk}{1,0,0,0}
\definecolor{lightcyan}{cmyk}{0.5,0,0,0}
\definecolor{pastelcyan}{cmyk}{0.25,0,0,0}
\definecolor{magenta}{cmyk}{0,1,0,0}
\definecolor{yellow}{cmyk}{0,0,1,0}
\definecolor{lightyellow}{cmyk}{0,0,0.5,0}
\definecolor{pastelyellow}{cmyk}{0,0,0.25,0}
\definecolor{black}{cmyk}{0,0,0,1}
\definecolor{darkgray}{cmyk}{0,0,0,0.75}
\definecolor{gray}{cmyk}{0,0,0,0.5}
\definecolor{lightgray}{cmyk}{0,0,0,0.25}
\definecolor{white}{cmyk}{0,0,0,0}
\definecolor{red}{cmyk}{0,1,1,0}
\definecolor{orange}{cmyk}{0,0.5,1,0}
\definecolor{scarlet}{cmyk}{0,1,0.5,0}
\definecolor{brown}{cmyk}{0.5,0.75,1,0}
\definecolor{camel}{cmyk}{0.25,0.375,0.5,0}
\definecolor{cream}{cmyk}{0,0.2,0.3,0}
\definecolor{green}{cmyk}{1,0,1,0}
\definecolor{lightgreen}{cmyk}{0.5,0,0.5,0}
\definecolor{pastelgreen}{cmyk}{0.25,0,0.25,0}
\definecolor{mossgreen}{cmyk}{0.64,0.4,1,0}
\definecolor{yellowgreen}{cmyk}{0.5,0,1,0}
\definecolor{skyblue}{cmyk}{0.4,0.16,0,0}
\definecolor{royal}{cmyk}{1.0,0.5,0,0}
\definecolor{navyblue}{cmyk}{0.9,0.75,0.5,0}
\definecolor{lightnavy}{cmyk}{0.4,0.3,0.2,0}
\definecolor{blue}{cmyk}{1,1,0,0}
\definecolor{lightblue}{cmyk}{0.5,0.5,0,0}
\definecolor{lavender}{cmyk}{0.25,0.25,0,0}
\definecolor{violet}{cmyk}{0.75,1,0.25,0}
\definecolor{purple}{cmyk}{0.5,1,0.5,0}
\definecolor{pink}{cmyk}{0,0.5,0,0}
\definecolor{pastelpink}{cmyk}{0,0.25,0,0}
\newcommand\NN{{\mathbb N}}
\newcommand\RR{{{\mathbb R}}}

\newcommand\vk{{\bf k}}

\newcommand{\rr}{\mathbb{R}}

\newcommand{\cc}{\mathbb{C}}

\newcommand\cD{\mathcal D}
\newcommand\cE{\mathcal E}

\newcommand\cF{{\mathcal F}}
\newcommand\cL{{\mathcal L}}

\newcommand\cS{{\mathcal S}}
\newcommand\cA{{\mathcal A}}
\newcommand\cB{{\mathcal B}}
\newcommand\cK{{\mathcal K}}

\newcommand\cT{{\mathcal T}}
\newcommand\cH{{\mathcal H}}
\def\SS {\mathbb{S}}

\newcommand\ra{{\rangle}}
\newcommand\la{{\langle}}

\newcommand\pa{\partial}

\begin{document}
\title
{Global solutions in the critical Besov space
for the non cutoff  Boltzmann equation 
}

\author[M]{Yoshinori Morimoto \corref{cor1}}
\address[M]{Graduate School of Human and Environmental Studies,
Kyoto University,
Kyoto, 606-8501, Japan} \ead[M]
{morimoto@math.h.kyoto-u.ac.jp}

\author[S]{Shota Sakamoto}
\address[S]{Graduate School of Human and Environmental Studies,
Kyoto University,
Kyoto, 606-8501, Japan} \ead[S]
{sakamoto.shota.76r@st.kyoto-u.ac.jp}

\cortext[cor1]{corresponding author, tel:81-75-753-6737, fax:81-75-753-2929}


\date{}

\begin{abstract}
The Boltzmann equation is studied without the cutoff assumption. Under a perturbative setting, a unique global solution of the Cauchy problem of the equation is established in a critical Chemin-Lerner space. In order to analyse the collisional term of the equation, a Chemin-Lerner norm is combined with a non-isotropic norm with respect to a velocity variable, which yields an apriori estimate for an energy estimate. Together with local existence following from commutator estimates and the Hahn-Banach extension theorem, the desired solution is obtained. Also, the non-negativity of the solution 
is rigorously shown.
\end{abstract}

\maketitle

\section{Introduction}
We consider the Boltzmann equation in $\mathbb{R}^3$, 
\begin{align}\label{eq: Boltzmann}
\begin{cases}
\partial_t f(x,v,t)+v\cdot \nabla_x f(x,v,t)=Q(f,f)(x,v,t), \\
f(x,v,0)=f_0(x,v),
\end{cases}
\end{align}
where $f=f(x,v,t)$ is a density distribution of particles in a rarefied gas with position $x \in \mathbb{R}^3$ and velocity $v \in \mathbb{R}^3$ at time $t \ge 0$, and $f_0$ is an initial datum.

The aim of this paper is to show the global existence and uniqueness of solution to \eqref{eq: Boltzmann} in an appropriate Besov space under a perturbative framework. There are many references concerning well-posedness of solutions to the Boltzmann equation around the global  Maxwellian equilibrium in different spaces. Recently, Duan, Liu, and Xu \cite{DLX} proved that the Boltzmann equation has a unique global strong solution near Maxwellian in the Chemin-Lerner space $\tilde{L}^\infty_T \tilde{L}^2_v (B^{3/2}_{2,1})$ (defined later: this space is a space-velocity-time Besov space) under the Grad's angular cutoff assumption. 
This was the first result which applied Besov spaces to the Cauchy problem of the Boltzmann equation 
around  the equilibrium. In this paper, we consider \eqref{eq: Boltzmann} without angular cutoff by using 
the triple norm, introduced by Alexandre, Morimoto, Ukai, Xu, and Yang (AMUXY in what follows) \cite{AMUXY1, AMUXY2}.  The triple norm was originally 
adopted to discuss the equation in $L^\infty_t (0,\infty; H_l^k(\mathbb{R}^6_{x,v}))$ 
for suitable $l, k \in \mathbb{N}$.
We combine both ideas \cite{DLX} and \cite{AMUXY2}(see also \cite{amuxy4-2, amuxy4-3, AMUXY3}) 
to discuss the Cauchy problem without angular cutoff. 
For this purpose, we later introduce a Chemin-Lerner type triple norm, which enables us to take 
full advantage of the two ideas.

We discuss \eqref{eq: Boltzmann} in greater detail. The bilinear operator $Q(f,g)$ is called the Boltzmann collision operator and is defined by
\begin{align*}
Q(f,g)(v)=\int_{\mathbb{R}^3}\int_{\mathbb{S}^2} B(v-v_*,\sigma) (f_*'g'-f_*g)d\sigma dv_*,
\end{align*}
where $f_*'=f(v_*')$, $g'=g(v')$, $f_*=f(v_*)$, $g=g(v)$, and
\begin{align*}
v'=\frac{v+v_*}{2}+\frac{\vert v-v_*\vert}{2}\sigma,\ v_*'=\frac{v+v_*}{2}-\frac{\vert v-v_* \vert}{2}\sigma\quad (\sigma \in \mathbb{S}^2).
\end{align*}
The operator $Q$ acts only on $v$ and $Q(f,g)(x,v,t):=Q(f(x,\cdot,t),g(x,\cdot,t))(v)$. 
A collision kernel $B$ is correspondently defined to reflect various physical phenomena. The first assumption is
\begin{align*}
B(v-v_*, \sigma)=b\left(\frac{v-v_*}{\vert v-v_* \vert}\cdot \sigma\right) 
\Phi_\gamma(|v-v_*|)
\ge 0,
\end{align*}
where $\Phi_\gamma(|v-v_*|)=
\vert v-v_*\vert^\gamma$ for $\gamma > -3$. This model is physically derived when the particle interaction obeys the inverse power law $\phi(r)=r^{-(p-1)}$ ($2<p<\infty$), where $r$ is a distance between two interacting particles. Before mentioning to the second assumption, we should recall that $b$ is a function of the deviation angle $\theta \in [0,\pi]$, defined by $\cos \theta =\frac{v-v_*}{\vert v-v_* \vert}\cdot \sigma$, so we write $b\left(\frac{v-v_*}{\vert v-v_* \vert}\cdot \sigma\right)=b(\cos\theta)$. As an usual convention, $b(\cos \theta)$ $(\theta \in [0, \pi])$ is replaced by
\begin{equation*}
\left[b(\cos \theta)+b(\cos(\pi-\theta))\right] \mathbf{1}_{\{ 0\le \theta \le \pi/2\}}
\end{equation*}
so that a domain of $b$ is $[0,\pi/2]$.
The second one is called the non-cutoff assumption: for some $\nu \in (0,2)$ and $K>0$,
\begin{align*}
b\left(\frac{v-v_*}{\vert v-v_* \vert}\cdot \sigma\right) = b(\cos \theta)\sim K \theta^{-2-\nu}\ (\theta \downarrow 0),
\end{align*}
which implies that $b$ is not integrable near $\theta=0$. 
In the physical model of the inverse power law potential,  
relations among $\gamma$, $\nu$, and $p$ are given by $\gamma=(p-5)/(p-1)$ and $\nu=2/(p-1)$. In this paper, we particularly consider the case 
when 
$$\gamma > \max\{-3, -3/2 -\nu\}. $$ The latter component of this below estimate is due to a technical reason which will appear when we derive various estimates for the collision operator and a nonlinear term derived from $Q$. However, it is easily verified that $-1 < \gamma + \nu <1$ as long as $\gamma = (p-5)/(p-1)$ and $\nu = 2/(p-1)$, that is, all the physically possible combinations of $\gamma$ and $\nu$ are contained in our restriction.

We consider the Boltzmann equation around the normalized Maxwellian 
\begin{align*}
\mu=\mu(v)=(2\pi)^{-3/2}e^{-\vert v \vert^2/2},
\end{align*}
thus we set $f=\mu + \mu^{1/2}g$, so that \eqref{eq: Boltzmann} is equivalent to the equation
\begin{align}\label{eq: perturbation}
\partial_tg+ v\cdot \nabla_x g +\cL g = \Gamma(g,g), 
\end{align}
with initial datum $g_0$ defined from $f_0 = \mu + \mu^{1/2}g_0$, 
where  
\begin{align*}
\Gamma (g,h)=\mu^{-1/2} Q(\mu^{1/2}g, \mu^{1/2}h),\quad \cL g=\cL_1g+\cL_2g:=-\Gamma(\mu^{1/2}, g)-\Gamma(g, \mu^{1/2}).
\end{align*}
These are the nonlinear and the linear term of the Boltzmann equation, respectively. It is known that $\cL$ is nonnegative-definite on $L^2_v(\mathbb{R}^3)$ and 
\begin{equation*}
\ker{\cL}=\text{span} \{\mu^{1/2}, \mu^{1/2}v_i\ (i=1,2,3), \mu^{1/2}\vert v \vert\}. 
\end{equation*}
The projection operator on $\ker{\mathcal{L}}$ is denoted by $\mathbf{P}$, therefore for each $g$,
\begin{align*}
\mathbf{P}g(x,v,t)=\left[ a(x,t)+v\cdot b(x,t)+\vert v \vert^2c(x,t)\right] \mu^{1/2}(v).
\end{align*}
The decomposition $g=\mathbf{P}g+(\mathbf{I}-\mathbf{P})g$ is called the macro-micro decomposition, where $\mathbf{I}$ is the identity map.

The energy term and the dissipation term are respectively defined by
\begin{align*}
\mathcal{E}_T(g) &\sim \Vert g \Vert_{\tilde{L}^\infty_T \tilde{L}^2_v(B^{3/2}_x)}\\
\intertext{and}
\mathcal{D}_T(g) &= \Vert \nabla_x(a,b,c) \Vert_{\tilde{L}^2_T(B^{1/2}_x)} + \Vert (\mathbf{I}-\mathbf{P})g\Vert_{\mathcal{T}^{3/2}_{T,2,2}}.
\end{align*}
Here and hereafter, $B^s_x:= B^s_{2,1}(\RR^3_x)$. We have just showed hidden indices, and the above norms will be rigorously defined in Section 2.

The main theorem of this paper is the following.
\begin{thm}\label{main theorem}
Let $0< \nu <2$ and $\gamma > \max\{ -3, -\nu -3/2\}$. There are constants $\varepsilon_0>0$ and $ C>0$ such that if
\begin{equation*}
\Vert g_0 \Vert_{\tilde{L}^2_v (B^{3/2}_x)}\le \varepsilon_0,
\end{equation*}
then there exists a unique global solution $g(x,v,t)$ of \eqref{eq: perturbation} with initial datum $g_0(x,v)$. This solution satisfies 
\begin{equation}\label{apriori-estimate-0}
\mathcal{E}_T(g) +\mathcal{D}_T(g) \le C \Vert g_0 \Vert_{\tilde{L}^2_v (B^{3/2}_x)}
\end{equation}
for any $T>0$. Moreover if $f_0(x,v) = \mu + \sqrt \mu g_0(x,v)  \ge 0$ then 
$f(x,v,t) = \mu + \sqrt \mu g(x,v,t)  \ge 0$.
\end{thm}

Let us make some comments on the theorem. We note that Besov embedding $B^{3/2}_{2,1}(\mathbb{R}^3) \hookrightarrow L^\infty (\mathbb{R}^3)$ ensures that $B^{3/2}_{2,1}(\mathbb{R}^3)$ is a Banach algebra, that is, if $f$ and $g$ are elements of $B^{3/2}_{2,1}(\mathbb{R}^3)$, then so is the product $fg$ (see Corollary 2.86 of \cite{BCD} for instance). This is the reason why we employed this space with respect to $x$, combined with $L^\infty_T L^2_v$. In this sense, we used the word \textit{critical} in the title; compare with the well-known fact that we need $s>3/2$ for $H^s(\mathbb{R}^3) \hookrightarrow L^\infty (\mathbb{R}^3)$. We also remark that the Chemin-Lerner space $\tilde{L}^\infty_T \tilde{L}^2_v (B^{3/2}_{2,1})$ defines stronger topology than $L^\infty_T L^2_v (B^{3/2}_{2,1})$. Thus we shall adopt the Chemin-Lerner norm to the problem.

We review some known results, confining  ourselves to research on the Cauchy problem around the global Maxwellian. The first fruitful work was done by Ukai \cite{U1}, \cite{U2}, 
based on spectral analysis of the collision operator and the bootstrap argument. This is the first result 
obtaining a global mild solution to the equation. Notable works following this perspective are, 
for example, Nishida and Imai \cite{NI}, and Ukai and Asano \cite{UA}. Most of these results were secured in 
the 70's and the 80's. They adopted $L^\infty$-norm to estimate a solution with respect to $v$. 
Recent study based on $L^2_v$-framework has started from similar but independent results 
by both Guo\cite{G} and  Liu and Yu\cite{LY}, Liu, Yang, and Yu\cite{LYY}. They applied the macro-micro decomposition, which clarifies dissipative and equlibrating effects of the equation. 
This allowed us to utilize the energy method. They succeeded in constructing 
a classical solution in a $x$- and $v$-Sobolev space. 

There are two independent results concerning global existence around the equilibrium for the non-cutoff Boltzmann equation. 
To establish a global solution to the non cut-off Boltzmann equation in the whole space, AMUXY \cite{AMUXY1, amuxy4-2, AMUXY2} followed Guo's framework with intensive pseudo-differential calculus on the collisional term,
in both general hard potential case $\gamma + \nu >0$ and general soft potential case $\gamma + \nu \le 0$.
More precisely, when $\gamma +\nu >0$,  \cite{amuxy4-2} constructed global solutions  in 
$L^\infty_t (0,\infty; H_l^k(\mathbb{R}^6_{x,v}))$ with $k\ge 6$ and $l>3/2+\gamma + \nu$. Here, $H^k_l(\mathbb{R}^6_{x,v}) =\{ f\in \mathcal{S}' (\mathbb{R}^6_{x,v}); \langle v \rangle^l f \in H^k (\mathbb{R}^6_{x,v})\}$. 
For the general soft potential case 
under an additional condition 
$\gamma > \max\{ -3, -\nu -3/2\}$,  
AMUXY \cite{AMUXY2} found a solution in $L^\infty_t (0,\infty; L^2_v (\mathbb{R}^3 ; H^N_x (\mathbb{R}^3)))$ with 
$\NN \ni N \ge 4$ and another solution in $L^\infty_t (0,\infty; \tilde{\cH}_{\ell}^k(\RR^6_{x,v}))$
with $4 \le k \in \NN, \ell \ge k$, where 
\[\tilde{\cH}_{\ell}^k(\RR^6_{x,v}) = \{f \in \cS'(\RR^6_{x,v}); \sup_{|\alpha + \beta|\le k} \|\la v \ra^{(\ell-|\beta| )|\gamma+ \nu|}
\pa_x^\alpha \pa_v^\beta f\|^2_{L^2(\RR^6)} < \infty\}.\]
AMUXY \cite{amuxy4-3} also studied for qualitative
properties including smoothing effect, uniqueness, non-negativity, and convergence rate towards 
the equlibrium. When the space variable moves in a torus, the global existence was shown by Gressman and Strain \cite{GS} with convergence rates.
They employed an anisotropic 
metric 
on the ``lifted" paraboloid $d(v,v')=\sqrt{\vert v-v'\vert^2 + (\vert v\vert^2 -\vert v'\vert^2)^2/4}$ 
to capture changes in the power of the weight. 
For the general hard potential case, their result is sharper than \cite{amuxy4-2} and 
covers the global well-posedness in non-weighted and non-velocity derivative energy 
functional spaces $L^\infty_t (0,\infty; L^N_v (\mathbb{R}^3 ; H^2_x (\mathbb{T}^3)))$ with $N \ge 2$. 
It should be mentioned that the soft potential case was also studied  by \cite{GS} in the space
$L^\infty_t (0,\infty; \tilde{\cH}_{\ell}^k(\mathbb{T}^3_x \times \RR^3_{v}))$
with $4 \le k \in \NN, \ell \ge 0$, 
without additional condition $\gamma > \max\{ -3, -\nu -3/2\}$.
The main theorem of  the present paper improves the results
in  non-weighted and non-velocity derivative energy 
functional spaces, given in \cite{AMUXY2} for the general soft potential case and  in \cite{GS}
for the general hard potentical case, respectively.

Our work is motivated by Duan, Liu, and Xu \cite{DLX}. Under the cut-off assumption, they studied the hard potential case $\gamma >0$ by using the Chemin-Lerner space 
$\tilde{L}^\infty_T \tilde{L}^2_v (B^{3/2}_x)$, which was originally invented in \cite{CL} to investigate 
the existence of the flow for solutions of the Navier-Stokes equations.
One advantage of this time-velocity-space Besov space over a usual space $L^\infty_T L^2_v B^{3/2}_x$ 
can be explained as follows: in general, it is  easier to bound dyadic blocks of 
the Littlewood-Paley decomposition in $L^\infty_T L^2_v L^2_x$ rather than to estimate 
directly the solution for the equation in  $L^\infty_T L^2_v B^{3/2}_x$.  
Since the global well-posed theory in $\tilde{L}^\infty_T \tilde{L}^2_v (B^{3/2}_x)$ was established  by
\cite{DLX} for the cutoff case, 
we extend their methods to the non cut-off 
case by using upper bound estimates of the nonlinear term and calculation of commutators which were devised in a series of AMUXY's works.

There are some other papers applying the Besov space to analysis of the Boltzmann equation. 
We will simply introduce some of them concerning the Cauchy problem. Recently, Tang and Liu \cite{TL} 
followed \cite{DLX} under the cut-off assumption and improved the result by replacing $B^{3/2}_{2,1}$ 
with narrower spaces $B^s_{2,r}$ ($(s,r)\in (3/2,\infty) \times [1,2]$ or $=(3/2, 1)$). 
For other research, see  Alexandre \cite{A}, Ars\'{e}nio and Masmoudi \cite{AM}, Sohinger and Strain \cite{SS}.

This paper is organized as follows. In Section 2, we review definitions of the Littlewood-Paley decomposition, Besov spaces, and Chemin-Lerner spaces. Taking into account the triple norm defined in 
\cite{AMUXY1, AMUXY2}, 
we also define a Chemin-Lerner type triple norm. In Section 3, we deduce Besov-type trilinear estimates. 
We have to derive three different estimates of 
Besov-type for the nonlinear term $\Gamma$, which lead us to the global and local existence
of solutions.
Since $\sigma$-integration on the unit sphere is significant in the non cutoff case,
the nonlinear term can not be estimated by separating the so-called gain and loss
terms, different from the cut-off case in \cite{DLX}.
To get estimates analogous to those in Lemma 3.1 of \cite{DLX}, 
the $\sigma$-integration should be dealt with by the triple norm, involving two velocity variables $v, v'$. 
Section 4 is devoted to deduce an a priori estimate for the global existence. The microscopic part $(\mathbf{I}-\mathbf{P})g$ is handled by the estimates established in Section 3, however, 
we cannot adopt them for the macroscopic part $\mathbf{P}g$. To overcome this difficulty, we will bring in a fluid-type system of $(a, b, c)$. Applying the energy method to it, we will acquire an estimate of the macroscopic dissipation term. Local existence of a solution will be shown in Section 5. We want to establish 
a solution by the duality argument and the Hahn-Banach extension theorem, but a problem arises since, 
as far as we know, the dual space of $\tilde{L}^\infty_T \tilde{L}^2_v(B^{3/2}_x)$ is unknown. Therefore, 
we first find a solution to a linearized equation in a wider space $L^\infty(0,T; L^2(\mathbb{R}^6_{x,v}))$ and construct a approximate sequence. Then we show that a limit of the sequence is a solution to the equation, 
and it belongs to a suitable solution space indeed. This is our strategy, however, we have to employ 
very delicate commutator estimates involving the nonlinear term $\Gamma$ and cut-off functions 
with respect to both $x$ and $v$. Proofs of these estimates are fairly technical, so in this section 
we postpone proving them and focus on solving the equation assuming that the necessary lemmas 
are evidenced. Such lemmas are collected in Appendix with proofs. Non-negativity of the solution 
obtained will be given in Section 6. Known inequalities of Besov spaces and the triple norm used in 
this paper are also collected in Appendix.

\section{Preliminaries}\label{S2}
\setcounter{equation}{0}
In this section, we define some function spaces for later use. Readers may consult \cite{BCD} for this topic. First, we introduce a dyadic partition of unity, also known as the Littlewood-Paley decomposition. Let $\mathcal{C}$ be the annulus $\{ \xi \in \mathbb{R}^3 \ \vert\  3/4 \le \vert \xi \vert \le 8/3\}$ and $\mathcal{B}$ be the ball $B(0,4/3)$. Then there exist radial functions $\chi \in C^\infty_0(\mathcal{B})$ and $\phi \in C^\infty_0(\mathcal{C})$ satisfying the following conditions:
\begin{gather*}
0 \le \chi(\xi), \phi(\xi) \le 1 \quad \mathrm{for \ any} \ \xi \in \mathbb{R}^3,\\
\chi(\xi)+\sum_{q \ge 0} \phi(2^{-q}\xi)=1\quad \mathrm{for \ any}\   \xi \in \mathbb{R}^3,\\
\sum_{q \in \mathbb{Z}} \phi(2^{-q}\xi)=1\quad \mathrm{for\  any}\  \xi \in \mathbb{R}^3\setminus\{0\},\\
\vert q-q'\vert \ge 2 \Rightarrow  \mathrm{supp}\ \phi(2^{-q}\cdot)\cap\mathrm{supp}\ \phi(2^{-q'}\cdot) =\emptyset,\\
q\ge 1 \Rightarrow \mathrm{supp}\ \chi \cap \mathrm{supp}\ \phi(2^{-q}\cdot )=\emptyset,\\
\vert q-q'\vert \ge 5 \Rightarrow 2^{q'}\tilde{\mathcal{C}}\cap 2^q\mathcal{C}=\emptyset,
\end{gather*}
where $\tilde{\mathcal{C}}:=B(0,2/3)+\mathcal{C}$.
We fix these functions and write $h:=\mathcal{F}^{-1}\phi$ and $\tilde{h}:=\mathcal{F}^{-1}\chi$. For each $f \in \mathcal{S}'(\mathbb{R}^3_x)$, the nonhomogeneous dyadic blocks $\Delta_q$ are defined by
\begin{align*}
&\Delta_{-1}f:= \chi(D)f =\int_{\mathbb{R}^3} \tilde{h}(y)f(x-y)dy,\\
&\Delta_qf:= \phi(2^{-q}D)f= 2^{3q}\int_{\mathbb{R}^3} h(2^{q}y) f(x-y) dy \quad (q\in \mathbb{N}\cup \{0\})
\end{align*}
and $\Delta_qf:=0$ if $q\le -2$. The nonhomogeneous low-frequency cutoff operator $S_q$ is defined by
\begin{align*}
S_qf=\sum_{q' \le q-1} \Delta_{q'} f.
\end{align*}

We denote the set of all polynomials on $\mathbb{R}^3$ by $\mathcal{P}$. Regarding any polynomial $P$ as a distribution, we notice that  $\dot{\Delta}_q P =0$. Therefore, the homogeneous dyadic blocks $\dot{\Delta}_q$ are defined by
\begin{align*}
\dot{\Delta}_qf:= \phi(2^{-q}D)f:=2^{3q} \int_{\mathbb{R}^3} h(2^{q}y)f(x-y)dy
\end{align*}
for any $f \in \mathcal{S}'(\mathbb{R}^3_x)/\mathcal{P}$ and $q \in \mathbb{Z}$. 

We now give the definition of nonhomogeneous Besov spaces as follows.
\begin{defi}
Let $1 \le p,r \le \infty$ and $s \in \mathbb{R}$. The nonhomogeneous Besov space $B^s_{pr}$ is defined by
\begin{align*}
B^s_{pr}:= \left\{ f \in \mathcal{S}'(\mathbb{R}^3_x) \ \vert \ \Vert f \Vert_{B^s_{pr}} := \left\Vert (2^{qs} \Vert \Delta_q f \Vert_{L^p_x} )_{q \ge -1} \right\Vert_{\ell^r} < \infty \right\}.
\end{align*}
When $r=\infty$, we set $\Vert f \Vert_{B^s_{p\infty}}:={\displaystyle\sup_{q \ge -1}2^{qs}} \Vert \Delta_q f \Vert_{L^p_x} $.
\end{defi}
Here $L^p_x:=L^p(\mathbb{R}^3_x)$ and this kind of abuse of notation will be used throughout the paper.

The definition of homogeneous Besov spaces is as follows:
\begin{defi}
Let $1 \le p,r \le \infty$ and $s \in \mathbb{R}$. The homogeneous Besov space $\dot{B}^s_{pr}$ is defined by
\begin{align*}
\dot{B}^s_{pr}:= \left\{ f \in \mathcal{S}'(\mathbb{R}^3_x)/ \mathcal{P} \ \vert \ \Vert f \Vert_{\dot{B}^s_{pr}} := \left\Vert (2^{qs} \Vert \dot{\Delta}_q f \Vert_{L^p_x} )_{q \in \mathbb{Z}} \right\Vert_{\ell^r} < \infty \right\}.
\end{align*}
When $r=\infty$, we set $\Vert f \Vert_{\dot{B}^s_{pr}} := {\displaystyle\sup_{q \in \mathbb{Z}}2^{qs}} \Vert \dot{\Delta}_q f \Vert_{L^p_x} $.
\end{defi}

For simplicity of notations, we will write $B^s_{2,1}=:B^s_x$ and $\dot{B}^s_{2,1}=:\dot{B}^s_x$,
as stated in the introduction. 

Next, we define Chemin-Lerner spaces, which is a generalization of the Besov space.
\begin{defi}
Let $ 1\le p,r,\alpha, \beta \le \infty$ and $s\in \mathbb{R}$. For $T \in [0,\infty)$, the Chemin-Lerner space $\tilde{L}^\alpha_T \tilde{L}^\beta_v (B^s_{pr})$ is defined by
\begin{align*}
\tilde{L}^\alpha_T \tilde{L}^\beta_v (B^s_{pr}):= \left\{ f(\cdot,v,t) \in \mathcal{S}' \ \vert\ \Vert f \Vert_{\tilde{L}^\alpha_T \tilde{L}^\beta_v (B^s_{pr})}  < \infty \right\},
\end{align*}
where
\begin{gather*}
\Vert f \Vert_{\tilde{L}^\alpha_T \tilde{L}^\beta_v (B^s_{pr})} := \left\Vert (2^{qs} \Vert \Delta_q f \Vert_{L^\alpha_T L^\beta_v L^p_x} )_{q \ge -1} \right\Vert_{\ell^r},\\
\Vert \Delta_q f \Vert_{L^\alpha_T L^\beta_v L^p_x}:= \left( \int_0^T \left( \int_{\mathbb{R}^3} \left( \int_{\mathbb{R}^3} \vert \Delta_q f(x,v,t) \vert^p dx \right)^{\beta/p} dv \right)^{\alpha/\beta} dt \right)^{1/\alpha}
\end{gather*}
with the usual convention when at least one of $p, r, \alpha, \beta$ is equal to $\infty$.
We also define $\tilde{L}^\alpha_T \tilde{L}^\beta_v (\dot{B}^s_{pr})$ similarly.
\end{defi}

We denote $\tilde{L}^\alpha_T \tilde{L}^\beta_v (B^s_{2,1})$ by $\tilde{L}^\alpha_T \tilde{L}^\beta_v (B^s_x)$, and $\tilde{L}^\alpha_T \tilde{L}^\beta_v (\dot{B}^s_{2,1})$ by $\tilde{L}^\alpha_T \tilde{L}^\beta_v (\dot{B}^s_x)$.
Finally, we give the definition of the non-isotropic norm $|\!|\!| \cdot |\!|\!|$ and the space $\mathcal{T}^s_{Tpr}$ and $ \dot{\mathcal{T}}^s_{Tpr}$, which are endowed with the ``Chemin-Lerner type triple norm''.

\begin{defi}
Let $1\le p, r \le \infty$, $T>0$ and $s\in \mathbb{R}$. $|\!|\!| f |\!|\!|$ is defined by
\begin{align*}
|\!|\!| f |\!|\!|^{2} :=& \iiint B(v-v_*,\sigma) \mu_* (f'-f)^2 dvdv_*d\sigma \\
&+ \iiint B(v-v_*,\sigma) f^2_* (\sqrt{\mu'}-\sqrt{\mu})^2 dvdv_* d\sigma\\
=& J_1^{\Phi_\gamma}(g) + J_2^{\Phi_\gamma}(g)\,,
\end{align*}
and the space $\mathcal{T}^s_{Tpr}$ is defined by
\begin{align*}
\mathcal{T}^s_{Tpr} := \left\{ f \ \vert \ \Vert f \Vert_{\mathcal{T}^s_{Tpr}} = \left\Vert (2^{qs} \Vert |\!|\!| \Delta_q f |\!|\!| \Vert_{L^{p}_T L^r_x} )_{q \ge -1} \right\Vert_{\ell^{1}} <\infty \right\}.
\end{align*}
$\dot{\mathcal{T}}^s_{Tpr}$ is defined in the same manner.
\end{defi}

$|\!|\!| \cdot |\!|\!|$ is called the triple norm, and it is known that this norm  is estimated from both above and below by weighted Sobolev norms (see \cite[Proposition 2.2]{AMUXY2}):
\begin{equation}\label{up-down-triple}
\Vert f \Vert^2_{H^{\nu/2}_{\gamma/2}} +\Vert f \Vert^2_{L^2_{(\nu+\gamma)/2}} \lesssim  |\!|\!| f |\!|\!|^{2} \lesssim \Vert f \Vert^2_{H^{\nu/2}_{(\nu+\gamma)/2}}.
\end{equation}
If $p$ or $r$ is an explicit number, $\mathcal{T}^s_{Tpr}$ is denoted by $\mathcal{T}^s_{T,p,r}$ to avoid ambiguity.

In order to deduce Chemin-Lerner estimates in the following sections, we will use some properties of the above spaces. See Appendix for these properties. It must be emphasized again that $\Vert \cdot \Vert_{\tilde{L}^\alpha_T \tilde{L}^\beta_v (B^s_{pr})}$ is usually easier to handle than $\Vert \cdot \Vert_{L^\alpha_T L^\beta_v (B^s_{pr})}$.

\section{Chemin-Lerner type trilinear estimates}\label{3}
\setcounter{equation}{0}
We will show an imporatnt estimate of the nonlinear term $\Gamma$ in the suitable Chemin-Lerner space. This estimate will be used many times throughout the paper. We denote the usual $L^2(\mathbb{R}^3_x \times \mathbb{R}^3_v)$ and $L^2(\mathbb{R}^3_v)$ inner product by $(\cdot,\cdot)_{x,v}$ and $(\cdot, \cdot)_v$, respectively.
\begin{lem}\label{trilinear}
Let $0<s\le \frac{3}{2}$ and $0<T \le \infty$. Define 
\begin{equation*}
A_T(f, g, h):=\sum_{q \ge -1} 2^{qs}\left( \int^T_0 \vert ( \Delta_q \Gamma (f,g), \Delta_qh)_{x,v} \vert dt \right)^{1/2}.
\end{equation*}
Then the following trilenear estimates hold:

\begin{align}\label{tri1}
A_T&(f, g, h) \lesssim \Vert f \Vert_{\tilde{L}^\infty_T \tilde{L}^2_v(B^{3/2}_x)}^{1/2} \Vert g \Vert_{\mathcal{T}^{3/2}_{T,2,2}}^{1/2} \Vert h \Vert_{\mathcal{T}^s_{T,2,2}}^{1/2} , \\
\label{tri2} A_T&(f, g, h)  \lesssim   \left( \Vert f \Vert_{L^2_T L^2_v(\dot{B}^{3/2}_x)}^{1/2} \Vert g \Vert_{\mathcal{T}^s_{T,\infty,2}}^{1/2} + \Vert f \Vert_{\tilde{L}^\infty_T \tilde{L}^2_v(B^s_x)}^{1/2} \Vert g \Vert_{\dot{\mathcal{T}}^{3/2}_{T,2,2}}^{1/2} \right)\Vert h \Vert_{\mathcal{T}^s_{T,2,2}}^{1/2}.
\end{align}
Moreover, if $\gamma + \nu \le 0$, we have
\begin{align}\label{tri3}
A_T(f, g, h) \notag
 \lesssim &\, \bigg(\Vert \mu^{1/10}f \Vert_{L^2_T L^2_v(\dot{B}^{3/2}_x)}^{1/2} \Vert g \Vert_{\mathcal{T}^s_{T,\infty,2}}^{1/2}
\\
 &\quad \qquad +\Vert \mu^{1/10}f \Vert_{\tilde{L}^\infty_T \tilde{L}^2_v(B^s_x)}^{1/2} \Vert g \Vert_{\dot{\mathcal{T}}^{3/2}_{T,2,2}}^{1/2}\bigg)\Vert h \Vert_{\mathcal{T}^s_{T,2,2}}^{1/2}
\notag\\
 & +\bigg( \Vert f \Vert_{L^2_T L^2_{v, (\nu+\gamma)/2}(\dot{B}^{3/2}_x)}^{1/2} \Vert g \Vert_{\tilde{L}^\infty_T \tilde{L}^2_v(B^{s}_x)}^{1/2}
\notag\\
 &\quad \qquad +\Vert f \Vert_{\tilde{L}^\infty_T \tilde{L}^2_{v, (\nu+\gamma)/2}(B^s_x)}^{1/2} \Vert g \Vert_{L^2_T L^2_v(\dot{B}^{3/2}_x)}^{1/2}\bigg)\Vert h \Vert_{\mathcal{T}^s_{T,2,2}}^{1/2}.
\end{align}

\end{lem}

\begin{proof}
Before starting the proof of this lemma, we recall the Bony decomposition. Since $\sum_j \Delta_j =\mathrm{Id}$, we have, at least formally, $fg=\sum_{j,j'} \Delta_j f \Delta_{j'} g$. Dividing this summation according to the frequencies, we obtain the following Bony decomposition:
\begin{align*}
 fg=\mathcal{T}_fg+\mathcal{T}_gf+\mathcal{R}(f,g),
\end{align*}
where 
\begin{align*}
\mathcal{T}_fg:=\sum_j S_{j-1}f\Delta_jg,\ \mathcal{T}_gf := \sum_j \Delta_j f S_{j-1} g,\, \\
\mathcal{R}(f,g):=\sum_j \left(\sum_{\vert j-j'\vert \le 1} \Delta_{j'}f\Delta_jg\right).
\end{align*}

Using this decomposition, we divide the inner product into three parts:
\begin{align*}
( \Delta_q \Gamma (f,g), \Delta_qh) = ( \Delta_q(\Gamma^1(f,g)+\Gamma^2(f,g)+\Gamma^3(f,g)), \Delta_q h),
\end{align*}
where $\Gamma^1(f,g) :=\sum_{j} \Gamma(S_{j-1}f, \Delta_j g) = \sum_j \Gamma^1_j(f,g)$, and $\Gamma^2(f,g)$ and $\Gamma^3(f,g)$ are similarly defined.

First, we treat $\Gamma^1(f,g)$. From shapes of supports of $\chi$ and $\phi(2^{-j}\cdot)$, we notice
\begin{align*}
\Delta_q \sum_j (S_{j-1}f\Delta_jg) = \Delta_q \sum_{\vert j-q \vert \le 4} (S_{j-1}f\Delta_jg),
\end{align*} 
that is, when $\Delta_q$ is applied to, the above summation with respect to $j$ becomes finite.
Hence we have
\begin{align*}
\left(\Delta_q \sum_j \Gamma^1_j(f,g), \Delta_q h \right)_{x,v} &=\sum_{\vert j-q \vert \le 4} (\Delta_q \Gamma^1_j(f,g), \Delta_q h)_{x,v} \\ \notag
&= \sum_{\vert j-q \vert \le 4} ( \Gamma^1_j(f,g), \Delta_q^2 h)_{x,v} \\ \notag
& \lesssim \sum_{\vert j-q \vert \le 4}\int_{\mathbb{R}^3} \Vert S_{j-1}f \Vert_{L^2_v} |\!|\!| \Delta_j g |\!|\!| |\!|\!| \Delta_q^2 h |\!|\!| dx \\ \notag
&\lesssim \sum_{\vert j-q \vert \le 4} \Vert f \Vert_{L^2_v L^\infty_x} \Vert |\!|\!| \Delta_j g |\!|\!| \Vert_{L^2_x} \Vert |\!|\!| \Delta_q h |\!|\!| \Vert_{L^2_x}, \notag
\end{align*}
where we used Corollary \ref{AMUXYtrilinear} and Lemma \ref{bounded operators}. We also used the inequlaity 
$$ \Vert |\!|\!| \Delta_q^2 h |\!|\!| \Vert_{L^2_x} \lesssim \Vert |\!|\!| \Delta_q h |\!|\!| \Vert_{L^2_x}, $$
which is verified by direct calculation.
Hence, we have so far
\begin{align}\label{ineq: AfterTrilinear1}\notag
&\sum_{q \ge -1} 2^{qs}\left( \int^T_0 \vert ( \Delta_q \Gamma^1 (f,g), \Delta_qh)_{x,v} \vert dt \right)^{1/2} \\
&\lesssim \sum_{q \ge -1}2^{qs} \left(\int_0^T \sum_{\vert j-q \vert \le 4} \Vert f \Vert_{L^2_v L^\infty_x} \Vert |\!|\!| \Delta_j g |\!|\!| \Vert_{L^2_x} \Vert |\!|\!| \Delta_q h |\!|\!| \Vert_{L^2_x} dt \right)^{1/2}.
\end{align}
Starting from \eqref{ineq: AfterTrilinear1}, we shall estimates $\sum_{q} 2^{qs}\left( \int^T_0 \vert ( \Delta_q \Gamma^1 (f,g), \Delta_qh)_{x,v} \vert dt \right)^{1/2}$ in two different ways. Since appearing terms will be lengthy, we set
\begin{equation*}
X_j = \left( \int_0^T \Vert |\!|\!| \Delta_j g |\!|\!| \Vert_{L^2_x}^2 dt \right)^{1/2},\quad Y_q=\left(\int_0^T \Vert |\!|\!| \Delta_q h |\!|\!| \Vert_{L^2_x}^2 dt \right)^{1/2}.
\end{equation*}

 Using the Cauchy-Schwarz inequality, embedding $ B^{3/2}_x \hookrightarrow L^\infty_x$ and Lemma \ref{Besov and Chemin-Lerner}, we have
\begin{align}\label{tri-11}\notag
&\sum_{q \ge -1} 2^{qs}\left( \int^T_0 \vert ( \Delta_q \Gamma^1 (f,g), \Delta_qh)_{x,v} \vert dt \right)^{1/2} \\ \notag
&{\lesssim} \sum_{q \ge -1}2^{qs} \Vert f \Vert_{L^\infty_T L^2_v L^\infty_x}^{1/2} \left( \sum_{\vert j-q \vert \le 4} (X_j)^2 \right)^{1/4} (Y_q)^{1/2} \\
\notag
&\lesssim \Vert f \Vert_{\tilde{L}^\infty_T \tilde{L}^2_v (B^{3/2}_x)}^{1/2} \left( \sum_{q \ge -1} 2^{qs} \left( \sum_{\vert j-q \vert \le 4} X_j \right)^{1/2} \right)^{1/2} \notag \left( \sum_{q \ge -1} 2^{qs} Y_q \right)^{1/2}   \\ 
&\le \Vert f \Vert_{\tilde{L}^\infty_T \tilde{L}^2_v (B^{3/2}_x)}^{1/2} \Vert h \Vert_{\mathcal{T}^s_{T,2,2}}^{1/2} \left( \sum_{q \ge -1} 2^{qs}  \sum_{\vert j-q \vert \le 4} 2^{-js}c_j \right)^{1/2} \Vert g \Vert_{\mathcal{T}^s_{T,2,2}}^{1/2}.
\end{align}
Here, we defined a $\ell^1$-sequence $\{c_j\}$ by $c_j := 2^{js} X_j/\Vert g \Vert_{\mathcal{T}^s_{T,2,2}}$.
We will use abuse of notation $c_j$ in the sequal to express similar sequences, defined by other appropriate $\mathcal{T}^s_{Tpr}$-norms in each inequality. Note that all of them are in $\ell^1$. Now, the double summation is finite because by Fubini's theorem and Young's inequality, we have
\begin{align*}
\sum_{q \ge -1} 2^{qs}  \sum_{\vert j-q \vert \le 4} 2^{-js}c_j 
&=\sum_{q\ge -1} [(\mathbf{1}_{\vert j \vert\le 4}2^{js})*c_j](q)\\
&\le  \sum_j \mathbf{1}_{\vert j\vert\le 4}2^{js}\cdot\sum_j c_j <\infty.
\end{align*}
Therefore, we have
\begin{align*}
\sum_{q \ge -1} 2^{qs}\left( \int^T_0 \vert ( \Delta_q \Gamma^1 (f,g), \Delta_qh)_{x,v} \vert dt \right)^{1/2}\lesssim \Vert f \Vert_{\tilde{L}^\infty_T \tilde{L}^2_v (B^{3/2}_x)}^{1/2} \Vert g \Vert_{\mathcal{T}^s_{T,2,2}}^{1/2}\Vert h \Vert_{\mathcal{T}^s_{T,2,2}}^{1/2}.
\end{align*}

The $\Gamma^1$ part of the second  estimate \eqref{tri2} is as follows. In this case, we apply embedding $ \dot{B}^{3/2}_x \hookrightarrow L^\infty_x$ to \eqref{ineq: AfterTrilinear1}.
\begin{align*}
&\sum_{q \ge -1} 2^{qs}\left( \int^T_0 \vert ( \Delta_q \Gamma^1 (f,g), \Delta_qh)_{x,v} \vert dt \right)^{1/2} \\
&\lesssim  \Vert f \Vert_{L^2_TL^2_v (\dot{B}^{3/2}_x)}^{1/2}\sum_{q \ge -1} 2^{qs} \left( \sum_{\vert j-q \vert \le 4}  \Vert |\!|\!| \Delta_j g |\!|\!| \Vert_{L^\infty_TL^2_x} \right)^{1/2}  \left(\int_0^T \Vert |\!|\!| \Delta_q h |\!|\!| \Vert_{L^2_x}^2 dt \right)^{1/4} \\
& \le  \Vert f \Vert_{L^2_TL^2_v (\dot{B}^{3/2}_x)}^{1/2} \Vert g \Vert_{\mathcal{T}^s_{T,\infty,2}}^{1/2} \Vert h \Vert_{\mathcal{T}^s_{T,2,2}}^{1/2}\left( \sum_{q \ge -1} 2^{qs}  \sum_{\vert j-q \vert \le 4} 2^{-js}c_j \right)^{1/2} \\
& \lesssim \Vert f \Vert_{L^2_TL^2_v (\dot{B}^{3/2}_x)}^{1/2} \Vert g \Vert_{\mathcal{T}^s_{T,\infty,2}}^{1/2} \Vert h \Vert_{\mathcal{T}^s_{T,2,2}}^{1/2}.
\end{align*}
 
Next, we estimate $\sum_q 2^{qs}\left( \int^T_0 \vert ( \Delta_q \Gamma^2 (f,g), \Delta_qh)_{x,v} \vert dt \right)^{1/2}$. Since $\Gamma^1$ and $\Gamma^2$ have symmetry, we can similarly calculate it and obtain 
\begin{align*}
&\sum_{q \ge -1} 2^{qs}\left( \int^T_0 \vert ( \Delta_q \Gamma^2 (f,g), \Delta_qh)_{x,v} \vert dt \right)^{1/2} \\ 
& \lesssim \sum_{q \ge -1} 2^{qs} \left( \sum_{\vert j-q \vert \le 4}\int^T_0 \Vert |\!|\!| S_{j-1} g |\!|\!| \Vert_{L^\infty_x} \Vert \Delta_j f \Vert_{L^2_{x,v}} \Vert |\!|\!| \Delta_q h |\!|\!| \Vert_{L^2_x} dt \right)^{1/2} \\
& \le \sum_{q \ge -1} 2^{qs} \left( \sum_{\vert j-q \vert \le 4} \Vert \Delta_j f \Vert_{L^\infty_T L^2_{x,v}} \Vert |\!|\!| S_{j-1} g |\!|\!| \Vert_{L^2_TL^\infty_x} \Vert |\!|\!| \Delta_q h |\!|\!| \Vert_{L^2_TL^2_x} \right)^{1/2} \\
& \lesssim \sum_{q \ge -1} 2^{qs}\left( \sum_{\vert j-q \vert \le 4} \Vert \Delta_j f \Vert_{L^\infty_T L^2_{x,v}} \Vert g\Vert_{\dot{\mathcal{T}}^{3/2}_{T,2,2}}\right)^{1/2} \Vert |\!|\!| \Delta_q h |\!|\!| \Vert_{L^2_TL^2_x}^{1/2}  \\
& \le \Vert g\Vert_{\dot{\mathcal{T}}^{3/2}_{T,2,2}}^{1/2} \left(\sum_{q \ge -1} 2^{qs}\sum_{\vert j-q \vert \le 4} \Vert \Delta_j f \Vert_{L^\infty_T L^2_{x,v}} \right)^{1/2} \left( \sum_{q \ge -1} 2^{qs} \Vert |\!|\!| \Delta_q h |\!|\!| \Vert_{L^2_TL^2_x} \right)^{1/2} \\
& \lesssim \Vert f \Vert_{\tilde{L}^\infty_T \tilde{L}^2_v(B^s_x)}^{1/2} \Vert g\Vert_{\dot{\mathcal{T}}^{3/2}_{T,2,2}}^{1/2} \Vert h \Vert_{\mathcal{T}^s_{T,2,2}}^{1/2},
\end{align*}
where we used Lemma \ref{T-estimate}.
Since $\Vert g\Vert_{\dot{\mathcal{T}}^{3/2}_{T,2,2}} \lesssim \Vert g\Vert_{\mathcal{T}^{3/2}_{T,2,2}}$ and $B^{s_1}_{pr} \hookrightarrow B^{s_2}_{pr}$ when $s_2 \le s_1$, we can deduce that 
\begin{align*}
\sum_q 2^{qs}\left( \int^T_0 \vert ( \Delta_q \Gamma^2 (f,g), \Delta_qh)_{x,v} \vert dt \right)^{1/2} \lesssim \Vert f \Vert_{\tilde{L}^\infty_T \tilde{L}^2_v(B^{3/2}_x)}^{1/2} \Vert g \Vert_{\mathcal{T}^{3/2}_{T,2,2}}^{1/2} \Vert h \Vert_{\mathcal{T}^s_{T,2,2}}^{1/2}.
\end{align*}
Therefore the estimates of the second part $\Gamma^2$ for
\eqref{tri2} and  \eqref{tri1} are completed.

Finally, we calculate the third part $\Gamma^3$. Recalling the size of supports of $\mathcal{F}[\Delta_{j'}f]$ and $\mathcal{F}[\Delta_{j} g]$, first we have
\begin{align*}
\Delta_q\left( \sum_j \sum_{\vert j-j'\vert \le 1} (\Delta_{j'} f, \Delta_j g) \right)= \Delta_q \left(\sum_{\max\{j,j'\} \ge q-2 } \sum_{\vert j-j' \vert \le 1} (\Delta_{j'} f, \Delta_j g)\right).
\end{align*}
We notice that the summation concerning $j$ is not finite at this time. However, double summation appearing later can be shown finite. 
Similar to $\Gamma^1$, an estimate of $\Gamma^3$ with respect to $x$ and $v$ is
\begin{align}\label{ineq: AfterTrilinear3}\notag
&\sum_{q \ge -1} 2^{qs}\left( \int^T_0 \vert ( \Delta_q \Gamma^3 (f,g), \Delta_qh)_{x,v} \vert dt \right)^{1/2} \\
& \lesssim \sum_{q \ge -1}2^{qs} \left( \sum_{j \ge q-3} \int^T_0 \Vert f \Vert_{L^2_vL^\infty_x} \Vert |\!|\!| \Delta_j g |\!|\!| \Vert_{L^2_x} \Vert |\!|\!| \Delta_q h |\!|\!| \Vert_{L^2_x} dt \right)^{1/2}.
\end{align}
Using the Cauchy-Schwarz inequality with  respect to $t$-integration, we have 
\begin{align*}
&\sum_{q \ge -1} 2^{qs}\left( \int^T_0 \vert ( \Delta_q \Gamma^3 (f,g), \Delta_qh)_{x,v} \vert dt \right)^{1/2} \\
& \lesssim \sum_{q \ge -1}2^{qs} \Vert f \Vert_{L^\infty_T L^2_v (B^{3/2}_x)}^{1/2} \left( \sum_{j \ge q-3} \Vert |\!|\!| \Delta_j g |\!|\!| \Vert_{L^2_{T,x}} \right)^{1/2} \left(\Vert |\!|\!| \Delta_q h |\!|\!| \Vert_{L^2_{T,x}}\right)^{1/2} \\
& \le \Vert f \Vert_{\tilde{L}^\infty_T \tilde{L}^2_v (B^{3/2}_x)}^{1/2} \Big( \sum_{q \ge -1}\sum_{j \ge q-3} 2^{qs}\Vert |\!|\!| \Delta_j g |\!|\!| \Vert_{L^2_{T,x}}\Big)^{1/2} \Big(\sum_{q \ge -1}2^{qs} \Vert |\!|\!| \Delta_q h |\!|\!| \Vert_{L^2_{T,x}} \Big)^{1/2} \\
& \lesssim \Vert f \Vert_{\tilde{L}^\infty_T \tilde{L}^2_v (B^{3/2}_x)}^{1/2} \Vert g \Vert_{\mathcal{T}^s_{T,2,2}}^{1/2} \Vert h \Vert_{\mathcal{T}^s_{T,2,2}}^{1/2}\left( \sum_{q \ge -1}2^{qs}\sum_{j \ge q-3} 2^{-js}c_j \right)^{1/2} .
\end{align*} 
The last factor of the previous line is finite because
\begin{align*}
\sum_{q \ge -1}2^{qs}\sum_{j \ge q-3} 2^{-js}c_j= \sum_{j \ge -4} c_j \sum_{j-q \ge -3}2^{-(j-q)s} <\infty.
\end{align*}
Here we applied Fubini's theorem and Young's inequality again. This is an estimate corresponding to \eqref{tri1}. Moreover, we calculate that 
\begin{align*}
&\sum_{q \ge -1} 2^{qs}\left( \int^T_0 \vert ( \Delta_q \Gamma^3 (f,g), \Delta_qh)_{x,v} \vert dt \right)^{1/2} \\
& \lesssim \sum_{q \ge -1} 2^{qs} \Vert f \Vert_{L^2_TL^2_v L^\infty_x}^{1/2} \left(\sum_{j \ge q-3} \Vert |\!|\!| \Delta_j g |\!|\!| \Vert_{L^\infty_T L^2_x} \right)^{1/2} \Vert |\!|\!| \Delta_q h |\!|\!| \Vert_{L^2_{T,x}}^{1/2}  \\
& \lesssim  \Vert f \Vert_{L^2_T L^2_v(\dot{B}^{3/2}_x)}^{1/2} \left( \sum_{q \ge -1} \sum_{j \ge q-3}2^{qs}\Vert |\!|\!| \Delta_j g |\!|\!| \Vert_{L^\infty_T L^2_x} \right)^{1/2} \Vert h \Vert_{{\mathcal{T}}^s_{T,2,2}}^{1/2}\\
& \lesssim \Vert f \Vert_{L^2_T L^2_v(\dot{B}^{3/2}_x)}^{1/2}\Vert g \Vert_{\mathcal{T}^s_{T,\infty,2}}^{1/2} \Vert h \Vert_{\mathcal{T}^s_{T,2,2}}^{1/2}.
\end{align*}
Combining the above estimates properly, we obtain the second estimate \eqref{tri2}.
For the third estimate \eqref{tri3}, we apply Corollary \ref{another-type} instead of Corollary  \ref{AMUXYtrilinear}. Then the completely same argument of 
\eqref{tri2} yields \eqref{tri3}. Notice that all we have to do is to replace $f$ with $\mu^{1/10}f$ in the first term on the right hand side of 
\eqref{tri3}, and to replace $\mathcal{T}^s_{T,\infty,2}$ and $\dot{\mathcal{T}}^{3/2}_{T,2,2}$  with $\tilde{L}^\infty_T \tilde{L}^2_v(B^{s}_x)$ and $L^2_T L^2_v(\dot{B}^{3/2}_x)$ respectively in the second term.
\end{proof}

\section{Estimate on nonlinear term and a priori estimate}\label{S4}
\setcounter{equation}{0}
Inserting certain functions into the inequalities of Lemma \ref{trilinear}, we first estimate some nonlinear terms by the energy and the dissipation term. After that, we will derive an apriori estimate. In this section, we decompose $f$ into $\mu + \mu^{1/2}g$, and all Lemmas are statements on this $g$.
 
\begin{lem}\label{4-1}
Assume $0<s\le\frac{3}{2}$. Then we have
\begin{align*}
\sum_{q \ge -1} 2^{qs}\left( \int^T_0 \vert ( \Delta_q \Gamma (g,g), \Delta_q(\mathbf{I}-\mathbf{P})g) \vert dt \right)^{1/2} \lesssim \sqrt{\mathcal{E}_T(g)}\mathcal{D}_T(g).
\end{align*}
\end{lem}
\begin{proof}
We devide $\Gamma(g,g)$ into $\Gamma(g,\mathbf{P}g)$ and $\Gamma(g, (\mathbf{I}-\mathbf{P})g)$, and estimate the respective terms. Using \eqref{tri1} of Lemma \ref{trilinear}, we obtain 
\begin{align*}
\sum_{q \ge -1} &2^{qs}\left( \int^T_0 \vert ( \Delta_q \Gamma (g,(\mathbf{I}-\mathbf{P})g), \Delta_q(\mathbf{I}-\mathbf{P})g) \vert dt \right)^{1/2} \\
& \lesssim \Vert g \Vert_{\tilde{L}^\infty_T \tilde{L}^2_v (B^{3/2}_x)}^{1/2} \Vert (\mathbf{I}-\mathbf{P})g \Vert_{\mathcal{T}^s_{T,2,2}} \lesssim \sqrt{\mathcal{E}_T(g)}\mathcal{D}_T(g).
\end{align*}
When $\gamma+ \nu >0$, \eqref{tri2} of Lemma \ref{trilinear} yields 
\begin{align*}
\sum_{q \ge -1} &2^{qs}\left( \int^T_0 \vert ( \Delta_q \Gamma (g,\mathbf{P}g), \Delta_q(\mathbf{I}-\mathbf{P})g) \vert dt \right)^{1/2} \\
& \lesssim \Vert(\mathbf{I}-\mathbf{P})g\Vert_{\mathcal{T}^s_{T,2,2}}^{1/2}  \left( \Vert g \Vert_{L^2_T L^2_v(\dot{B}^{3/2}_x)}^{1/2} \Vert \mathbf{P}g \Vert_{\mathcal{T}^s_{T,\infty,2}}^{1/2} + \Vert g \Vert_{\tilde{L}^\infty_T \tilde{L}^2_v(B^s_x)}^{1/2} \Vert \mathbf{P}g \Vert_{\dot{\mathcal{T}}^{3/2}_{T,2,2}}^{1/2} \right) \\
& \lesssim \sqrt{\mathcal{E}_T(g)}\mathcal{D}_T(g),
\end{align*}
where we used the following inequalities:
\begin{gather*}
\Vert g \Vert_{L^2_T L^2_v(\dot{B}^{3/2}_x)} \lesssim \Vert g \Vert_{\dot{\mathcal{T}}^{3/2}_{T,2,2}} \lesssim \Vert \mathbf{P}g \Vert_{\dot{\mathcal{T}}^{3/2}_{T,2,2}}+ \Vert (\mathbf{I}-\mathbf{P})g \Vert_{\mathcal{T}^{3/2}_{T,2,2}} \sim \mathcal{D}_T(g), \\
\Vert \mathbf{P}g \Vert_{\mathcal{T}^s_{T,\infty,2}} \lesssim \Vert g \Vert_{\tilde{L}^\infty_T \tilde{L}^2_v (B^s_x)}.
\end{gather*}
The first inequality is deduced by Lemmas \ref{Besov embedding}-\ref{Besov and Chemin-Lerner}, and the fact that $\Vert \cdot \Vert_{L^2_v} \le \Vert \cdot \Vert_{L^2_{v, (\gamma+\nu)/2}} \lesssim |\!|\!| \cdot |\!|\!|$ when $\gamma + \nu > 0$. We recall  $|\!|\!| \mathbf{P}f |\!|\!| 
\lesssim \Vert f \Vert_{L^2_v}$ and Lemma \ref{Besov and Chemin-Lerner} for the second one.

When $\gamma+ \nu \le 0$, \eqref{tri3} of Lemma  \ref{trilinear}  similarly yields
\begin{align*}
\sum_{q \ge -1} &2^{qs}\left( \int^T_0 \vert ( \Delta_q \Gamma (g,\mathbf{P}g), \Delta_q(\mathbf{I}-\mathbf{P})g) \vert dt \right)^{1/2} \\
&\lesssim  \Vert \mu^{1/10}g \Vert_{L^2_T L^2_v(\dot{B}^{3/2}_x)}^{1/2} \Vert \mathbf{P}g \Vert_{\mathcal{T}^s_{T,\infty,2}}^{1/2}\Vert(\mathbf{I}-\mathbf{P})g\Vert_{\mathcal{T}^s_{T,2,2}}^{1/2} \\
&+ \Vert \mu^{1/10}g \Vert_{\tilde{L}^\infty_T \tilde{L}^2_v(B^s_x)}^{1/2} \Vert \mathbf{P}g \Vert_{\dot{\mathcal{T}}^{3/2}_{T,2,2}}^{1/2}\Vert(\mathbf{I}-\mathbf{P})g\Vert_{\mathcal{T}^s_{T,2,2}}^{1/2}   \\
&+ \Vert g \Vert_{L^2_T L^2_{v, (\nu+\gamma)/2}(\dot{B}^{3/2}_x)}^{1/2} \Vert \mathbf{P}g \Vert_{\tilde{L}^\infty_T \tilde{L}^2_v(B^{s}_x)}^{1/2} \Vert (\mathbf{I}-\mathbf{P})g \Vert_{\mathcal{T}^s_{T,2,2}}^{1/2}\\
&+ \Vert g \Vert_{\tilde{L}^\infty_T \tilde{L}^2_{v, (\nu+\gamma)/2}(B^s_x)}^{1/2} \Vert \mathbf{P}g \Vert_{L^2_T L^2_v(\dot{B}^{3/2}_x)}^{1/2}\Vert (\mathbf{I}-\mathbf{P})g \Vert_{\mathcal{T}^s_{T,2,2}}^{1/2} \\
& \lesssim \sqrt{\mathcal{E}_T(g)}\mathcal{D}_T(g).
\end{align*}
Owing to the fact $|\!|\!| \mathbf{P}g |\!|\!| \lesssim \Vert g \Vert_{L^2_v}$, $\Vert \mathbf{P}g \Vert_{\tilde{L}^\infty_T \tilde{L}^2_v(B^{s}_x)}$ and $\Vert \mathbf{P}g \Vert_{\mathcal{T}^s_{T,\infty,2}}$ are similarly estimated. So are $\Vert \mathbf{P}g \Vert_{\dot{\mathcal{T}}^{3/2}_{T,2,2}}$ and $\Vert \mathbf{P}g \Vert_{L^2_T L^2_v(\dot{B}^{3/2}_x)}$.

\end{proof}

\begin{lem}\label{4-2}
Let $\phi \in \mathcal{S}(\mathbb{R}^3_v)$ and $0<s\le \frac{3}{2}$. Then we have
\begin{align*}
\sum_{q \ge -1} 2^{qs}\left( \int^T_0 \Vert\Delta_q (\Gamma (g,g), \phi) \Vert^2_{L^2_x}dt \right)^{1/2} \lesssim \mathcal{E}_T(g)\mathcal{D}_T(g).
\end{align*}
\end{lem}

\begin{proof}
First, we consider the case $\gamma+\nu >0$. Together with the decomposition $\Gamma(g,g)=\Gamma^1(g,g)+\Gamma^2(g,g)+\Gamma^3(g,g)$, 
the generalized Minkowski inequality
gives
\begin{align*}
&\left( \int^T_0 \Vert\Delta_q (\Gamma (g,g), \phi) \Vert^2_{L^2_x}dt \right)^{1/2} 
\le \sum_{i=1}^3\left( \int^T_0 \Vert\Delta_q (\Gamma^i (g,g), \phi) \Vert^2_{L^2_x}dt  \right)^{1/2} \\
&\le \sum_{\vert j-q \vert \le 4} \left( \int^T_0 \Vert 
\Delta_q (\Gamma (S_{j-1}g,\Delta_jg),  \phi) \Vert^2_{L^2_x}
dt \right)^{1/2} \\
&+\sum_{\vert j-q \vert \le 4} \left( \int^T_0  \Vert 
\Delta_q (\Gamma (\Delta_j g, S_{j-1}g),  \phi) \Vert^2_{L^2_x}
dt \right)^{1/2} \\
&+\sum_{\max\{j,j'\} \ge q-2}  \sum_{\vert j-j' \vert \le 1} \left( \int^T_0
 \Vert \Delta_q (\Gamma ( \Delta_{j'} g,\Delta_jg),  \phi) \Vert^2_{L^2_x} dt \right)^{1/2} \\
&=: I^1_q+I^2_q+I^3_q.
\end{align*}
By using 
Corollary \ref{AMUXYtrilinear} and
Lemma \ref{bounded operators},  each term is calculated as follows:
We use the macro-micro decomposition $g= \mathbf{P}g+ (\mathbf{I}-\mathbf{P})g$ to estimate the first and third term properly.
Then we have
\begin{align*}
\sum_{q \ge -1} 2^{qs} I^1_q &\le  
\sum_{q \ge -1} 2^{qs}\sum_{\vert j-q \vert \le 4} \left( \int^T_0\int \Vert S_{j-1}g \Vert_{L^2_v}^2 |\!|\!| \Delta_j \mathbf{P}g |\!|\!|^2 dx dt \right)^{1/2}\\
&+\sum_{q \ge -1} 2^{qs}  \sum_{\vert j-q \vert \le 4} 
\left( \int^T_0\int \Vert S_{j-1}g \Vert_{L^2_v}^2 |\!|\!| \Delta_j (\mathbf{I}-\mathbf{P})g |\!|\!|^2 dx dt \right)^{1/2} 
\\
&\le \sum_{q \ge -1} 2^{qs}\sum_{\vert j-q \vert \le 4} 
\left(\int_0^T \Vert g 
\Vert_{L^2_vL^\infty_x} ^2 dt\right)^{1/2}
\Vert\Delta_j \mathbf{P}g \Vert_{L_T^\infty L^2_v L^2_x}\\ 
&+\sum_{q \ge -1} 2^{qs}  \sum_{\vert j-q \vert \le 4} \Vert g \Vert_{L^\infty_T L^2_{v}L^\infty_x} 
\left( \int^T_0\Vert |\!|\!| \Delta_j (\mathbf{I}-\mathbf{P})g |\!|\!| \Vert_{L^2_x}^2 dt \right)^{1/2} \\
&\lesssim \Vert g \Vert_{L^2_T L^2_v(\dot{B}^{3/2}_x)} \Vert \mathbf{P}g \Vert_{\tilde L^\infty_T \tilde L^2_v(B_x^{s})}
+\Vert g \Vert_{\tilde{L}^\infty_T \tilde{L}^2_v (B^{3/2}_x)} \Vert (\mathbf{I}-\mathbf{P})g \Vert_{\mathcal{T}^s_{T,2,2}}  \\
& \lesssim \mathcal{E}_T(g)\mathcal{D}_T(g),
\end{align*}

\begin{align*}
\sum_{q \ge -1} 2^{qs} I^2_q &\le \sum_{q \ge -1} 2^{qs}  \sum_{\vert j-q \vert \le 4} \Vert  \Delta_j g \Vert_{L^\infty_T L^2_{x,v}} \left( \int^T_0 \Vert |\!|\!| S_{j-1}g |\!|\!| \Vert_{L^\infty_x}^2 dt \right)^{1/2}\\
&\lesssim \Vert g \Vert_{\tilde L^\infty_T \tilde L^2_v(B^s_x)}\Vert g \Vert_{\dot{\mathcal{T}}^{3/2}_{T,2,2}}\lesssim \mathcal{E}_T(g)\mathcal{D}_T(g),
\end{align*}

\begin{align*}
\sum_{q \ge -1} 2^{qs} I^3_q &\le \sum_{q \ge -1}  \sum_{j \ge q-3} 2^{qs} \Vert \Delta_j g \Vert_{L^\infty_T L^2_{x,v}} \left( \int^T_0 \Vert |\!|\!|  \mathbf{P}g |\!|\!| \Vert_{L^\infty_x}^2 dt \right)^{1/2}\\
&
+
\sum_{q \ge -1} \sum_{j' \ge q-3} 2^{qs}\Vert g \Vert_{L^\infty_T L^2_v L^\infty_x} \left( \int^T_0 \Vert |\!|\!| \Delta_{j'} (\mathbf{I}-\mathbf{P})g |\!|\!| \Vert_{L^2_x}^2 dt \right)^{1/2} 
\\
& \lesssim \Vert g \Vert_{\tilde L^\infty_T \tilde L^2_v(B^s_x)} \Vert \mathbf{P}g \Vert_{\dot{\mathcal{T}}^{3/2}_{T,2,2}} +
\Vert g \Vert_{{L}^\infty_T {L}^2_v (B^{3/2}_x)} \Vert (\mathbf{I}-\mathbf{P})g \Vert_{\mathcal{T}^s_{T,2,2}} \\
& \lesssim \mathcal{E}_T(g)\mathcal{D}_T(g).
\end{align*}

When $\gamma+\nu \le 0$, apply Corollary \ref{another-type}
instead of Corollary \ref{AMUXYtrilinear} for the estimate of $I_q^1$.  Then we have 
\begin{align*}
\sum_{q \ge -1} 2^{qs} I^1_q &\le  
\sum_{q \ge -1} 2^{qs}\sum_{\vert j-q \vert \le 4} \left( \int^T_0\int \Vert 
\la v \ra^{(\gamma+ \nu)/2} S_{j-1}g \Vert_{L^2_v}^2 |\!|\!| \Delta_j \mathbf{P}g |\!|\!|^2 dx dt \right)^{1/2}\\
&+\sum_{q \ge -1} 2^{qs}  \sum_{\vert j-q \vert \le 4} 
\left( \int^T_0\int \Vert S_{j-1}g \Vert_{L^2_v}^2 |\!|\!| \Delta_j (\mathbf{I}-\mathbf{P})g |\!|\!|^2 dx dt \right)^{1/2} 
\\
&\lesssim \Vert \la v \ra^{(\gamma+ \nu)/2}g \Vert_{L^2_T L^2_v(\dot{B}^{3/2}_x)} \Vert \mathbf{P}g \Vert_{\tilde L^\infty_T \tilde L^2_v(B_x^{s})}\\
&+\Vert  g \Vert_{\tilde{L}^\infty_T \tilde{L}^2_v (B^{3/2}_x)} \Vert (\mathbf{I}-\mathbf{P})g \Vert_{\mathcal{T}^s_{T,2,2}}  \\
& \lesssim \mathcal{E}_T(g)\mathcal{D}_T(g).
\end{align*}
Other parts follow in the same manner.
\end{proof}

As for the upper bound of the linear term $\cL$ we have
\begin{lem}\label{4-3}
Let $\phi \in \mathcal{S}(\mathbb{R}^3_v) $ and $s>0$. Then we have
\begin{align*}
\sum_{q \ge -1} 2^{qs}\left( \int^T_0 \Vert\Delta_q (\cL(\mathbf{I}-\mathbf{P})g, \phi) \Vert^2_{L^2_x}dt \right)^{1/2} \lesssim \Vert (\mathbf{I}-\mathbf{P})g \Vert_{\mathcal{T}^s_{T,2,2}}.
\end{align*}
\end{lem}
\begin{proof} Since it follows from Corollaries  \ref{AMUXYtrilinear} and \ref{another-type}
that 
\begin{align*}
&\left |\Delta_q (\cL(\mathbf{I}-\mathbf{P})g, \phi)_{L^2_v}\right |\\
& \le \left
| (\Gamma(\mu^{1/2}, \Delta_q(\mathbf{I}-\mathbf{P})g), \phi)_{L^2_v}\right |
+\left | (\Gamma( \Delta_q(\mathbf{I}-\mathbf{P})g, \mu^{1/2}), \phi)_{L^2_v}\right |\\
& \lesssim |\!|\!|\Delta_q(\mathbf{I}-\mathbf{P})g|\!|\!| + \| \Delta_q(\mathbf{I}-\mathbf{P})g\|_{L^2_{(\gamma
+\nu)/2}(\RR^3_v)},
\end{align*}
we obtain the desired estimate. 
\end{proof}

We next estimate derivative of the macroscopic part.
\begin{lem}\label{macro-micro}
It holds that
\begin{align}\label{macro-apriori}\nonumber
&\Vert \nabla_x (a,b,c) \Vert_{\tilde{L}^2_T(B^{1/2}_x)} \\
&\quad \quad  \lesssim \Vert g_0 \Vert_{\tilde{L}^2_v(B^{3/2}_x)} +\mathcal{E}_T(g) +\Vert (\mathbf{I}-\mathbf{P})g \Vert_{\mathcal{T}^{3/2}_{T,2,2}}+\mathcal{E}_T(g)\mathcal{D}_T(g).
\end{align}
\end{lem}
\begin{proof}
We start from \eqref{eq: Boltzmann}. Multiplying the equation by $1$, $v$, and $\vert v \vert^2$ and taking $v$-integration, we have the following local macroscopic balance laws:
\begin{align*}
&\partial_t \int_{\mathbb{R}^3_v} fdv + \nabla_x \cdot \int_{\mathbb{R}^3_v} vfdv =0, \\
&\partial_t \int_{\mathbb{R}^3_v} vfdv + \nabla_x\cdot \int_{\mathbb{R}^3_v} v\otimes vfdv =0,\\
&\partial_t \int_{\mathbb{R}^3_v} \vert v\vert^2fdv + \nabla_x \cdot \int_{\mathbb{R}^3_v} \vert v\vert^2 vfdv =0.
\end{align*}
We decompose $f$ into $f = \mu + \mu^{1/2}g$ and further decompose $g$ into $g=\mathbf{P}g + (\mathbf{I-P})g =g_1+g_2$. In order to express the above balance laws with the macroscopic functions ($a$, $b$, $c$), we need to calculate some moments of Maxwellian $\mu$.
\begin{align*}
&(1, \mu)=1, (\vert v_i \vert^2, \mu)=1, (\vert v \vert^2, \mu)=3,\\
&(\vert v_i \vert^2 \vert v_j \vert^2, \mu)=1\ (i \neq j), (\vert v_i \vert^4, \mu ) =3, \\
&(\vert v \vert^2 \vert v_i \vert^2, \mu) = 5, (\vert v \vert^4, \mu) = 15, ( \vert v\vert^4 \vert v_i \vert^2 \mu )=35.
\end{align*}
Keeping these values in mind, we compute that 
\begin{align*}
&\int_{\mathbb{R}^3_v} fdv = 1 + (a+ 3c),\\
&\int_{\mathbb{R}^3_v} vfdv = b,\\
&\int_{\mathbb{R}^3_v} \vert v\vert^2fdv = 3 + 3(a+5c),\\
&\int_{\mathbb{R}^3_v} v_i v_j fdv = \delta_{ij} + (a+5c)\delta_{ij} + (v_i v_j \mu^{1/2}, g_2),\\
&\int_{\mathbb{R}^3_v} \vert v\vert^2 vfdv = 5b + (\vert v \vert^2 v \mu^{1/2}, g_2).
\end{align*}
Here, $\delta_{ij}$ is Kronecker's delta. Inserting these identities into the balance laws, we have
\begin{align*}
&\partial_t (a+3c) + \nabla_x \cdot b=0,\\
&\partial_t b + \nabla_x (a+5c) + \nabla_x \cdot (v\otimes v \mu^{1/2}, g_2) =0,\\
&3\partial_t(a+5c) +5\nabla_x \cdot b + \nabla_x \cdot (\vert v\vert^2 v \mu^{1/2} g_2)=0. 
\end{align*}
This is equivalent to the system
\begin{align*}
&\partial_t a - \frac{1}{2} \nabla_x \cdot(\vert v \vert^2 v\mu^{1/2}, g_2)=0,\\
&\partial_t b + \nabla_x (a+5c) + \nabla_x \cdot (v\otimes v \mu^{1/2}, g_2) =0,\\
&\partial_t c + \frac{1}{3}\nabla_x \cdot b +\frac{1}{6} \nabla_x \cdot (\vert v \vert^2 v \mu^{1/2}, g_2) =0.
\end{align*}
Next, we rewrite \eqref{eq: perturbation} as 
\begin{align*}
\partial_t g_1 +v\cdot \nabla_x g_1 =-\partial_t g_2 + R_1+ R_2,
\end{align*}
where $R_1= -v \cdot \nabla_x g_2$ and $R_2 = -\mathcal{L}g_2 + \Gamma(g, g)$. We rewrite this equation once again so that we express the equation in terms of ($a$, $b$, $c$).
\begin{align}\label{eq: macro perturbation}\notag
&\partial_t a \mu^{1/2} + (\partial_t b + \nabla_x a)\cdot v \mu^{1/2} + \sum_i  (\partial_t c+\partial_i b_i) \vert v_i \vert^2 \mu^{1/2} \\
&\quad+\sum_{i< j} (\partial_j b_i +\partial_i b_j) v_i v_j \mu^{1/2} + \nabla_x c \cdot \vert v \vert^2 v \mu^{1/2} =-\partial_t g_2 + R_1+ R_2. 
\end{align}
Here, $\partial_i = \partial_{x_i}$. We define the high-order moment functions $A(g)=(A_{ij}(g))_{3\times 3}$ and $B(g) = (B_i(g))_{1\times 3}$ by
\begin{align*}
A_{ij}(g) = ( (v_iv_j -\delta_{ij})\mu^{1/2}, g), B_i(g) = \frac{1}{10}( (\vert v\vert^2-5)v_i \mu^{1/2}, g).
\end{align*}
Notice that these functions operate only on $v$ and 
$$\vert A_{ij}(g) \vert, \vert B_i(g) \vert \le C\Vert g \Vert_{L^2_{v,(\gamma+\nu)/2}}$$
since we take inner products of $g$ and rapidly decreasing functions.
Applying $A_{ij}$ and $B_i$ to both sides of \eqref{eq: macro perturbation}, we have
\begin{align*}
&\partial_t (A_{ij}(g_2)+2c\delta_{ij}) + \partial_j b_i + \partial_i b_j = A_{ij} (R_1+R_2),\\
&\partial_t B_i(g_2) + \partial_i c = B_i(R_1+R_2).
\end{align*}
Thus, we obtained the following system:
\begin{align}\label{eq: fluid-type}
\begin{cases}
\partial_t a - \frac{1}{2} \nabla_x \cdot(\vert v \vert^2 v\mu^{1/2}, g_2)=0,\\
\partial_t b + \nabla_x (a+5c) + \nabla_x \cdot (v\otimes v \mu^{1/2}, g_2) =0,\\
\partial_t c + \frac{1}{3}\nabla_x \cdot b +\frac{1}{6} \nabla_x \cdot (\vert v \vert^2 v \mu^{1/2}, g_2) =0,\\
\partial_t (A_{ij}(g_2)+2c\delta_{ij}) + \partial_j b_i + \partial_i b_j = A_{ij} (R_1+R_2),\\
\partial_t B_i(g_2) + \partial_i c = B_i(R_1+R_2).
\end{cases}
\end{align}
We derive the desired estimate from this system. For later use, we define the energy functional 
\begin{align*}
&E_q(g) = E^1_q(g) + \delta_2 E^2_q(g) + \delta_3 E^3_q(g),\\
&E^1_q(g)=\sum_{i}(B_i (\Delta_q g_2), \partial_i \Delta_q c),\\
&E^2_q(g)=\sum_{i,j} (A_{ij}(\Delta_q g_2 )+2\Delta_q c\delta_{ij}, \partial_j \Delta_qb_i + \partial_i \Delta_qb_j),\\
&E^3_q(g)=(\Delta_q b_,\nabla_x \Delta_q a).
\end{align*}
$\delta_2$ and $\delta_3$ are two small positive numbers, which will be chosen later.
First, we apply $\Delta_q$ with $q \ge -1$ to \eqref{eq: fluid-type}. We have
\begin{align}\label{eq: fluid-type2} 
\begin{cases}
\partial_t \Delta_qa - \frac{1}{2} \nabla_x \cdot(\vert v \vert^2 v\mu^{1/2}, \Delta_qg_2)=0,\\
\partial_t \Delta_qb + \nabla_x\Delta_q(a+5c) + \nabla_x \cdot (v\otimes v \mu^{1/2}, \Delta_qg_2) =0,\\
\partial_t \Delta_qc + \frac{1}{3}\nabla_x \cdot \Delta_qb +\frac{1}{6} \nabla_x \cdot (\vert v \vert^2 v \mu^{1/2}, \Delta_qg_2) =0,\\
\partial_t (A_{ij}(\Delta_qg_2)+2\Delta_qc\delta_{ij}) + \partial_j \Delta_qb_i + \partial_i \Delta_qb_j = A_{ij} (\Delta_qR_1+\Delta_qR_2),\\
\partial_t B_i(\Delta_qg_2) + \partial_i \Delta_qc = B_i(\Delta_qR_1+\Delta_qR_2).
\end{cases}
\end{align}
In what follows, by $\eqref{eq: fluid-type2}_n$ we denote the n-th equation of $\eqref{eq: fluid-type2}$.
Multiplying $\eqref{eq: fluid-type2}_5$ by $\partial_i \Delta_qc$, we have
\begin{align*}
\frac{d}{dt} (B_i (\Delta_q g_2), \partial_i \Delta_qc ) - (B_i (\Delta_q g_2), \partial_i \Delta_q\partial_tc )&+ \Vert \partial_i \Delta_qc \Vert^2_{L^2_x} \\
&= (B_i(\Delta_qR_1+\Delta_qR_2), \partial_i \Delta_qc).
\end{align*}
By Young's inequality it is obvious that
\begin{align*}
(B_i(\Delta_q R_1), \partial_i \Delta_q c) \le \varepsilon \Vert \partial_i \Delta_qc \Vert^2_{L^2_x} + \frac{C}{\varepsilon} \Vert \nabla_x \Delta_q g_2 \Vert^2_{L^2_{v,(\gamma+\nu)/2} L^2_x}
\end{align*}
for small positive $\varepsilon$. We note that if $\Vert g \Vert_{\tilde{L}^\infty_T \tilde{L}^2_v (B^{3/2}_x)} < \infty$, the Cauchy-Schwarz inequality gives $(g(x, \cdot, t), \phi) \in \tilde{L}^\infty_T (B^{3/2}_x)$ for any $\phi \in \mathcal{S}(\mathbb{R}^3_v)$. Also, that $\Vert g \Vert_{\mathcal{T}^{3/2}_{T,2,2}} < \infty$ implies $\sum_q 2^{3/2} \left( \int^T_0 \int \vert (\Delta_q g(x, \cdot, t), \phi) \vert^2 dx \right)^{1/2} < \infty$ for any $\phi \in \mathcal{S}(\mathbb{R}^3_v)$. In order to estimate $(B_i (\Delta_q g_2), \partial_i \Delta_q\partial_tc )$, we use $\eqref{eq: fluid-type2}_3$ and partial integral. As long as $f$, $g \in B^{3/2}_{2,1}$, $(\partial_i \Delta_q f, \Delta_q g) = - (\Delta_q f, \partial_i \Delta_q g)$. Indeed, if $f\in B^{3/2}_{2,1}$, $\Delta_q f$ is also in $B^{3/2}_{2,1} \subset H^1$. Moreover, that $f \in L^\infty$ implies $\Delta_q f \in C^\infty_b$. So both $(\partial_i \Delta_q f, \Delta_q g)$ and $-(\Delta_q f, \partial_i \Delta_q!
  g)$ exist and the same.
\begin{align*}
\vert (B_i (\Delta_q g_2), \partial_i \Delta_q\partial_tc ) \vert &= \left\vert \left(\partial_i B_i (\Delta_q g_2), \frac{1}{3}\nabla_x \cdot \Delta_qb +\frac{1}{6} \nabla_x \cdot (\vert v \vert^2 v \mu^{1/2}, \Delta_qg_2)\right)\right\vert \\
&\le \varepsilon \Vert \nabla_x \Delta_q b \Vert^2_{L^2_x} + \frac{C}{\varepsilon} \Vert \nabla_x \Delta_q g_2 \Vert^2_{L^2_{v,(\gamma+\nu)/2} L^2_x}.
\end{align*}
Here we set $\Vert \nabla_x \Delta_q b \Vert_{L^2x} :=\Vert \nabla_x \otimes \Delta_q b \Vert_{L^2x}$ for brevity. From this inequality, we have for small $\lambda'>0$ and $\varepsilon_1>0$,
\begin{align}\label{energy1}\notag
\frac{d}{dt}E^1_q(g)+& \lambda' \Vert \nabla_x \Delta_qc \Vert^2_{L^2_x} \le \varepsilon_1  \Vert \nabla_x \Delta_q b \Vert^2_{L^2_x} \\
&+ \frac{C}{\varepsilon_1} \Vert \nabla_x \Delta_q g_2 \Vert^2_{L^2_{v,(\gamma+\nu)/2} L^2_x} + \frac{C}{\varepsilon_1}\sum_i\Vert B_i(\Delta_q R_2) \Vert^2_{L^2_x}. 
\end{align}
Next, we multiply $\eqref{eq: fluid-type2}_4$ by $\partial_j \Delta_qb_i + \partial_i \Delta_qb_j$ and take summation with $i$ and $j$. We remark that 
\begin{align*}
\Vert \partial_j \Delta_qb_i + \partial_i \Delta_qb_j \Vert^2_{L^2_x} = \Vert \partial_j \Delta_qb_i \Vert^2_{L^2_x} + \Vert \partial_i \Delta_qb_j \Vert^2_{L^2_x} +2(\partial_j \Delta_qb_i,\partial_i \Delta_qb_j)\\
= \Vert \partial_j \Delta_qb_i \Vert^2_{L^2_x} + \Vert \partial_i \Delta_qb_j \Vert^2_{L^2_x} +2(\partial_i \Delta_qb_i,\partial_j \Delta_qb_j).
\end{align*}
Thus,
\begin{align*}
\sum_{i,j}\Vert \partial_j \Delta_qb_i + \partial_i \Delta_qb_j \Vert^2_{L^2_x} = 2\Vert \nabla_x  \Delta_q b\Vert^2_{L^2_x} + 2\Vert \nabla_x \cdot \Delta_q b\Vert^2_{L^2_x},
\end{align*}
and this implies
\begin{align*}
&\frac{d}{dt} E^2_q(g) - \sum_{i,j}(A_{ij}(\Delta_q g_2 )+2\Delta_qc\delta_{ij}, \partial_j \Delta_q\partial_tb_i + \partial_i \Delta_q\partial_tb_j) \\
&+ 2\Vert \nabla_x  \Delta_q b\Vert^2_{L^2_x} + 2\Vert \nabla_x \cdot \Delta_q b\Vert^2_{L^2_x} = ( A_{ij} (\Delta_qR_1+\Delta_qR_2), \partial_j \Delta_qb_i + \partial_i \Delta_qb_j).
\end{align*}
Substituting $\eqref{eq: fluid-type2}_2$ to eliminate $\partial_t \Delta_q b$, we have
\begin{align*}
&\vert (A_{ij}(\Delta_q g_2 )+2\Delta_qc\delta_{ij}, \partial_j \Delta_q\partial_tb_i) \vert \\
&= \left\vert \left( \partial_j A_{ij}(\Delta_q g_2 )+2\partial_j \Delta_q c\delta_{ij}, \partial_i \Delta_q(a+5c) + \partial_i \cdot (\sum_l v_iv_l \mu^{1/2}, \Delta_qg_2) \right) \right\vert \\
&\le \varepsilon \Vert \nabla_x \Delta_qa \Vert^2_{L^2_x}+ \frac{C}{\varepsilon} \Vert \nabla_x \Delta_qc \Vert^2_{L^2_x} + \frac{C}{\varepsilon} \Vert \nabla_x \Delta_q g_2 \Vert^2_{L^2_{v,(\gamma+\nu)/2} L^2_x}
\end{align*}
and other terms on the left hand side are similarly estimated. Hence, for small $\varepsilon_2 >0$ we have
\begin{align}\label{energy2}\notag
\frac{d}{dt} \sum_{i,j} E^2_q(g) +&\lambda' \Vert \nabla_x \Delta_qb \Vert^2_{L^2_x} \le \varepsilon_2 \Vert \nabla_x \Delta_qa \Vert^2_{L^2_x} + \frac{C}{\varepsilon_2} \Vert \nabla_x \Delta_qc \Vert^2_{L^2_x} \\
&+ \frac{C}{\varepsilon_2} \Vert \nabla_x \Delta_q g_2 \Vert^2_{L^2_{v,(\gamma+\nu)/2} L^2_x} + \sum_{i,j} \Vert A_{ij}(\Delta_q R_2) \Vert^2_{L^2_x}. 
\end{align}

Lastly, from $\eqref{eq: fluid-type2}_2$ we have
\begin{align*}
\frac{d}{dt} E^3_q(g)- \sum_{i}(\Delta_q b_i,\partial_i \Delta_q \partial_t a) +\Vert   \nabla_x \Delta_qa \Vert^2_{L^2_x} = -5\sum_{i} (\partial_i \Delta_q c,  \partial_i \Delta_q a) \\
- \sum_{i,j} (\partial_j(v_i v_j \mu^{1/2} ,\Delta_q g_2), \partial_i \Delta_q a).
\end{align*} 
Eliminating $\partial_t \Delta_q a$ by $\eqref{eq: fluid-type2}_1$, we have
\begin{align}\label{energy3}
\frac{d}{dt} E^3_q(g) + \lambda' \Vert \nabla_x \Delta_qa \Vert^2_{L^2_x} \le C \Vert \nabla_x \Delta_q(b,c) \Vert^2_{L^2_x} +C \Vert \nabla_x \Delta_q g_2 \Vert^2_{L^2_{v,(\gamma+\nu)/2} L^2_x}.
\end{align}
For sufficiently small $\delta_2$ and $\delta_3$ with $0 < \delta_3 \ll \delta_2 \ll 1$, taking summation $\eqref{energy1}+\delta_2\eqref{energy2}+ \delta_3\eqref{energy3}$ and then choosing small $\varepsilon_1$ and $\varepsilon_2$, we obtain for small $\lambda >0$,
\begin{align}\label{energy4}\notag
\frac{d}{dt} E_q(g(t))+\lambda &\Vert \nabla_x \Delta_q(a, b, c) \Vert^2_{L^2_x} \lesssim \Vert \nabla_x \Delta_q g_2 \Vert^2_{L^2_{v,(\gamma+\nu)/2} L^2_x} \\
&\quad + \sum_{i,j} \Vert A_{ij}(\Delta_q R_2) \Vert^2_{L^2_x} + \sum_{i} \Vert B_i(\Delta_q R_2) \Vert^2_{L^2_x}. 
\end{align}

Integrating \eqref{energy4} on $[0,T]$ and taking the square root of the resultant inequality, we have
\begin{align}\label{addition}\notag
\Vert \nabla_x \Delta_q(a, b, c) \Vert_{L^2_T L^2_x} &\lesssim \sqrt{\left|E_q(g(0))\right|} + \sqrt{\left|E_q(g(T))\right|} \\
&\notag + \Vert \nabla_x \Delta_q g_2 \Vert_{L^2_T L^2_{v,(\gamma+\nu)/2} L^2_x}
+ \sum_{i,j} \Vert A_{ij}(\Delta_q R_2) \Vert_{L^2_T L^2_x} \\&  + \sum_{i} \Vert B_i(\Delta_q R_2) \Vert_{L^2_T L^2_x}.
\end{align}
Multiplying this inequality by $2^{q/2}$ and taking summation over $q \ge -1$ yield
\begin{align}\label{ineq: macro-nexttolast}\notag
\Vert \nabla_x (a,b,c) \Vert_{\tilde{L}^2_T(B^{1/2}_x)} &\lesssim \sum_{q\ge -1}2^{q/2}\sqrt{\vert E_q(g(0))\vert} + \sum_{q\ge -1}2^{q/2}\sqrt{\vert E_q(g(T))\vert} \\
&+ \Vert g_2 \Vert_{\tilde{L}^2_T \tilde{L}^2_{v,(\gamma+\nu)/2} (B^{3/2}_x)} + \sum_{i,j}\sum_{q\ge -1}2^{q/2} \Vert A_{ij}(\Delta_q R_2) \Vert_{L^2_T L^2_x} \notag\\
&+ \sum_{i}\sum_{q\ge -1}2^{q/2} \Vert B_i(\Delta_q R_2) \Vert_{L^2_T L^2_x}.
\end{align}
Here we used Lemma \ref{Besov embedding} to estimate the sum
corresponding to 
the third 
term of the right hand side of \eqref{addition} by 
$\Vert g_2 \Vert_{\tilde{L}^2_T \tilde{L}^2_{v,(\gamma+\nu)/2} (B^{3/2}_x)}$. This term is governed by $\Vert g_2\Vert_{\mathcal{T}^{3/2}_{T,2,2}}$ since $\Vert \cdot \Vert_{L^2_{v,(\gamma+\nu)/2}} \lesssim |\!|\!| \cdot |\!|\!| $. The other four terms on the right hand side are estimated as follows. The Cauchy-Schwarz inequality yields
\begin{align*}
\sum_{q\ge -1}2^{q/2}\sqrt{\vert E_q(g(t))\vert} \lesssim \sum_{q\ge -1}2^{q/2} \left[ \Vert \nabla_x \Delta_q (a,b,c)(t) \Vert_{L^2_x} + \Vert \Delta_q (b,c)(t)\Vert_{L^2_x} \right. \\
\left.+\Vert \nabla_x \Delta_q g_2(t)\Vert_{L^2_vL^2_x}\right].
\end{align*}
Using Lemma \ref{Besov embedding} once again, we have
\begin{align*}
\sum_{q\ge -1}2^{q/2}\sqrt{\vert E_q(g(0))\vert} \lesssim \Vert g_0 \Vert_{\tilde{L}^2_v(B^{3/2}_x)}, \quad \sum_{q\ge -1}2^{q/2}\sqrt{\vert E_q(g(t))\vert} \lesssim \mathcal{E}_T(g).
\end{align*}
Applying Lemmas \ref{4-2} and \ref{4-3} with $\phi(v) =(v_iv_j -\delta_{ij})\mu^{1/2}(v)$ and $\phi(v)=\frac{1}{10}(\vert v \vert^2-5)v_i \mu^{1/2}(v)$ implies
\begin{align*}
\sum_{i,j}\sum_{q\ge -1}2^{q/2} \Vert A_{ij}(\Delta_q R_2) \Vert_{L^2_T L^2_x} + \sum_{i}\sum_{q\ge -1}2^{q/2} \Vert B_i(\Delta_q R_2) \Vert_{L^2_T L^2_x}\\
\lesssim \Vert g_2 \Vert_{\mathcal{T}^{3/2}_{T,2,2}} + \mathcal{E}_T(g)\mathcal{D}_T(g).
\end{align*}
Therefore, substituting the above four inequalities into \eqref{ineq: macro-nexttolast}, we finally obtained the desired estimate.
\end{proof}

At the end of this section, we derive an a priori estimate of the energy and the dissipation terms.
\begin{lem}\label{energy estimate}
It holds that
\begin{align*}
\mathcal{E}_T(g) + \mathcal{D}_T(g) \lesssim \Vert g_0 \Vert_{\tilde{L}^2_v(B^{3/2}_x)} +\left( \mathcal{E}_T(g) + \sqrt{\mathcal{E}_T(g)} \right) \mathcal{D}_T(g).
\end{align*}
\end{lem}

\begin{proof}
Applying $\Delta_q$ to \eqref{eq: perturbation} and taking the inner product with $\Delta_q g$ over $\mathbb{R}^3_x \times \mathbb{R}^3_v$, we have
\begin{align*}
(\partial_t \Delta_q g, \Delta_q g) + (v\cdot \Delta_q \nabla_x g, \Delta_q g) +(\mathcal{L}(\Delta_q g), \Delta_q g) = (\Delta_q \Gamma(g,g), \Delta_q g).
\end{align*}
Since $\Delta_q \nabla_x g= \nabla_x \Delta_q g$, we have $(v\cdot \Delta_q \nabla_x g, \Delta_q g)_{L^2_x}=0$. Moreover,
\begin{align*}
&(\Delta_q \Gamma(g, g), \Delta_q \mathbf{P}g)_{L^2_v}\\
&=\int_{\mathbb{R}^3_v \times \mathbb{R}^3_{v_*}\times \mathbb{S}^2} B \mu^{1/2}_* (\Delta_q (g_*' g') -\Delta_q (g_* g)) \Delta_q \mathbf{P}g dvdv_* d\sigma \\
&=\frac{1}{2} \int_{\mathbb{R}^3_v \times \mathbb{R}^3_{v_*}\times \mathbb{S}^2} B  \Delta_q (g_* g) \Delta_q ((\mathbf{Pg})'\mu'^{1/2}_* + (\mathbf{Pg})'_*\mu'^{1/2} \\
&\qquad \qquad \qquad \qquad \qquad \qquad \qquad \qquad-\mathbf{Pg}\mu^{1/2}_* -(\mathbf{Pg})_*\mu^{1/2})  dvdv_* d\sigma \\
&=0.
\end{align*}
Therefore, we have 
\begin{align*}
\frac{1}{2}\frac{d}{dt} \Vert \Delta_q g \Vert^2_{L^2_v L^2_x} +2\lambda_0 \Vert |\!|\!|\Delta_q (\mathbf{I}-\mathbf{P}) g) |\!|\!| \Vert^2_{L^2_x} \le \vert (\Delta_q \Gamma(g,g), \Delta_q (\mathbf{I}-\mathbf{P})g) )\vert,
\end{align*}
where $\lambda_0$ is taken in Lemma \ref{linear term}.

Integrate this inequality over $[0, t]$ with $0 \le t \le T$, take the square root of the resultant inequality and multiply it by $2^{3q/2}$. Then we have
\begin{align*}
2^{3q/2} \Vert \Delta_q g(t) \Vert_{L^2_v L^2_x} + \sqrt{\lambda_0} 2^{3q/2} \left( \int^t_0 \Vert |\!|\!| (\mathbf{I}-\mathbf{P})\Delta_q g(\tau) |\!|\!| \Vert^2_{L^2_x} d\tau \right)^{1/2} \\
\le 2^{3q/2} \Vert \Delta_q g_0 \Vert_{L^2_v L^2_x} + 2^{3q/2} \left( \int^t_0 \vert (\Delta_q \Gamma(g,g), \Delta_q (\mathbf{I}-\mathbf{P})g) )\vert d\tau \right)^{1/2}.
\end{align*}
We take supremum over $0\le t \le T$ on the left hand side and take the summation over $q \ge -1$. Together with Lemma \ref{4-1}, we have
\begin{align}\label{ineq: before-apriori}
\Vert g \Vert_{\tilde{L}^\infty_T \tilde{L}^2_v (B^{3/2}_x)} + \sqrt{\lambda_0} \Vert (\mathbf{I}-\mathbf{P})g \Vert_{\mathcal{T}^{3/2}_{T,2,2}} \le \Vert g_0 \Vert_{\tilde{L}^2_v (B^{3/2}_x)} + \sqrt{\mathcal{E}_T(g)}\mathcal{D}_T(g).
\end{align}
We now use Lemma \ref{macro-micro}. Taking summation $\delta \eqref{macro-apriori} + \eqref{ineq: before-apriori}$ with small positive $\delta$, we have
\begin{align*}
(1 -\delta) \mathcal{E}_t(g) + (\sqrt{\lambda_0} -\delta) \Vert (\mathbf{I}-\mathbf{P})g \Vert_{\mathcal{T}^{3/2}_{T,2,2}} + \delta \Vert \nabla_x (a, b, c) \Vert_{\tilde{L}^2_T (B^{1/2}_x)}\\
\lesssim \Vert g_0 \Vert_{\tilde{L}^2_v (B^{3/2}_x)} + (\sqrt{\mathcal{E}_T(g)}+\mathcal{E}_T(g))\mathcal{D}_T(g).
\end{align*}
The proof is complete by choosing sufficiently small $\delta$.
\end{proof}

Since \eqref{apriori-estimate-0} follows from Lemma \ref{energy estimate}, the existence of a unique 
global solution in Theorem \ref{main theorem} is a direct consequence of  Theorem \ref{local-existence}
 in the next section, concerning the existence of a local solution. The non-negativity of solutions to
the Cauchy problem \eqref{eq: Boltzmann} is sent to Section \ref{S6}.

\section{Local existence}\label{S5}
\setcounter{equation}{0}

\subsection{Local existence of linear and nonlinear equations}
\begin{lem}[Local existence for a linear equation]\label{lemma2}
There exist some
$C_0 > 1$, $\epsilon_0>0$, $T_0 >0$ such that for all $0 < T \le T_0$, $g_0 \in \tilde L^2_v(B_x^{3/2})$, $f \in 
\tilde L_T^{\infty} \tilde L_v^2(B_x^{3/2})$ 
satisfying
$$ \|f\|_{\tilde L_T^{\infty} \tilde L_v^2(B_x^{3/2})} \leq \epsilon_0,$$
the Cauchy problem
\begin{align}\label{l-e-c}\begin{cases}
\partial_tg+v\cdot \nabla_{x}g+\mathcal{L}_1g=\Gamma(f,g) -\mathcal{L}_2f 
,\\
g|_{t=0}=g_0,
\end{cases}\end{align}
admits a weak solution $g \in L^{\infty}([0,T]; L^2(\RR_{x,v}^6))$ satisfying
\begin{equation}\label{energy-important}
\|g\|_{\tilde L_T^{\infty} \tilde L^2_v (B_x^{3/2})}
+\|g\|_{\mathcal{T}_{T,2,2}^{3/2}} \le  C_0\Big( \|g_0\|_{\tilde L^2_v(B_x^{3/2})} + \sqrt{ T} \|f\|_{\tilde L_T^{\infty} \tilde L^2_v (B_x^{3/2})}\Big)
\end{equation}
\end{lem}
\begin{proof}
Consider
$$\mathcal{Q}=-\partial_t+(v\cdot \nabla_{x}+\mathcal{L}_1-\Gamma(f,\cdot))^*,$$
where the adjoint operator $(\cdot)^*$ is taken with respect to the scalar product in $L^2(\RR_{x,v}^6)$. 
Then,
for all $h \in C^{\infty}([0,T], \mathcal{S}(\rr_{x,v}^6))$, with $h(T)=0$ and~$0 \leq t \leq T$,
\begin{align*}
& \textrm{Re}\big(h(t),\mathcal{Q}h(t)\big)_{x,v} =   -\frac{1}{2}\frac{d}{dt}(\|h\|^2_{L^2_{x,v}})\\ \notag
&  \quad +\textrm{Re}(v\cdot\nabla_{x}h,h)_{{x,v}}+\textrm{Re}(\mathcal{L}_1 h,h)_{{x,v}}-\textrm{Re}(\Gamma(f,h),h)_{{x,v}} \\ \notag
\geq& \ -\frac{1}{2}\frac{d}{dt}\big(\|h(t)\|^2_{{L^2_{x,v}}} \big)
+\frac{1}{C}\Vert  |\!|\!| h(t) |\!|\!| \Vert^2_{L^2_x}- C \|h(t)\|_{{L^2_{x,v}}}^2\\
& \qquad -C\|f\|_{L^{\infty}([0,T] \times \RR_x^3; L^2(\rr_{v}^2))} 
\Vert |\!|\!| h(t) |\!|\!| \Vert^2_{L^2_x}\,,
\end{align*}
because $\mathcal{L}_1$ is a selfadjoint operator and $\textrm{Re}(v\cdot \nabla_{x}h,h)_{L^2_{x,v}}=0$.

Since $\tilde L_T^{\infty} \tilde L_v^2(B_x^{3/2}) \subset L^{\infty}([0,T] \times \RR_x^3; L^2(\rr_{v}^3))$, we have 
\begin{align*}
-\frac{d}{dt}\big(e^{2Ct}\|h(t)\|_{{L^2_{x,v}}}^2\big)+&\frac{1}{C}e^{2Ct}\Vert
 |\!|\!| h(t) |\!|\!| \Vert^2_{L^2_x} \\
&\leq 2e^{2Ct}\|h(t)\|_{{L^2_{x,v}}}\|\mathcal{Q}h(t)\|_{{L^2_{x,v}}},
\end{align*}
if $\varepsilon_0$ is sufficiently small. Since $h(T)=0$, for all $t \in [0,T]$ we have
\begin{align*}
& \ \|h(t)\|_{{L^2_{x,v}}}^2+\frac{1}{C}
\Vert |\!|\!| h |\!|\!| \Vert^2_{L^2([t,T]\times \RR^3_x)}\\
&\quad \leq  \  2\int_t^Te^{2C(\tau-t)}\|h(\tau)\|_{{L^2_{x,v}}}\|\mathcal{Q}h(\tau)\|_{{{L^2_{x,v}}}}d\tau \\
&\quad \leq 2e^{2CT}\|h\|_{L^{\infty}([0,T] ;L^2(\rr_{x,v}^6))}\|\mathcal{Q}h\|_{L^{1}([0,T],L^2(\rr_{x,v}^6))}, \enskip \text{so that}
\end{align*}
\begin{equation}\label{tr7} \|h\|_{L^{\infty}([0,T] ;L^2(\rr_{x,v}^6))} \leq 2 e^{2CT}\|\mathcal{Q}h\|_{L^{1}([0,T],L^2(\rr_{x,v}^6))}.
\end{equation}

Consider the vector subspace
\begin{align*}
\mathbb{W}&=\{w=\mathcal{Q}h : h \in C^{\infty}([0,T],\mathcal{S}(\rr_{x,v}^6)), \ h(T)=0\} \\
&\subset L^{1}([0,T],
L^2(\rr_{x,v}^6)).
\end{align*}
This inclusion holds because  it follows from Proposition \ref{upper-recent} that for $g \in L^2_{x,v}$
\begin{align*}
|(\Gamma(f,\cdot)^*h,g)_{L^2_{x,v}}|=|(h,\Gamma(f,g))_{L^2_{x,v}}|
\lesssim \|f \|_{L^\infty_x(L^2_v)}
\|g\|_{L^2_{x,v}} \|h\|_{L^2_x(H^{\nu}_{\gamma +\nu})}\,,
\end{align*}
and hence, for all $t \in [0,T]$,
$$\|\Gamma(f,\cdot)^*h\|_{L^2_{x,v}}  \lesssim
\|f \|_{L^\infty_x(L^2_v)} \|h\|_{L^2_x(H^{\nu}_{\gamma +\nu})}.$$
Since $g_0 \in L^2(\rr_{x,v}^6)$ we define 
the linear functional
\begin{align*}
\mathcal{G} \ : \qquad &\mathbb{W} \enskip \longrightarrow \cc\\ \notag
w=&\mathcal{Q}h \mapsto (g_0,h(0))_{L^2_{x,v}} -
(\mathcal{L}_2 f\,, h)_{L^2([0,T]; L^2_{x,v})}\,,
\end{align*}
where $h \in C^{\infty}([0,T],\mathcal{S}(\rr_{x,v}^6))$, with $h(T)=0$.
According to (\ref{tr7}), the operator $\mathcal{Q}$ is injective. The linear functional $\mathcal{G}$ is therefore well-defined. It follows from (\ref{tr7}) that $\mathcal{G}$ is a continuous linear form on $(\mathbb{W},\|\cdot\|_{L^{1}([0,T]; L^2(\rr_{x,v}^6))})$,
\begin{align*}
|\mathcal{G}(w)| &\leq \|g_0\|_{L^2_{x,v}}\|h(0)\|_{L^2_{x,v}}
+ C_T \|f\|_{L^\infty([0,T]; L^2_{x,v})} \|h\|_{L^{1}([0,T]; L^2(\rr_{x,v}^6))}
\\
&\le C'_T \Big ( \|g_0\|_{L^2_{x,v}}+ \|f\|_{L^\infty([0,T]; L^2_{x,v})}
\Big)
\|\mathcal{Q}h\|_{L^1([0,T];L^2(\rr_{x,v}^2))}\\
& = C'_T\Big ( \|g_0\|_{L^2_{x,v}}+ \|f\|_{L^\infty([0,T]; L^2_{x,v})}
\Big)\|w\|_{L^1([0,T];L^2(\rr_{x,v}^2))}\,.
\end{align*}
By using the Hahn-Banach theorem, $\mathcal{G}$ may be extended as a continuous linear form on
$$L^{1}([0,T]; L^2(\rr_{x,v}^6)),$$
with a norm smaller than 
$ C'_T \Big ( \|g_0\|_{L^2_{x,v}}+ \|f\|_{L^\infty([0,T]; L^2_{x,v})}
\Big)$.
It follows that there exists $g \in L^{\infty}([0,T]; L^2(\rr_{x,v}^6))$ satisfying
$$\|g\|_{L^{\infty}([0,T],L^2(\rr_{x,v}^6))} \leq 
 C'_T \Big ( \|g_0\|_{L^2_{x,v}}+ \|f\|_{L^\infty([0,T]; L^2_{x,v})}
\Big),$$
such that
$$\forall w \in L^{1}([0,T]; L^2(\rr_{x,v}^6)), \quad \mathcal{G}(w)=\int_0^T(g(t),w(t))_{L^2_{x,v}}dt.$$
This implies
that for all $h \in 
C_0^{\infty}((-\infty,T),\mathcal S(\rr_{x,v}^6))$,
\begin{align*}
\mathcal{G}(\mathcal{Q}h)&=\int_0^T(g(t),\mathcal{Q}h(t))_{L^2_{x,v}}dt\\
&=(g_0,h(0))_{L^2_{x,v}}
- \int_0^T (\mathcal{L}_2 f(t)\,, h(t))_{ L^2_{x,v}}dt\,.
\end{align*}
This shows that $g \in L^{\infty}([0,T]; L^2(\rr_{x,v}^6))$ is a weak solution of the Cauchy problem
\begin{equation}\label{cl1ff1}
\begin{cases}
\partial_tg+v\cdot \nabla_{x}g+\mathcal{L}_1g=\Gamma(f,g) -\mathcal{L}_2f 
,\\
g|_{t=0}=g_0.
\end{cases}
\end{equation}

It remains to prove that $g$ satisfies \eqref{energy-important}. Here we only give
a formal proof, temporarily assuming 
$g \in \tilde L_T^\infty \tilde L^2_v (B^{3/2}) \cap \cT^{3/2}_{T,2,2}$, 
because too many ingredients are required for its rigorous proof, which 
will be given at the end of this section.

\noindent
{\it Formal proof of \eqref{energy-important}}:
Applying $\Delta_q ( q \ge -1)$ to \eqref{cl1ff1} and taking the inner product $2^{3q} \Delta_q g$ over $\RR^6_{x,v}$, we obtain
\begin{align*}
&\frac{d }{dt}2^{3q} \|\Delta_q g\|^2_{x,v}
 + \frac{2^{3q}}{C} \| |\!|\!|\Delta_q g|\!|\!| \|_{L^2_x}^2 \\
&\quad \le 2^{3q+1} \left(\Delta_q \Gamma(f,g), \Delta_q g\right)_{x,v} +  2^{3q} \|\Delta_q f\|^2_{x,v}
+ C 2^{3q}\|\Delta_q g\|^2_{x,v}.
\end{align*}
Integrate this with respect to the time variable over $[0,t]$ with $0 \le t \le T$,
take the square root of
both sides of the resulting inequality and sum up over $q\ge -1$. Then it follows from \eqref{tri1} of Lemma \ref{trilinear} that
\begin{align}\label{local-energy}
&\|g\|_{\tilde L_T^\infty \tilde L^2_v (B^{3/2}) }+ \frac{1}{\sqrt C}\|g\|_{\cT^{3/2}_{T,2,2}}\le 
C'\|f\|^{1/2}_{\tilde L_T^\infty \tilde L^2_v (B^{3/2}) }\|g\|_{\cT^{3/2}_{T,2,2}} \notag \\
&\quad  + 2\Big( \|g_0\|_{\tilde L^2_v (B^{3/2}) } + \sqrt{ T} \|f\|_{\tilde L_T^\infty \tilde L^2_v (B^{3/2}) }  + 
 \sqrt{ CT} \|g\|_{\tilde L_T^\infty \tilde L^2_v (B^{3/2}) }\Big)\,.
\end{align}

\end{proof}

\begin{thm}[Local Existence]
\label{local-existence}
There exist $\varepsilon_1$, $T>0$ such that  if $ g_0 \in \tilde{L}^2_v (B^{3/2}_x)$ and 
\begin{equation*}
\Vert g_0 \Vert_{\tilde{L}^2_v (B^{3/2}_x)}\le \varepsilon_1,
\end{equation*}
then the Cauchy problem \eqref{cl1ff1}
admits a  unique solution 
$$g(x,v,t) \in \tilde L^\infty_T \tilde L^2_v (B^{3/2}) \, \mathop{\cap} \,  \cT_{T,2,2}^{3/2}\,\,.
$$
\end{thm}

\begin{proof}
Consider the sequence of approximate solutions defined by
\begin{equation}\label{itere}\begin{cases}
\partial_tg^{n+1}+v\cdot \nabla_{x}g^{n+1}+\mathcal{L}_1g^{n+1}=\Gamma(g^n,g^{n+1}) -\mathcal{L}_2g^n 
,\\
g^{n+1}|_{t=0}=g_0,   \hskip1.5cm ( n = 0,1,2,\cdots, \enskip g^0 = 0)\,.
\end{cases}
\end{equation}
Use Lemma \ref{lemma2} with $g = g^{n+1}, f = g^n$ and $T = \min \{T_0, 1/(4C_0^2)\}$.
Then we have 
\begin{equation}\label{iterate-sol}
{\|g^n \|_{\tilde L^\infty_T \tilde L^2 (B^{3/2})} + \|g^n\|_{\cT_{T,2,2}^{3/2}} \le \varepsilon_0}, 
\end{equation}
inductively,
if $\varepsilon_1$ is taken such that $2C_0 \varepsilon_1  \le \varepsilon_0$.
It remains to prove the convergence of the sequence 
$$\{g^n\,, \, n \in \NN\} \subset  \tilde L^\infty_T \tilde L^2_v (B^{3/2}) \, \mathop{\cap} \,  \cT_{T,2,2}^{3/2}\,.
$$
Setting $w^n = g^{n+1} - g^n$, from \eqref{itere} we have 
\[
\partial_t w^{n}+v\cdot \nabla_{x}w^{n}+\mathcal{L}_1w^{n} =\Gamma(g^n, w^n)+
\Gamma(w^{n-1},g^{n}) -\mathcal{L}_2w^{n-1}\,,
\]
with $w^{n}|_{t=0}=0$. Similar to the computation for \eqref{local-energy}, we obtain 
\begin{align*}
&\|w^n\|_{\tilde L_T^\infty \tilde L^2_v (B^{3/2}) }+ \frac{1}{\sqrt C}\|w^n\|_{\cT^{3/2}_{T,2,2}}  \le 
C\Big( \|g^n\|^{1/2}_{\tilde L_T^\infty \tilde L^2_v (B^{3/2}) }\|w^n\|_{\cT^{3/2}_{T,2,2}} \\
& \qquad \qquad \qquad  + \|w^{n-1}\|^{1/2}_{\tilde L_T^\infty \tilde L^2_v (B^{3/2}) } \|g^n\|^{1/2}_{\cT^{3/2}_{T,2,2}}
\|w^n\|^{1/2}_{\cT^{3/2}_{T,2,2}}\\
&\qquad \qquad \qquad + \sqrt{ T} \|w^{n-1}\|^{1/2}_{\tilde L_T^\infty \tilde L^2_v (B^{3/2}) }  
\|w^{n}\|^{1/2}_{\tilde L_T^\infty \tilde L^2_v (B^{3/2}) }\Big)\,.
\end{align*}
If $\varepsilon_0$ and $T$ are sufficiently small then we have  
\begin{align*}  \|w^n\|_{\tilde L_T^\infty \tilde L^2_v (B^{3/2}) } + \frac{1}{\sqrt C}\|w^n\|_{\cT^{3/2}_{T,2,2}}
 &\le \lambda \,\,\|w^{n-1}\|_{\tilde L_T^\infty \tilde L^2_v (B^{3/2}) }\\
& \le \lambda^{n-1} \,\,\|w^{1}\|_{\tilde L_T^\infty \tilde L^2_v (B^{3/2}) }
\end{align*}
for some $0< \lambda <1$,  
which concludes that
\[
\mbox{ $\{g^n\}$ is a Cauchy sequence in $\tilde L_T^\infty \tilde L^2_v (B^{3/2})\cap 
\cT^{3/2}_{T,2,2}$, }
\]
and the limit function $g$ is a desired solution to the Cauchy problem
\begin{align*}
\partial _t g + v\cdot\nabla_x g + \cL g=\Gamma(g,\,g), \,\,\,\,  g|_{t=0}=g_0\,.
\end{align*}

\end{proof}

\subsection{Rigorous proof of \eqref{energy-important}}

The preceding proof of \eqref{energy-important} is formal
since we a priori  assumed the left hand side of \eqref{local-energy} is finite. The rigorous proof requires more involved 
procedure. We start by the following lemma. 

\begin{lem}
\label{modify}
Let $0<s\le \frac{3}{2}$ and $0<T \le \infty$. For $ M \in \NN$ put 
$\displaystyle f_M = S_M f = \sum _{q\ge -1}^{M -1} \Delta_q f$. 
If $g$ satisfies 
\begin{align}\label{triple-norm-finite}
\left \Vert |\!|\!| g  |\!|\!| \right \Vert_{L^2_T L^2_x}^2  = \int_0^T \int_{\RR_x^3}
|\!|\!| g|\!|\!| ^2 dx dt < \infty\,,
\end{align}
then there exists a  $C>0$ independent of $M$ such that   for any $\kappa >0$
\begin{align}\label{cut-x}\notag
\sum_{q \ge -1} &\frac{2^{qs}}{1+ \kappa 2^{2qs}}\left( \int^T_0 \vert ( \Delta_q \Gamma (f_M,g), \Delta_q g)_{x,v} \vert dt \right)^{1/2}\\ &\le C \Vert f \Vert_{\tilde{L}^\infty_T \tilde{L}^2_v(B^{{3/2}}_x)}^{1/2} \Vert g \Vert_{\mathcal{T}_{T,2,2}^{s, \kappa}} \notag
 \\
&\qquad + C_M  \Vert f_M \Vert_{\tilde{L}^\infty_T \tilde{L}^2_v(B^{{3/2}}_x)}^{1/2} \left \Vert |\!|\!| g  |\!|\!| \right \Vert_{L^2_T L^2_x}
 \,,  
\end{align}
where $C_M >0$ is a constant depending only on $M$ and 
\[
\Vert g \Vert_{\cT^{s, \kappa}_{T,2,2}} = \sum_{q \ge -1}^\infty \frac{2^{qs}}
{1+\kappa 2^{2qs}} \left \Vert \,
|\!|\!| \Delta_q g|\!|\!|\, \right\Vert_{L^{2}_T L^2_x}  .
\]
\end{lem}
\begin{proof} We notice that 
$\Vert g \Vert_{\cT^{s, \kappa}_{T,2,2}} < \infty$  for each $\kappa >0$ follows from \eqref{triple-norm-finite}.
The proof of the lemma is the almost same procedure as the one for \eqref{tri1}
of Lemma \ref{trilinear}. Indeed, recalling the Bony decomposition, for $\Gamma^1 (f_M, g)$
we obtain 
\begin{align*}
&\sum_{q \ge -1} \frac{2^{qs}}{1+\kappa 2^{2qs}}\left( \int^T_0 \vert ( \Delta_q \Gamma^1 (f_M,g), \Delta_q g)_{x,v} \vert dt \right)^{1/2} \\ \notag
&\lesssim \sum_{q \ge -1}\frac{2^{qs}}{1+\kappa 2^{2qs}} \left(\int_0^T \sum_{\vert j-q \vert \le 4} \Vert f \Vert_{L^2_v L^\infty_x} \Vert |\!|\!| \Delta_j g |\!|\!| \Vert_{L^2_x} \Vert |\!|\!| \Delta_q g |\!|\!| \Vert_{L^2_x} dt \right)^{1/2} \\
\notag
&\lesssim \Vert f \Vert_{\tilde{L}^\infty_T \tilde{L}^2_v (B^{3/2}_x)}^{1/2} \left( \sum_{q \ge -1} 
\frac{2^{qs}}{1+\kappa 2^{2qs}}\left( \sum_{\vert j-q \vert \le 4} \int_0^T \Vert |\!|\!| \Delta_j g |\!|\!| \Vert_{L^2_x}^2 dt \right)^{1/2} \right)^{1/2} \\
\notag
& \qquad \qquad \times \left( \sum_{q \ge -1} \frac{2^{qs}}{1+\kappa 2^{2qs}} \left(\int_0^T \Vert |\!|\!| \Delta_q g |\!|\!| \Vert_{L^2_x}^2 dt \right)^{1/2} \right)^{1/2} \\ \notag
&\lesssim \Vert f \Vert_{\tilde{L}^\infty_T \tilde{L}^2_v (B^{3/2}_x)}^{1/2} \Vert g \Vert_{\mathcal{T}^{s, \kappa}_{T,2,2}}, 
\end{align*}
similar as \eqref{tri-11}.
As for $\Gamma^2(f_M,  g)$, we have\begin{align*}
&\sum_{q \ge -1} \frac{2^{qs}}{1+\kappa 2^{2qs}}\left( \int^T_0 \vert ( \Delta_q \Gamma^2 (f_M,g), \Delta_q g)_{x,v} \vert dt \right)^{1/2} \\ 
& \lesssim \sum_{q \ge -1} \frac{2^{qs}}{1+\kappa 2^{2qs}} \left( \sum_{\stackrel{\vert j-q \vert \le 4}{
j \le M}}\int^T_0 \Vert |\!|\!| S_{j-1} g |\!|\!| \Vert_{L^\infty_x} \Vert \Delta_j f_M \Vert_{L^2_{x,v}} \Vert |\!|\!| \Delta_q g |\!|\!| \Vert_{L^2_x} dt \right)^{1/2} \\
& \lesssim \sum_{q \ge -1}^{M+4} \frac{2^{qs}}{1+\kappa 2^{2qs}}\left( 
\Vert f_M \Vert_{\tilde{L}^\infty_T \tilde{L}^2_v (B^{3/2}_x)}^{1/2}\Vert |\!|\!| S_{M+4} g |\!|\!| \Vert_{L^2_TL^2_x} \Vert |\!|\!| \Delta_q g |\!|\!| \Vert_{L^2_TL^2_x} \right)^{1/2} \\
& \le C_M 
\Vert f_M \Vert_{\tilde{L}^\infty_T \tilde{L}^2_v (B^{3/2}_x)}^{1/2} \Vert |\!|\!| g |\!|\!| \Vert_{L^2_TL^2_x} .
\end{align*}
Since 
\begin{align*}
\Delta_q\left( \sum_j \sum_{\vert j-j'\vert \le 1} (\Delta_{j'} f_M, \Delta_j g) \right)&= \Delta_q \left(\sum_{\max\{j,j'\} \ge q-2 } \sum_{\vert j-j' \vert \le 1} (\Delta_{j'} f_M, \Delta_j g)\right)\\
&=0  \enskip \mbox{if $q \geq M+3$},
\end{align*}
the term corresponding to $\Gamma^3$ is estimated as follows:
\begin{align*}
&\sum_{q \ge -1} ^{M+2} \frac{2^{qs}}
{1+ \kappa 2^{2qs}}\left( \int^T_0 \vert ( \Delta_q \Gamma^3 (f_M,g), \Delta_qg)_{x,v} \vert dt \right)^{1/2} \\
& \lesssim \sum_{q \ge -1}^{M+2}  \frac{2^{qs}}
{1+ \kappa 2^{2qs}}\left( \sum_{j \le M+1} \int^T_0 \Vert f_M \Vert_{L^2_vL^\infty_x} \Vert |\!|\!| \Delta_j g |\!|\!| \Vert_{L^2_x} \Vert |\!|\!| \Delta_q g |\!|\!| \Vert_{L^2_x} dt \right)^{1/2} \\
& \le C_M  \Vert f_M \Vert_{\tilde L^\infty_T \tilde L^2_v (B^{3/2}_x)}^{1/2} \Vert |\!|\!|  g |\!|\!| \Vert_{L^2_{T,x}}  .
\end{align*} 
Thus the proof is completed.
\end{proof}

For each $f_M$ $(M \in \NN)$ we consider a weak solution $g_M \in  L^{\infty}([0,T]; L^2(\rr_{x,v}^6))$ to the Cauchy problem
\eqref{l-e-c} with $f$ replaced by $f_M$. If $g_M$ satisfies \eqref{triple-norm-finite}, then by the same procedure as in the formal proof of 
\eqref{energy-important}  we obtain 
\begin{align*}
&\sum_{q \ge -1} \frac{2^{3q/2}}{1+ \kappa 2^{3q}} \sup_{0 \le t \le T } \|\Delta_q g_M(t)\|_{x,v} + \frac{1}{\sqrt C}
\Vert g_M \Vert_{\cT^{3/2, \kappa}_{T,2,2}} \\
&\quad \le C'\|f\|^{1/2}_{\tilde L_T^\infty \tilde L^2_v (B^{3/2}) }\|g_M\|_{\cT^{3/2,\kappa}_{T,2,2}} \notag\\
&\quad  + 2\Big( \sum_{q \ge -1} \frac{2^{3q/2}}{1+ \kappa 2^{3q}} \|\Delta_q g_0\|_{x,v} + \sqrt{ T} \|f\|_{\tilde L_T^\infty \tilde L^2_v (B^{3/2}) } 
+ 
 \sqrt{ CT}\\
& \quad  \times \sum_{q \ge -1} \frac{2^{3q/2}}{1+ \kappa 2^{3q}} \sup_{0 \le t \le T } \|\Delta_q g_M(t)\|_{x,v}
\Big) + C_M  \Vert f_M \Vert_{\tilde{L}^\infty_T \tilde{L}^2_v(B^{{3/2}}_x)}^{1/2} \left \Vert |\!|\!| g_M  |\!|\!| \right \Vert_{L^2_T L^2_x}\,.
\end{align*}
Choose a small $T >0$ independent of $M$ and $\kappa$.  Then letting $\kappa \rightarrow 0$, we get 
\[
\|g_M \|_{\tilde L_T^\infty \tilde L^2_v (B^{3/2}) }+ \|g_M \|_{\cT^{3/2}_{T,2,2}} < \infty\,,
\]
which permits the preceding formal proof and we obtain the energy estimate \eqref{local-energy} for each $g_M$. 
Set $w_{M,M'} = g_M - g_{M'}$, $M, M' \in \NN$.  Then it follows from \eqref{l-e-c} that
\begin{align*}
&\partial_t w_{M,M'} +v\cdot \nabla_{x} w_{M, M'} +\mathcal{L}_1w_{M,M'}\\
& \qquad \qquad =\Gamma(f_M - f_{M'}, g_M)
+ \Gamma(f_{M'}, w_{M,M'}) -\mathcal{L}_2(f_M - f_{M'})\,.
\end{align*}
Since $\{ f_M\}$ is a Cauchy sequence in $\tilde L_T^{\infty} \tilde L_v^2(B_x^{3/2})$, it is easy to see that 
$\{g_M\}$ is a Cauchy sequence in $ \tilde L^\infty_T \tilde L^2_v (B^{3/2}) \, \mathop{\cap} \,  \cT_{T,2,2}^{3/2}$, 
by the similar manipulation as in the proof of Theorem \ref{local-existence}. Since $ \displaystyle g = \lim_{M\rightarrow \infty} g_M$
belongs to $ \tilde L^\infty_T \tilde L^2_v (B^{3/2}) \, \mathop{\cap} \,  \cT_{T,2,2}^{3/2}$, the preceding formal proof is justified. 

It remains to show \eqref{triple-norm-finite} for a weak solution
$g_M \in L^{\infty}([0,T]; L^2(\rr_{x,v}^6))$, under the assumption that 
$\|f_M \|_{L^\infty([0,T]\times \RR_x^3; L_2(\RR^3_v))}$ is sufficiently small, independent of $M$, and
moreover $\|\nabla_x f_M \|_{L^\infty([0,T]\times \RR_x^3; L_2(\RR^3_v))} < \infty$.
For the brevity,
we write $g$ and $f$ instead of $g_M$ and $f_M$, respectively.

Let us take different $1 > \delta, \delta' >0$. We use a weight function
$W_{\delta'}(v) = \la \delta' v \ra^{-N}$ for $N \ge 1$ and mollifiers 
$M^\delta(D_v), S_\delta(D_x)$ defined in subsections \ref{ap-4}, \ref{ap-6}, respectively. 
Multiply
$
W_{\delta'}(v)S_\delta(D_x)\Big(M^\delta(D_v)\Big)^2S_\delta(D_x)W_{\delta'}(v) g
$
by the equation \eqref{l-e-c}, and integrate with respect to $t \in [0,T]$ and $(x,v) \in \RR^6$. 
Notice that 
\begin{align*}
	&\Big(\Gamma(f,g), W_{\delta'}S_\delta\Big(M^\delta\Big)^2S_\delta W_{\delta'} g\Big)_{{x,v}}
- \Big(\Gamma(f, M^\delta S_\delta W_{\delta'}g)
 , M^\delta S_\delta W_{\delta'} g\Big)_{x,v}\\
&=\Big(W_{\delta'}\Gamma(f,g)  - \Gamma(f, W_{\delta'}g) , S_\delta\Big(M^\delta\Big)^2S_\delta W_{\delta'} g\Big)_{x,v}\\
&+\Big( S_\delta \Gamma(f, W_{\delta'}g) - \Gamma(f, S_\delta W_{\delta'}g) , \Big(M^\delta\Big)^2S_\delta W_{\delta'} g\Big)_{x,v}\\
&+ \Big(\Gamma(f, S_\delta W_{\delta'}g)- Q(\mu^{1/2}f, S_\delta W_{\delta'}g)
 , \Big(M^\delta\Big)^2S_\delta W_{\delta'} g\Big)_{x,v}\\
&+ \Big(M^\delta Q(\mu^{1/2}f, S_\delta W_{\delta'}g)-Q(\mu^{1/2}f, M^\delta S_\delta W_{\delta'}g)
 , M^\delta S_\delta W_{\delta'} g\Big)_{x,v}\\
&+ \Big(Q(\mu^{1/2}f, M^\delta S_\delta W_{\delta'}g)-\Gamma(f, M^\delta S_\delta W_{\delta'}g)
 , M^\delta S_\delta W_{\delta'} g\Big)_{x,v}\\
&=\mbox{(1)} +\mbox{(2)} +\mbox{(3)}+\mbox{(4)}+\mbox{(5)}\,.
\end{align*}
It follows from Proposition \ref{weight-commutator} and Lemma \ref{M-bounded-triple} that 
\begin{align}\label{1}\notag
\int_0^T (1) dt &\lesssim {\delta'}^{\nu/2} \|f\|_{L^\infty([0,T]\times \RR^3_x; L^2_v)}
\|\la v\ra^{(\gamma+\nu)/2} W_{\delta'} g\|_{L^2([0,T]\times \RR^6_{x,v})}\\
& \qquad \qquad \qquad \times \| |\!|\!| M^\delta S_\delta W_{\delta'}g|\!|\!| \|_{L^2(([0,T]\times \RR^3_x)}
\end{align}
By means of \eqref{1st}, we have
\begin{align}\label{2} \notag 
\int_0^T (2) dt &\lesssim \delta^{1-\nu/2}
\|\nabla f \|_{L^\infty_{T,x}(L^2_v )}
\|\la v \ra^{|\gamma|/2 + \nu} W_{\delta'} g\|_{L^2([0,T]\times \RR^6_{x,v})}\\
&\qquad\times \|M^\delta S_\delta W_{\delta'}g \|_{L^2([0,T]\times \RR^3_x; 
H^{\nu/2}_{\gamma/2}(\RR^3_v))}
\end{align}
because 
\[
\frac{\delta \la \xi \ra^{\nu}}{(1+ \delta\la \xi \ra)^{N_0}} \le \delta^{1-\nu/2} \la \xi \ra^{\nu/2}.
\]
It follows from \eqref{diff-G-Q} and Lemma \ref{M-bounded-triple} that 
\begin{align}\label{3-5}\notag
\int_0^T (3) + (5) dt &\lesssim \|f\|_{L^\infty([0,T]\times \RR^3_x; L^2_v)}
\|\la v\ra^{(\gamma+\nu)/2} W_{\delta'} g\|_{L^2([0,T]\times \RR^6_{x,v})}\\
& \qquad \qquad \qquad \times \| |\!|\!| M^\delta S_\delta W_{\delta'}g|\!|\!| \|_{L^2(([0,T]\times \RR^3_x)}
\end{align}
Thanks to Proposition \ref{IV-coro-2.15} and Proposition \ref{prop2.9_amuxy3}, for a $0< \nu' < \nu$ 
we have 
\begin{align*}
\int_0^T (4) dt &\lesssim \|f\|_{L^\infty([0,T]\times \RR^3_x; L^2_v)}\Big( \|M^\delta S_\delta W_{\delta'}g \|^2_{L^2([0,T]\times \RR^3_x; 
H^{\nu'/2}_{(\nu+\gamma)^+}(\RR^3_v))}
\\
&\notag \qquad \qquad \qquad + \|\la v\ra^{(\gamma+\nu)^+} W_{\delta'} g\|^2_{L^2([0,T]\times \RR^6_{x,v})}\Big)\,.
\end{align*}
Note 
that for any $\varepsilon, \kappa, \ell >0$ there exist $C_{\kappa, \varepsilon,\ell} >0$
and $N_{\ell, \varepsilon} >0$ such that
\[
\|h\|_{H^{\nu/2-\varepsilon}_{\ell}} \le \kappa \|h\|_{H^{\nu/2}} + C_{\kappa, \varepsilon,\ell}
\|h\|_{L^2_{N_{\ell, \varepsilon} }},
\]
(see for example Lemma 2.4 of \cite{HMUY}).
Therefore we obtain 
\begin{align}\label{4}\notag 
\int_0^T (4) dt &\lesssim \|f\|_{L^\infty([0,T]\times \RR^3_x; L^2_v)}\Big( 
\kappa \|M^\delta S_\delta W_{\delta'}g \|^2_{L^2([0,T]\times \RR^3_x; 
H^{\nu/2}_{\gamma/2}(\RR^3_v))}
\\
&\qquad \qquad \qquad + C_\kappa \|\la v\ra^N W_{\delta'} g\|^2_{L^2([0,T]\times \RR^6_{x,v})}\Big)\,.
\end{align}
for a sufficiently large $N >0$.  Summing up \eqref{1}-\eqref{4}, for any $\kappa >0$ we obtain 
\begin{align}\label{non-linear-2}\notag
&\int_0^T\Big(\Gamma(f,g), W_{\delta'}S_\delta\Big(M^\delta\Big)^2S_\delta W_{\delta'} g\Big)_{x,v}dt 
\le ( C_1\|f\|_{L^\infty_{T,x}(L^2_v)} + \kappa )\\
&  \quad \qquad \times \| |\!|\!| M^\delta S_\delta W_{\delta'}g|\!|\!| \|^2_{L^2(([0,T]\times \RR^3_x)} + C_{\kappa, f}  \|\la v\ra^N W_{\delta'} g\|^2_{L^2([0,T]\times \RR^6_{x,v})},
\end{align}
by means of  \eqref{upper-fundamental1}. Similarly we obtain 
\begin{align}\label{linear-0}\notag 
&\int_0^T\Big(\cL_1 g , W_{\delta'}S_\delta\Big(M^\delta\Big)^2S_\delta W_{\delta'} g\Big)_{x,v}dt\\
& \qquad  \notag  \ge 
\int_0^T\Big(\cL_1 M^\delta S_\delta W_{\delta'} g, M^\delta S_\delta W_{\delta'} g\Big)_{x,v}dt
\\
& \qquad \quad - \frac{\lambda_0}{2} \| |\!|\!| M^\delta S_\delta W_{\delta'}g|\!|\!| \|^2_{L^2(([0,T]\times \RR^3_x)} + C \|\la v\ra^N W_{\delta'} g\|^2_{L^2([0,T]\times \RR^6_{x,v})}.
\end{align}
To handle $v\cdot \nabla_x$ term, we use the similar device as in (3.2.4) of \cite{AMUXY2010}. Indeed,
\begin{align}\label{nabla}\nonumber
&\Big(v \cdot \nabla_x g, W_{\delta'}S_\delta\Big(M^\delta\Big)^2S_\delta W_{\delta'} g\Big)_{x,v}\\
&\quad  =\Big([M^\delta, v] \cdot \nabla_x S_\delta W_{\delta'}g, M^\delta S_\delta W_{\delta'} g\Big)_{x,v}
 \le 2 \|M^\delta S_\delta W_{\delta'} g\|^2_{L^2(\RR_{x,v}^6)},
\end{align}
because $|\big(\nabla_\xi M^\delta(\xi)\big) \cdot \big(\eta S_\delta(\eta)\big)|
\le N_0 M^\delta(\xi) |\delta \eta| S(\delta \eta) \le 2 M^\delta(\xi) S_\delta(\eta)$.
It follows from Lemma \ref{linear term} and \eqref{non-linear-2}-\eqref{nabla}  that we obtain 
\[
\| |\!|\!| M^\delta S_\delta W_{\delta'}g|\!|\!| \|^2_{L^2(([0,T]\times \RR^3_x)} \le C_T\Big(
\|\la v\ra^N W_{\delta'} g\|^2_{L^2([0,T]\times \RR^6_{x,v})} + \|g_0\|^2_{L^2(\RR_{x,v}^6)}\Big),
\]
if $C_1\|f\|_{L^\infty_{T,x}(L^2_v)} \le \lambda_0 /8$. 
For arbitrary but fixed $\delta'>0$, we get 
\[
 \| |\!|\!| W_{\delta'}g|\!|\!| \|^2_{L^2(([0,T]\times \RR^3_x)} < C_{\delta'}
\]
by letting $\delta \rightarrow 0$. 

We repeat the same procedure without the factor $\Big(M^\delta(D_v)\Big)^2$, that is, 
multiply 
$
\Big(W_{\delta'}(v)S_\delta(D_x)\Big)^2 g
$
by the equation, and integrate with respect to $t \in [0,T]$ and $(x,v) \in \RR^6$.  By the same way as 
\eqref{1}, we get 
\begin{align*}
	&\int_0^T \left|\Big(\Gamma(f,g),  \Big(W_{\delta'}S_\delta \Big)^2g\Big)_{x,v}
- \Big(\Gamma(f, W_{\delta'}g)
 , \big(S_\delta\big)^2 W_{\delta'} g\Big)_{x,v} \right|dt \\
&\lesssim {\delta'}^{\nu/2} \|f\|_{L^\infty([0,T]\times \RR^3_x; L^2_v)}
 \| |\!|\!| W_{\delta'}g|\!|\!| \|^2_{L^2(([0,T]\times \RR^3_x)}\,.
\end{align*}
By using \eqref{AMUXYtrilinear} for $\Big(\Gamma(f, W_{\delta'}g)
 , \big(S_\delta\big)^2 W_{\delta'} g\Big)_{x,v} $, we obtain 
\begin{align*}
	&\int_0^T \left|\Big(\Gamma(f,g),  \Big(W_{\delta'}S_\delta \Big)^2g\Big)_{x,v}
\right|dt 
\lesssim  \|f\|_{L^\infty([0,T]\times \RR^3_x; L^2_v)}
 \| |\!|\!| W_{\delta'}g|\!|\!| \|^2_{L^2(([0,T]\times \RR^3_x)}\,.
\end{align*}
Since $\big(v\cdot \nabla_x g, \big(W_{\delta'}S_\delta \big)^2g\big)_{x,v} =0$, by letting 
$\delta \rightarrow 0$ we obtain 
\begin{align*}
&(\lambda_0 - C_1  \|f\|_{L^\infty([0,T]\times \RR^3_x; L^2_v)}) \| |\!|\!| W_{\delta'}g|\!|\!| \|^2_{L^2(([0,T]\times \RR^3_x)} \\
&\quad \lesssim    {\delta'}^{\nu} 
\| |\!|\!| W_{\delta'}g|\!|\!| \|^2_{L^2(([0,T]\times \RR^3_x)}
+ 
\|g\|^2_{L^2([0,T]\times \RR^6_{x,v})} + \|g_0\|^2_{L^2(\RR_{x,v}^6)}\,,
\end{align*}
because of Lemma \ref{linear term} and \eqref{linear-comm-weight}.
Finally we obtain  \eqref{triple-norm-finite}  by letting $\delta' \rightarrow 0$. 

\section{Non-negativity of solutions}\label{S6}
\setcounter{equation}{0}
The method of the proof is the almost same as the one of Proposition 5.2 in \cite{amuxy4-3}. For the self-containedness,   
we reproduce it.  If $\{g^n\}$ is the sequence of approximate solutions in the proof of Theorem \ref{local-existence}, and if $f^n = \mu + \mu^{1/2}g^n$, then 
$\{f^n\}$ is constructed successively by the following linear Cauchy problem
\begin{equation}\label{4.4.3}
\left\{\begin{array}{l} \partial_t f^{n+1} + v\,\cdot\,\nabla_x
f^{n+1} =Q (f^n, f^{n+1}), \\ f^{n+1}|_{t=0} = f_0 =\mu + \mu^{1/2}
 g_0\geq 0\, , \enskip (n=0,1,2,\cdots, \, f^0 =\mu).
\end{array} \right.
\end{equation}
Hence the non-negativity of the solution to the original Cauchy problem \eqref{eq: Boltzmann} comes  from the following induction argument: 
Suppose  that
\begin{equation}\label{4.4.1+100}
f^n = \mu + \mu^{1/2} g^n \geq 0\,,
\end{equation}
for some $n\in\NN$. Then  (\ref{4.4.1+100}) is true
for $n+1$.

If we put $\tilde f^n = \mu^{-1/2}f^n = \mu^{1/2} + g^n$ then $\tilde f^n$ satisfies 
\begin{equation}\label{4.4.3-bis}
\left\{\begin{array}{l} \partial_t \tilde f^{n+1} + v\,\cdot\,\nabla_x
\tilde f^{n+1} =\Gamma (\tilde f^n, \tilde f^{n+1}), \\ \tilde f^{n+1}|_{t=0} =\tilde f_0 = \mu^{1/2} + 
 g_0\geq 0\, , \enskip (n=0,1,2,\cdots, \, \tilde f^0 =\mu^{1/2}).
\end{array} \right.
\end{equation}
It follows from \eqref{iterate-sol} and Lemma \ref{T-estimate} that $\int_0^T \| |\!|\!|\tilde f^n |\!|\!| \|^2_{L^\infty_x} dt < \infty$, and hence,
if $\tilde f^n_{\pm} = \pm \max(\pm \tilde f^n, 0)$ then we have
\[
\int_0^T \| |\!|\!|\tilde f^n_+ |\!|\!| \|^2_{L^\infty_x} dt+\int_0^T \| |\!|\!|\tilde f^n_- |\!|\!| \|^2_{L^\infty_x} dt
< \infty
\]
by means of the same argument as in the proof of Theorem 5.2 in \cite{amuxy4-3}. 
Take the convex function $\beta (s) = \frac 1 2 (s^- )^2= \frac 1 2
s\,(s^- )  $ with $s^-=\min\{s, 0\}$.  Let $\varphi(v,x) = (1+|v|^2 +|x|^2)^{\alpha/2}$ with $\alpha >3/2$,
and notice that
\begin{align*}
\beta_s (\tilde f^{n+1}) \varphi(v,x)^{-1}&: =\left(\frac{d}{ds}\,\,\beta
\right)(\tilde f^{n+1}) \varphi(v,x)^{-1}\\
&={\tilde f_{ -}^{n+1}} \varphi(v,x)^{-1}\in
L^\infty([0,T];L^2(\RR^6_{x,v}).
\end{align*}
Multiply the first equation of \eqref{4.4.3-bis}
by $\beta_s (\tilde f^{n+1})\varphi(v,x)^{-2}$ $
= \tilde f_{-}^{n+1}\varphi(v,x)^{-2}$ and integrate over
$[0,t] \times \RR^6_{x,v}$, ($t \in (0,T]$). Then, in view of $\beta(f^{n+1}(0)) = \tilde f_{0, -}^2/2 =0$, we have
\begin{align*}
&\int_{\RR^6} \beta ( \tilde f^{n+1}(t)) \varphi(v,x)^{-2}dxdv
\\
&\qquad =\int_0^t \int_{\RR^6}
\Gamma(\tilde f^n(\tau),\, \tilde f^{n+1}(\tau) )\,\, \beta_s(\tilde f^{n+1}(\tau)) \varphi(v,x)^{-2} \,\,dxdvd\tau \\
&\qquad\qquad-\int_0^t \int_{\RR^6} { v\,\cdot\, \nabla_x \enskip  (  \beta
(\tilde f^{n+1}(\tau)) \varphi(v,x)^{-2}) }dxdv d\tau\\
&\qquad\qquad + \int_0^t \int_{\RR^6} {\big (\varphi(v,x)^{2}\,\,v\, \cdot\,
\nabla_x \,\varphi(v,x)^{-2} \big) }\enskip  \beta (\tilde f^{n+1}(\tau))\varphi(v,x)^{-2}dxdvd\tau, 
\end{align*}
where the first term on the right hand side is well defined because 
$$\int_0^T \| |\!|\!|\tilde f^{n+1} |\!|\!| \|^2_{L^\infty_x} dt 
+ \int_0^T \| |\!|\!|\tilde f^{n+1}_{-} |\!|\!| \|^2_{L^\infty_x} dt < \infty.
$$
Since the second term vanishes
and $|v\, \cdot\, \nabla_x\, \varphi(v,x)^{-2} | \leq C \varphi(v,x)^{-2}$,
we obtain
\begin{align*}
&\int_{\RR^6} \beta ( \tilde f^{n+1}(t)) \varphi(v,x)^{-2}dxdv
\\
&\qquad \le \int_0^t \Big(\int_{\RR^6}
\Gamma(\tilde f^n(\tau),\, \tilde f^{n+1}(\tau) )\,\, \beta_s(\tilde f^{n+1}(\tau)) \varphi(v,x)^{-2} \,\,dxdv\Big)d\tau \\
&\qquad\qquad \qquad \qquad \qquad +C \int_0^t 
\int_{\RR^6} \beta ( \tilde f^{n+1}(\tau)) \varphi(v,x)^{-2}dxdvd\tau\,.
\end{align*}
The integrand $(\cdot)$ of the  first term on the right hand side is equal to
\begin{align*}
&\int_{\RR^6}  \Gamma(\tilde f^n,  \tilde f_{ -}^{n+1} ) \tilde f_{ -}^{n+1} \varphi(v,x)^{-2} dxdv\\
&\qquad\qquad\qquad + \int   B \, \mu_{*}^{1/2}( {\tilde f_*}^n)' ( \tilde f_{+}^{n+1})' \tilde 
 f_{-}^{n+1} \varphi(v,x)^{-2}dvdv_* d\sigma dx \\
 &=A_1 + A_2\,.
\end{align*}
>From the induction hypothesis, the second term $A_2$ is non-positive.

On the other hand, we have
\begin{align*}
&A_1 = \int( \Gamma  (\tilde f^n, \varphi(v,x)^{-1}\tilde  f_{-}^{n+1}),  \varphi(v,x)^{-1} \tilde f_{-}^{n+1})_{L^2(\RR_v^3)} dx + R
\\
&= - \int (\cL_1(\varphi(v,x)^{-1}\tilde  f_{-}^{n+1}), \varphi(v,x)^{-1}\tilde  f_{-}^{n+1})_{L^2(\RR_v^3)} dx\\ 
 & + \int( \Gamma  ( g^n, \varphi(v,x)^{-1}\tilde  f_{-}^{n+1}),  \varphi(v,x)^{-1} \tilde f_{-}^{n+1})_{L^2(\RR_v^3)} dx +R ,
\end{align*}
where $R$ is a remainder term. It follows from Corollary \ref{cor-weight-commutator} with $f = \tilde f^n$ that for any $\kappa >0$
\begin{align*}
\int_0^t |R| d\tau \leq&  \kappa \int_0^t \int_{\RR_x^3}
|\!|\!| \varphi(v,x)^{-1}\tilde  f_{-}^{n+1}(\tau) |\!|\!|^2 dx d\tau  \\
&+  C_\kappa 
\int_0^t 
\int_{\RR^6} \beta ( \tilde f^{n+1}(\tau)) \varphi(v,x)^{-2}dxdvd\tau\,.
\end{align*}
By means of Lemma \ref{linear term} and Corollary  \ref{AMUXYtrilinear} with \eqref{iterate-sol}, we obtain
\begin{align*}
\int_0^t A_1 d\tau 
&\le -(\lambda_0 - C(\kappa + \varepsilon_0))
\int_0^t \int_{\RR_x^3}
|\!|\!| \varphi(v,x)^{-1}\tilde  f_{-}^{n+1}(\tau) |\!|\!|^2 dx  d\tau \\
&+  C_\kappa 
\int_0^t 
\int_{\RR^6} \beta ( \tilde f^{n+1}(\tau)) \varphi(v,x)^{-2}dxdvd\tau\,.
\end{align*}
Finally we get 
\begin{align*}
\int_{\RR^6} \beta ( \tilde f^{n+1}(t)) \varphi(v,x)^{-2}dxdv
\lesssim  \int_0^t 
\int_{\RR^6} \beta ( \tilde f^{n+1}(\tau)) \varphi(v,x)^{-2}dxdvd\tau,
\end{align*}
which implies that 
$\tilde f^{n+1}(t,x,v) \ge 0$ for $(t,x,v) \in [0,T]\times \RR^6$.

\section{Appendix}\label{AP}
\setcounter{equation}{0}
\subsection{Fundamental inequalities and Besov embedding theorems}\label{ap-1}
First of all, the following lemma plays the central role in this paper.
\begin{lem}\label{MS-general}
Let $\nu \in (0,2)$ and $\gamma > \max\{-3,-3/2-\nu\}$. For any $\alpha  \ge 0$ and any 
$\beta \in \RR$ there exists a $C = C_{\alpha, \beta} >0$ such that  
\begin{align}\label{upper-fundamental123}\notag
&\left\vert ( \Gamma(f,g),h )_{L^2_v} \right\vert \le C \Big(\Vert \mu^{1/10}f \Vert_{L^2_v} |\!|\!| g |\!|\!| |\!|\!| h |\!|\!| \\
& \quad \qquad \qquad + \|f\|_{L^2_{-\alpha} }
\|g\|_{L^2_{(\gamma+\nu)/2+\alpha - \beta}}
\|h \|_{L^2_{(\gamma+\nu)/2+\beta}}\Big). 
\end{align}
\end{lem}
This lemma will be proved in the next subsection together with another upper bound 
estimate. 
If we put $\alpha = \beta =0$ then we obtain \cite[Theorem 1.2]{AMUXY3}, namely,
\begin{cor}\label{AMUXYtrilinear}
Let $\nu \in (0,2)$ and $\gamma > \max\{-3,-3/2-\nu\}$. Then it holds 
\begin{align}\label{upper-fundamental1}
\left\vert ( \Gamma(f,g),h )_{L^2_v} \right\vert \lesssim \Vert f \Vert_{L^2_v} |\!|\!| g |\!|\!| |\!|\!| h |\!|\!|. 
\end{align}
\end{cor}
In the case $\gamma+\nu \le 0$, we employ the following, by setting
$\alpha = -(\gamma+ \nu)/2$ and $\beta =(\gamma+ \nu)/2$ or $0$.
\begin{cor}\label{another-type}
When $\gamma + \nu \le 0$ we have 
\begin{align}\label{cut-part-123-bis}\notag
\Big|\Big(\Gamma ( f, g)\, & , \,h \Big)_{L^2_v }\Big|
\lesssim \|\mu^{1/10}f\|_{L^2} |\!|\!|g|\!|\!| |\!|\!|h |\!|\!|\\
& + \|f\|_{L^2_{(\gamma+ \nu)/2} }\min \big\{
\|g\|_{L^2_{(\gamma+ \nu)/2}}
\|h \|_{L^2},  \, \|g\|_{L^2}
\|h \|_{L^2_{(\gamma+ \nu)/2}}\big \}\,.
\end{align}
\end{cor}
We also need the lower estimate of $(\cL_1 f,f)$ and the upper estimate of $ ({\cL}_2 f,f)$, respectively.
\begin{lem}\label{linear term} Let $\nu \in (0,2)$ and $\gamma > -3$. Then
there exists a constant $\lambda_0>$ such that
\begin{align*}
&(\cL_1 f,f)_{L^2_v} \ge \frac{1}{2} (\cL f,f)_{L^2_v } 
\ge \lambda_0 |\!|\!| (\mathbf{I}-\mathbf{P})f |\!|\!|^2,\\
& \quad \left\vert (\cL_2f, g)_{L^2_v} \right\vert \lesssim \Vert \mu^{10^{-3}}f\Vert_{L^2_v} \Vert \mu^{10^{-3}}g\Vert_{L^2_v} \,.
\end{align*}
\end{lem}

The first and second inequalities of the above lemma are from \cite[Proposition 2.1]{AMUXY2}
 and \cite[Lemma 2.15]{AMUXY2}, respectively. 

We catalogue a couple of lemmas which are frequently used in this paper.
\begin{lem}\label{bounded operators}
Let $ 1\le p \le \infty$ and $f\in L^p_x$, then there exists a constant $C>0$ independent of $p,q$ and $f$ such that
\begin{align*}
\Vert \Delta_q f \Vert_{L^p_x} \le C \Vert f \Vert_{L^p_x}, \quad \Vert S_q f \Vert_{L^p_x} \le C\Vert f \Vert_{L^p_x}. 
\end{align*}
In short, $\Delta_q$ and $S_q$ are bounded operators on $L^p_x$.
\end{lem}

\begin{lem}\label{embedding}
Let $1\le p, r \le \infty$. Then
\begin{enumerate}
\item $B^{s_1}_{pr} \hookrightarrow B^{s_2}_{pr}$ when $s_2 \le s_1$. This inclusion does not hold for the homogeneous Besov space.
\item $B^{3/p}_{p,1} \hookrightarrow L^\infty$ and $\dot{B}^{3/p}_{p,1} \hookrightarrow L^\infty$ when $1\le p < \infty$.
\end{enumerate}
\end{lem}

\begin{lem}\label{Besov embedding}
Let $ 1 \le p,q,r \le \infty$ and $s>0$. Then we have 
\begin{align*}
\Vert \nabla_x \cdot \Vert_{\tilde{L}^q_T(\dot{B}^s_{pr})} \sim \Vert \cdot \Vert_{\tilde{L}^q_T(\dot{B}^{s+1}_{pr})}, \quad \Vert \cdot \Vert_{\tilde{L}^q_T(\dot{B}^s_{pr})} \lesssim \Vert \cdot \Vert_{\tilde{L}^q_T(B^s_{pr})}. 
\end{align*}
\end{lem}

\begin{lem}\label{Besov and Chemin-Lerner}
Let $ 1 \le p,\alpha, \beta, r \le \infty$ and $s>0$. If $r\le \min\{\alpha, \beta\} $ then
\begin{align*}
\Vert f \Vert_{L^{\alpha}_TL^{\beta}_v (B^s_{pr})} \le \Vert f \Vert_{\tilde{L}^{\alpha}_T \tilde{L}^{\beta}_v (B^s_{pr})} \ and\  \Vert f \Vert_{L^{\alpha}_TL^{\beta}_v (\dot{B}^s_{pr})} \le \Vert f \Vert_{\tilde{L}^{\alpha}_T \tilde{L}^{\beta}_v (\dot{B}^s_{pr})}.
\end{align*}
\end{lem}

For the proof of Lemma \ref{bounded operators}, Lemma \ref{embedding} - Lemma \ref{Besov embedding} and Lemma \ref{Besov and Chemin-Lerner}, readers may refer to \cite{BCD}, \cite{XK} and \cite{DLX}, respectively.

We prove a useful lemma to deal with terms consisting of $\Vert |\!|\!| f |\!|\!| \Vert_{L^\infty_x}$.
\begin{lem}\label{T-estimate}
For each $T>0$, we have
\begin{align*}
\left( \int^T_0 \Vert |\!|\!| f |\!|\!| \Vert_{L^\infty_x}^2 dt \right)^{1/2} \lesssim \Vert f \Vert_{\dot{\mathcal{T}}^{3/2}_{T,2,2}}.
\end{align*}
\end{lem}

\begin{proof}
By defition of $|\!|\!| f |\!|\!|$ and using both generalized Minkowski's inequality and Besov embedding $ \dot{B}^{3/2}_x \hookrightarrow L^\infty_x$, we calculate that
\begin{align*}
&\left( \int^T_0 \Vert |\!|\!| f |\!|\!| \Vert_{L^\infty_x}^2 dt \right)^{1/2} \\
&\le \left[ \int^T_0 \int B \left\{ \mu_* \sup_x (f'-f)^2 + (\mu^{1/2} - \mu'^{1/2}) \sup_xf^2_* \right\} dvdv_* d\sigma dt \right]^{1/2} \\
& \lesssim  \left[ \int^T_0 \int B \left\{ \mu_* \left( \sum_l 2^{\frac{3}{2}l}\left( \int \vert \dot{\Delta}_l  (f'-f)\vert^2 dx\right)^{1/2} \right)^2 \right.\right.\\
& \left.\left. \qquad+ (\mu^{1/2} - \mu'^{1/2}) \left( \sum_l 2^{\frac{3}{2}l}\left( \int \vert \dot{\Delta}_l  f_*\vert^2 dx\right)^{1/2} \right)^2  \right\} dvdv_* d\sigma dt \right]^{1/2} \\
&\le \sum_l \left[  \int^T_0 \int B \left\{ \mu_* 2^{3l}\vert \dot{\Delta}_l (f'-f)\vert^2 \right.\right. \\ 
&\left.\left.\qquad \qquad+ (\mu^{1/2} - \mu'^{1/2}) 2^{3l} \vert \dot{\Delta}_l f_*\vert^2 \right\} dxdvdv_* d\sigma dt \right]^{1/2}  = \Vert f \Vert_{\dot{\mathcal{T}}^{3/2}_{T,2,2}}.
\end{align*}
\end{proof}

\subsection{Proof of Lemma \ref{MS-general}}
\label{ap-2}

In this subsection we also give another upper bound for $\Gamma$, which is a  supplementary variant of \cite[Lemma 4.7]{amuxy4-3} concerning 
the range of the index of the Sobolev space. 

\begin{prop}\label{upper-recent}
Let $0<\nu<2$, $\gamma > \max\{-3, -\nu-3/2\}$.  For any $\ell \in \RR$ and $m \in
[-\nu/2,\nu/2]$ we have
\begin{align}\label{different-g-h}
\Big | \Big(\Gamma( f, \, g),\, h \Big ) _{L^2} \Big|
\lesssim
\|f\|_{L^2}
\|g\|_{H^{\nu/2+m}_{(\ell +\gamma+\nu)^+}}\|h\|_{H^{\nu/2-m}_{-\ell}}\,.
\end{align}
\end{prop}

We decompose the kinetic factor of the cross-section {into} two parts,
\begin{align*}
\Phi(|z|)=|z|^\gamma
= |z|^\gamma \varphi_0(|z|)+ |z|^\gamma
\big(1-\varphi_0(|z|)\big) = \Phi_c(|z|) + \Phi_{\bar c}(|z|),
\end{align*}
where $\varphi_0\in C^\infty_0(\RR)$, Supp $\varphi_0\subset [-1, 1]; \varphi(t)=1$ for
$|t|\leq 1/2$, and put
\[
B_{c} = \Phi_{c} (|v-v_*|) b\Big(\frac{v-v_*}{|v-v_*|} \cdot \sigma\Big)\,,\quad
B_{\bar c} = \Phi_{\bar c} (|v-v_*|) b\Big(\frac{v-v_*}{|v-v_*|} \cdot \sigma\Big)\,.
\]
{Accordingly}, we write
\[
Q(f,g) = Q_c(f,g) + Q_{\bar c}(f,g)\,,
\]
and
\[
\Gamma(f,g) = \Gamma_c(f,g) + \Gamma_{\bar c}(f,g)\,.
\]

\begin{prop} \label{upper-c}
Let $0< \nu<2,  \gamma > \max\{-3, -\nu-3/2\}$.  If $m \in\  [-\nu/2,\nu/2]$
then we have
\[
| (Q_c (f,  g),h )|  \lesssim \|f\|_{L^{2} }\Vert g\Vert_{H^{\nu/2+ m}} \|h\|_{H^{\nu/2 -m}}\,.
\]
\end{prop}
\begin{rema}
For any $q \in [1,2)$,  let $\gamma > \max\{-3, -\nu -3 + 3/q\}$ and 
$-\nu/2 \le m \le \min\{\nu/2, 3/2 -\nu/2\}$.
Then we have
\[
| (Q_c (f,  g),h )|  \lesssim \|f\|_{L^q }\Vert g\Vert_{H^{\nu/2+ m}} \|h\|_{H^{\nu/2 -m}}\,.
\]
\end{rema}

This proposition is an improvement of \cite[Proposition 2.1]{amuxy4-3}, concerning 
the lower bound of $m$, where it was assumed $m \ge \nu/2 -1$. 
For the self-containedness and the convenience of readers, we repeat the detail proof,
including the general case pointed out at the above remark.

For the proof of Proposition \ref{upper-c}, we shall follow some of
the arguments from \cite{amuxy4-3}. First of all, by using the formula from
the Appendix of \cite{advw}, and as in \cite{amuxy4-3}, one has
\begin{align*}
( Q_c(f, g), h ) =& \int b \Big({\frac\xi{ | \xi |}} \cdot \sigma \Big) [ \hat\Phi_c (\xi_* - \xi^- ) - \hat \Phi_c (\xi_* ) ] \hat f (\xi_* ) \hat g(\xi - \xi_* ) \overline{{\hat h} (\xi )} d\xi d\xi_*d\sigma .\\
= & \int_{ | \xi^- | \leq {\frac 1 2} \la \xi_*\ra }  \cdots\,\, d\xi d\xi_*d\sigma
+ \int_{ | \xi^- | \geq {\frac 1 2} \la \xi_*\ra } \cdots\,\, d\xi d\xi_*d\sigma \,\\
=& A_1(f,g,h)  +  A_2(f,g,h) \,\,,
\end{align*}
where $\hat f (\xi )$ is the Fourier transform of $f$ with respect to $v\in\RR^3$ and
$ \xi^-=\frac{1}{2}(\xi-|\xi|\sigma)$.

Then, we write $A_2(f,g,h)$ as
\begin{align*}
A_2 &=  \int b \Big({\frac \xi{ | \xi |}} \cdot \sigma \Big) {\bf 1}_{ | \xi^- | \ge {\frac 1 2}\la \xi_*\ra }
\hat\Phi_c (\xi_* - \xi^- ) \hat f (\xi_* ) \hat g(\xi - \xi_* ) \overline{{\hat h} (\xi )} d\xi d\xi_*d\sigma .\\
&- \int b \Big({\frac\xi{ | \xi |}} \cdot
 \sigma \Big){\bf 1}_{ | \xi^- | \ge {\frac 1 2}\la \xi_*\ra } \hat \Phi_c (\xi_* ) \hat f (\xi_* ) \hat g(\xi - \xi_* ) \overline{{\hat h} (\xi )} d\xi d\xi_*d\sigma \\
&= A_{2,1}(f,g,h) - A_{2,2}(f,g,h)\,.
\end{align*}
While for $A_1$, we use the Taylor expansion of $\hat \Phi_c$ at
order $2$ to have
$$
A_1 = A_{1,1} (f,g,h) +A_{1,2} (f,g,h)
$$
where
$$
A_{1,1} = \int b\,\, \xi^-\cdot (\nabla\hat\Phi_c)( \xi_*)
{\bf 1}_{ | \xi^- | \leq {\frac 1 2} \la \xi_*\ra }  \hat f (\xi_* ) \hat g(\xi - \xi_* ) \overline{\hat{h}(\xi)} d\xi d\xi_*d\sigma,
$$
and $A_{1,2} (f,g,h)$ is the remaining term corresponding to the second order term in the Taylor expansion of $\hat\Phi_c$. The $A_{i,j}$ with
$i,j=1,2$ are estimated by the following lemmas.
\begin{lem}\label{A-1}
For any $q \in [1,2]$,  let $\gamma > \max\{-3, -\nu -3 + 3/q\}$. Furthermore, assume that 
$m \le 3/2 -\nu/2$ if $q <2$. Then 
we have
\[
| A_{1,1} |+| A_{1,2}|  \lesssim \|f\|_{L^q} \Vert f\Vert_{H^{\nu/2+ m}} \|h\|_{H^{\nu/2 -m}}\,.
\]
\end{lem}
\begin{proof}
Considering firstly $A_{1,1}$, by writing
\[
\xi^- = \frac{|\xi|}{2}\left(\Big(\frac{\xi}{|\xi|}\cdot \sigma\Big)\frac{\xi }{|\xi|}-\sigma\right)
+ \left(1- \Big(\frac{\xi}{|\xi|}\cdot \sigma\Big)\right)\frac{\xi}{2},
\]
we see that the integral corresponding to the first term on the right hand side vanishes because of the symmetry
on $\SS^2$.
Hence, we have
\[
A_{1,1}= \int_{\RR^6} K(\xi, \xi_*)
\hat f (\xi_* ) \hat g(\xi - \xi_* ) \overline{\hat h(\xi )} d\xi d\xi_* \,,
\]
where
\[
K(\xi,\xi_*) = \int_{\SS^2}
 b \Big({\frac \xi{ | \xi |}} \cdot \sigma \Big)
\left(1- \Big(\frac{\xi}{|\xi|}\cdot \sigma\Big)\right)\frac{\xi}{2}\cdot
(\nabla\hat\Phi_c)( \xi_*)
{\bf 1}_{ | \xi^- | \leq {\frac 1 2} \la \xi_*\ra } d \sigma \,.
\]
Note that $| \nabla \hat \Phi_c (\xi_*) | \lesssim {\frac 1{\la
\xi_*\ra^{3+\gamma +1}}}$, from the Appendix of \cite{AMUXY2}. If $\sqrt 2
|\xi| \leq \la \xi_* \ra$, then $|\xi^-| \leq \la \xi_* \ra/2$ and
this implies  the fact that $0 \leq \theta \leq \pi/2$, and we have
\begin{align*}
|K(\xi,\xi_*)| &\lesssim \int_0^{\pi/2} \theta^{1-\nu} d \theta\frac{ \la \xi\ra}{\la \xi_*\ra^{3+\gamma +1}}
\lesssim \frac{1  }{\la \xi_*\ra^{3+\gamma}}\left(
\frac{\la \xi \ra}{\la \xi_*\ra}\right) \,.
\end{align*}
On the other hand, if $\sqrt 2 |\xi| \geq \la \xi_* \ra$, then
\begin{align*}|K(\xi,\xi_*)| &\lesssim \int_0^{\pi\la \xi_*\ra /(2|\xi|)} \theta^{1-\nu} d \theta\frac{ \la \xi\ra}{\la \xi_*\ra^{3+\gamma +1}}
\lesssim \frac{1  }{\la \xi_*\ra^{3+\gamma}}\left(
\frac{\la \xi \ra}{\la \xi_*\ra}\right)^{\nu-1}\,.
\end{align*}
Hence we obtain
\begin{align}\label{later-use1}
&|K(\xi,\xi_*)|
\lesssim
\frac{1}{\la \xi_*\ra^{3+\gamma}}\left\{
\frac{\la \xi \ra}{\la \xi_*\ra}{\bf 1}_{\sqrt 2 |\xi| \leq \la \xi_* \ra}\right. \notag \\
&\qquad \left. +{\bf 1}_{ \sqrt 2 |\xi| \geq  \la \xi_* \ra \geq |\xi|/2}
+
\left(
\frac{\la \xi \ra}{\la \xi_*\ra}\right)^{\nu}
{\bf 1}_{\la \xi_* \ra \leq |\xi|/2}\right\}\,.
\end{align}
Notice that
\begin{equation}\label{equivalence-relation}
\left \{ \begin{array}{ll}
\la \xi \ra \lesssim \la \xi_* \ra \sim \la \xi-\xi_*\ra   &\mbox{on supp ${\bf 1}_{\la \xi_* \ra\geq \sqrt 2 |\xi|}$}\\
\la \xi \ra \sim \la \xi-\xi_*\ra   &\mbox{on supp ${\bf 1}_{\la \xi_*\ra \leq |\xi |/2 } $} \\
\la \xi \ra \sim \la \xi_* \ra \gtrsim   \la \xi-\xi_*\ra  &
\mbox{on supp ${\bf 1}_{\sqrt 2 |\xi| \geq \la \xi_*\ra \geq | \xi|/2 }$\,.}
\end{array}
\right.
\end{equation}
Take  an $\varepsilon >0$ such that  $3 +\gamma +\nu >  3/q + \varepsilon$. Then 
we have 
\begin{align*}
\frac{1}{\la \xi_*\ra^{3+\gamma}}
\frac{\la \xi \ra}{\la \xi_*\ra}{\bf 1}_{\sqrt 2 |\xi| \leq \la \xi_* \ra}
\lesssim \frac{\la \xi -\xi_*\ra^{\nu/2+m} \la \xi \ra^{\nu/2-m} \la \xi \ra^{-3/2-\varepsilon}}{
\la \xi_* \ra^{3/2+\gamma +\nu -\varepsilon}}
\end{align*}
in view of  $\nu/2-m -3/2- \varepsilon < 1$. Replacing the factor
$(\la\xi \ra/\la \xi_*\ra)^{\nu}{\bf 1}_{\la \xi_* \ra \leq |\xi|/2}$ on the right hand side of \eqref{later-use1}
by
\[
 \frac{\la \xi \ra^{\nu/2-m} \la \xi-\xi_*\ra^{\nu/2+m}}{\la \xi_*\ra^{\nu}}\,,
\]
we obtain
\begin{align}\label{kernel-estimate}
|K(\xi,\xi_*)| \lesssim  &
 \frac{\la \xi -\xi_*\ra^{\nu/2+m} \la \xi \ra^{\nu/2-m} \la \xi \ra^{-3/2-\varepsilon}}{
\la \xi_* \ra^{3/2+\gamma +\nu -\varepsilon}}+ \frac{\la \xi\ra^{\nu/2-m}  \la \xi-\xi_*\ra^{\nu/2+m} }{\la \xi_*\ra^{3+\gamma +\nu}} \notag
\\ 
&+ \frac{{\bf 1}_{ \la \xi -\xi *\ra \lesssim \la \xi_{*}  \ra}}{\la \xi_*\ra^{3+\gamma +\nu/2-m}
\la \xi-\xi_*\ra^{\nu/2+m} }   \la \xi\ra^{\nu/2-m}  \la \xi-\xi_*\ra^{\nu/2+m}
\,.
\end{align}
Putting $\tilde {\hat g}(\xi)=  \la \xi \ra^{\nu/2+m}  \hat g(\xi), \tilde {\hat h}(\xi)=  \la \xi \ra^{\nu/2 -m} \hat h(\xi)$,
we have 
\begin{align*}
|A_{1,1}|^2 &\lesssim \left(\int_{\RR^3} \frac{ |\tilde {\hat h}(\xi)|}{\la \xi\ra^{3/2 +\varepsilon}}
\left(\int_{\RR^3} \frac{|\hat f(\xi_*)|}{\la \xi_* \ra^{3/2+\gamma +\nu -\varepsilon}}|\tilde {\hat g}(\xi-\xi_*)|d\xi_*
\right)d\xi\right)^2 \\
&+\int_{\RR^6}
\frac{|\hat f(\xi_*)| }{\la \xi_*\ra^{3+\gamma +\nu}}|\tilde {\hat g}(\xi-\xi_*)|^2 d\xi d\xi_*
\int_{\RR^6}
\frac{|\hat f(\xi_*)| }{\la \xi_*\ra^{3+\gamma +\nu}}|\tilde {\hat h}(\xi)|^2 d\xi d\xi_*\\
&+
 \int_{\RR^6}
\frac{|\hat f(\xi_*)|^2 }{\la \xi_*\ra^{6+2\gamma +\nu-2m}} \frac{{\bf 1}_{ \la \xi -\xi *\ra \lesssim \la \xi_{*}  \ra}}{
\la \xi-\xi_*\ra^{\nu+2m} }d\xi d\xi_*
\int_{\RR^6}
|\tilde {\hat g}(\xi-\xi_*)|^2 |\tilde {\hat h}(\xi)|^2 d\xi d\xi_*\\
&= \cK^2 +\cA \cB +  \cD \cE\,,
\end{align*}
by the Cauchy-Schwarz inequality. 
It follows from H\"older inequality that
\begin{align*}
\cK &\lesssim \|g\|_{H^{\nu/2+m}}\|h\|_{H^{\nu/2-m}}\left( \int_{\RR^3} 
\frac{|\hat f(\xi_*)|^2}{\la \xi_* \ra^{3+2(\gamma +\nu -\varepsilon)}}d\xi_*\right)^{1/2}\\
&\lesssim
\|f\|_{L^q} \|g\|_{H^{\nu/2+m}}\|h\|_{H^{\nu/2-m}}\,,
\end{align*}
where we have used the fact that 
\begin{align}\label{p-estimate}
&\int_{\RR^3} 
\frac{|\hat f(\xi_*)|^2}{\la \xi_* \ra^{\ell}}d\xi_* \lesssim \|f\|^2_{L^q} \enskip \nonumber\\
&\qquad \mbox{if} \enskip \ell > -3 + 6/q
\enskip \mbox{for $q \in [1,2)$ and $\ell \ge 0$ for $q=2$}
\end{align}
by means of 
$\|\hat f\|_{L^{q'}} \le \|f\|_{L^q}$ with $1/q +1/q'=1$ for $q \in [1,2]$ .

Since it follows from $3+\gamma +\nu > 3/q$ that $\la \xi_*\ra^{-(3+\gamma +\nu)} \in L^q$, the H\"older inequality again shows
\[
\cA \lesssim  \int_{\RR^3}\frac{|\hat f(\xi_*)|}{\la \xi_*\ra^{3+\gamma +\nu}} d\xi_*  \|g\|^2_{H^{\nu/2+m}}
\lesssim \|f\|_{L^q}\|g\|^2_{H^{\nu/2+m}}
\,,\enskip \cB \lesssim \|f\|_{L^q} \|h\|^2_{H^{\nu/2-m}}\,.
\]
Note that
\[
\int\frac{{\bf 1}_{ \la \xi -\xi *\ra \lesssim \la \xi_{*}  \ra}}{
\la \xi-\xi_*\ra^{\nu+2m} } d\xi \enskip \lesssim
\left \{ \begin{array}{lcl}
\displaystyle  \frac{1}{\la \xi_*\ra^{-3+\nu+2m} }& \mbox{if} & \nu/2+m <3/2\\
 \log \la \xi_*\ra & \mbox{if} & \nu/2+m \ge 3/2\,.
\end{array}
\right.
\]
Since  $3+2(\gamma +\nu) >0$ when $q =2$, together with  $6+2 \gamma+ \nu-2m >0$,
we get
$\cD \le \|f \|_{L^2}^2 $, which concludes the desired bound for $A_{1,1}$ when $q=2$.
In the case where $q \in [1,2)$, it follows from \eqref{p-estimate} that $\cD \lesssim \|f\|^2_{L^q}$ because of 
$\nu/2+m \le 3/2$.

Now we consider  $A_{1,2} (f, g, h)$, which comes from the second order term of the Taylor expansion. Note that
$$
A_{1,2} = \int  b \Big({\frac \xi{ | \xi |}} \cdot \sigma \Big)\int^1_0 d\tau (\nabla^2\hat \Phi_c) (\xi_* -\tau\xi^- ) (\xi^-)^2  \hat f (\xi_* ) \hat g(\xi - \xi_* ) \bar{\hat h} (\xi ) d\sigma d\xi d\xi_*\, .
$$
Again from the Appendix of \cite{AMUXY2}, we have
$$
| (\nabla^2\hat \Phi_c) (\xi_* -\tau\xi^- ) | \lesssim {\frac 1{\la  \xi_* -\tau \xi^-\ra^{3+\gamma +2}}}
\lesssim
 {\frac 1{\la \xi_*\ra^{3+\gamma +2}}},
$$
because $|\xi^-| \leq \la \xi_*\ra/2$.
Similar to $A_{1,1}$, we can obtain
\[
|A_{1,2}| \lesssim
 \int_{\RR^6} \tilde K(\xi, \xi_*)
\hat f (\xi_* ) \hat g(\xi - \xi_* ) \bar{\hat h} (\xi ) d\xi d\xi_* \,,
\]
where $\tilde K(\xi,\xi_*)$ has the following upper bound
\begin{align}\label{later-use2}
\notag \tilde K(\xi,\xi_*) &\lesssim \int_0^{\min(\pi/2, \,\, \pi\la \xi_*\ra /(2|\xi|))} \theta^{1-\nu} d \theta
\frac{ \la \xi\ra^2}{\la \xi_*\ra^{3+\gamma +2}}\\
&\lesssim
\frac{1  }{\la \xi_*\ra^{3+\gamma}}\left\{
\left(
\frac{\la \xi \ra}{\la \xi_*\ra}\right)^2{\bf 1}_{\sqrt 2 |\xi| \leq \la \xi_* \ra}
+{\bf 1}_{ \sqrt 2 |\xi| \geq  \la \xi_* \ra \geq |\xi|/2}\right.\notag\\
&\qquad\qquad\qquad\qquad\qquad\qquad\left.+\left(
\frac{\la \xi \ra}{\la \xi_*\ra}\right)^{\nu}
{\bf 1}_{\la \xi_* \ra \leq |\xi|/2}\right\}\,,
\end{align}
from which we obtain the same inequality as \eqref{kernel-estimate} for $\tilde K(\xi,\xi_*)$.
Hence we obtain the desired bound for $A_{1,2}$. And this completes
the proof of the lemma.
\end{proof}

\begin{lem}\label{A-2}
For any $q \in [1,2]$,  let $\gamma > \max\{-3, -\nu -3 + 3/q\}$. Furthermore, assume that 
$m \le 3/2 -\nu/2$ if $q <2$. Then 
\[
 | A_{2,1} |+| A_{2,2}|
 \lesssim \|f\|_{L^q} \Vert f\Vert_{H^{\nu/2+ m}} \|h\|_{H^{\nu/2 -m}}\,.
\]
\end{lem}

\begin{proof}
In view of the definition of $A_{2,2}$, the fact that $|\xi| \sin(\theta/2) =|\xi^-| \geq \la \xi_*\ra/2$ and
$\theta \in [0,\pi/2]$ imply $\sqrt 2 |\xi| \geq \la \xi_*\ra$.
We can then
directly compute the spherical integral appearing inside $A_{2,2}$ together with $\hat \Phi_c$
as follows:
\begin{align}\label{A-2-2}
&\notag \left|\int  b \Big({\frac\xi{ | \xi |}} \cdot \sigma \Big)\hat \Phi_c(\xi_*) {\bf 1}_{ | \xi^- | \ge {\frac 1 2}\la \xi_*\ra } d\sigma \right|
 \lesssim  {\frac 1{\la \xi_* \ra^{3+\gamma }}} \frac{\la  \xi\ra^{\nu} }{\la \xi_*\ra^{\nu}}{\bf 1}_{\sqrt 2 |\xi| \geq \la \xi_* \ra} \\
 &\quad \lesssim
\frac{\la \xi\ra^{\nu/2-m}  \la \xi-\xi_*\ra^{\nu/2+m} }{\la \xi_*\ra^{3+\gamma +\nu}}\notag \\
&\quad + \frac{{\bf 1}_{ \la \xi -\xi *\ra \lesssim \la \xi_{*}  \ra}}{\la \xi_*\ra^{3+\gamma +\nu/2-m}
\la \xi-\xi_*\ra^{\nu/2+m} }   \la \xi\ra^{\nu/2-m}  \la \xi-\xi_*\ra^{\nu/2+m}\,,
\end{align}
which yields the desired estimate for $A_{2, 2}$.

We now turn to
$$
A_{2,1}=  \int b\,\, {\bf 1}_{ | \xi^- | \ge {\frac 1 2} \la  \xi_*\ra }\hat \Phi_c (\xi_* - \xi^-) \hat f (\xi_* ) \hat g(\xi - \xi_* ) \bar{\hat h} (\xi ) d\sigma d\xi d\xi_* .
$$
Firstly, note that we can  work on the set $| \xi_* \,\cdot\,\xi^-| \ge {\frac 1 2} | \xi^-|^2$. In fact, on the complementary of this
set, we have
 $| \xi_* \,\cdot\,\xi^-| \leq {\frac 1 2} | \xi^-|^2$ so that
 $|\xi_* -\xi^-| \gtrsim | \xi_*|$, and in this case,
we can proceed in the same way as for $A_{2,2}$. Therefore, it suffices to estimate
\begin{align*}
A_{2,1,p}&=  \int b\,\, {\bf 1}_{ | \xi^- | \ge {\frac 1 2} \la \xi_*\ra }{\bf 1}_{| \xi_* \,\cdot\,\xi^-| \ge {\frac 1 2} | \xi^-|^2}\hat \Phi_c (\xi_* - \xi^-) \hat f (\xi_* ) \hat g(\xi - \xi_* ) \overline{{\hat h} (\xi )} d\sigma d\xi d\xi_* \,.
\end{align*}
By
\[
{\bf 1}= {\bf 1}_{\la \xi_* \ra \geq |\xi|/2} {\bf 1}_{\la\xi-\xi_* \ra \leq \la \xi_* - \xi^- \ra}
+  {\bf 1}_{\la \xi_* \ra \geq |\xi|/2} {\bf 1}_{\la\xi-\xi_* \ra > \la \xi_* - \xi^-\ra}
+  {\bf 1}_{\la \xi_* \ra < |\xi|/2}
\]
we decompose
\begin{align*}
A_{2,1,p}
=
A_{2,1,p}^{(1)} + A_{2,1,p}^{(2)} +A_{2,1,p}^{(3)} \,.
\end{align*}
On the sets for
above integrals, we have $\la \xi_* -\xi^- \ra \lesssim \,
\la \xi_* \ra$, because $| \xi^- | \lesssim | \xi_*|$
that follows from  $| \xi^-|^2 \le 2 | \xi_* \cdot\xi ^-| \lesssim |\xi^-|\, | \xi_*|$.
Furthermore, on the sets for $A_{2,1,p}^{(1)}$ and $A_{2,1,p}^{(2)}$  we have $\la \xi \ra \sim \la \xi_* \ra$,
so that $\sup \Big(b\,\, {\bf 1}_{ | \xi^- | \ge {\frac 1 2} \la \xi_*\ra } {\bf 1}_{\la \xi_* \ra \geq |\xi|/2}\Big) \lesssim
{\bf 1}_{|\xi^- |\leq |\xi|/\sqrt 2}$ and
$\la \xi_* -\xi^- \ra \lesssim \,
\la \xi \ra$.
Hence we have, in view of $\nu/2-m \geq 0$,
\begin{align*}
|A_{2,1,p}^{(1)} | ^2  \lesssim& \int \frac{
|\hat \Phi_c (\xi_* - \xi^-) |^2 |\hat f (\xi_* )|^2  }
{\la \xi_* -\xi^- \ra^{\nu-2m}}\frac{ {\bf 1}_{\la\xi-\xi_* \ra \leq \la \xi_* - \xi^- \ra}}{\la\xi-\xi_* \ra^{\nu+2m}}d\xi d\xi_* d \sigma\\
& \times \int  |\la\xi-\xi_* \ra^{\nu/2+m}\hat g(\xi - \xi_* )|^2 |\la\xi \ra^{\nu/2-m}{\hat h} (\xi ) |^2 d\sigma d\xi d\xi_*\,.
\end{align*}
Note that  $3+2(\gamma +\nu) >0$ when $q =2$, together with  $6+2 \gamma+ \nu-2m >0$.
Then,  with $u = \xi_* -\xi^-$ we have
\begin{align*}
|A_{2,1,p}^{(1)} | ^2  \lesssim&  \int |\hat f(\xi_*)|^2 \left\{
\sup_{u}
{\la u \ra^{-( 6 +2\gamma+ \nu-2m)}}    \int \frac{ {\bf 1}_{\la \xi^+ -u \ra \leq \la u \ra}}{\la  \xi^+ -u \ra^{\nu+2m}}d\xi^+\right\} d\xi_*\\
&\qquad\qquad\qquad\times \|g\|^2_{H^{\nu/2+m}} \|h\|_{H^{\nu/2-m}}^2\\
\lesssim &
\|f\|^2_{L^2}  \|g\|^2_{H^{\nu/2+m}} \|h\|_{H^{\nu/2-m}}^2 \,,
\end{align*}
because $d\xi \sim d \xi^+$ on the support of ${\bf 1}_{|\xi^- |\leq |\xi|/\sqrt 2}$\,.
In the case where $q <2$, we use the condition $\nu/2+m \le 3/2$. If $q=1$ then 
$\gamma +\nu >0$, and by the change of variables $\xi_*-\xi^- \rightarrow u$ we have
\begin{align*}
|A_{2,1,p}^{(1)} | ^2  \lesssim&  \|\hat f\|^2_{L^\infty}  \|g\|^2_{H^{\nu/2+m}} \|h\|_{H^{\nu/2-m}}^2 \int
{\la u \ra^{-( 6 +2\gamma+ \nu-2m)}}    \int \frac{ {\bf 1}_{\la w \ra \leq \la u \ra}}{\la  w \ra^{\nu+2m}}dw  du \\
\lesssim &
\|f\|^2_{L^1} \|g\|^2_{H^{\nu/2+m}} \|h\|_{H^{\nu/2-m}}^2 \,.
\end{align*}
If $1 <q <2$, then $3+\gamma +\nu >3/q$ and 
  by the H\"older inequality and the change of variables $u = \xi_* -\xi^-$ we have
\begin{align*}
|A_{2,1,p}^{(1)} | ^2  \lesssim& \|g\|^2_{H^{\nu/2+m}} \|h\|_{H^{\nu/2-m}}^2 \left( \int |\hat f(\xi_*)|^{q/(q-1)} d\xi_*\right)^{2(q-1)/q} \\
\times & \left(\int \left(
{\la u \ra^{-( 6 +2\gamma+ \nu-2m)}}    \int \frac{ {\bf 1}_{\la \xi^+ -u \ra \leq \la u \ra}}{\la  \xi^+ -u \ra^{\nu+2m}}d\xi^+\right)^{q/(2-q)}
 du\right)^{2/q-1}\\
\lesssim &
\|f\|^{2}_{L^{q}}  \|g\|^2_{H^{\nu/2+m}} \|h\|_{H^{\nu/2-m}}^2 \,.
\end{align*}
As for $A_{2,1,p}^{(2)}$ we have by the Cauchy-Schwarz inequality
\begin{align*}
|A_{2,1,p}^{(2)} | ^2  \lesssim& \int \frac{
|\hat \Phi_c (\xi_* - \xi^-) | |\hat f (\xi_* )|  }
{\la \xi_* -\xi^- \ra^{\nu}}|{\la\xi-\xi_* \ra^{\nu/2+m}} \hat g(\xi -\xi_*)|^2 d \sigma d\xi d\xi_* \\
& \times \int \frac{
|\hat \Phi_c (\xi_* - \xi^-) | |\hat f (\xi_* )|  }
{\la \xi_* -\xi^- \ra^{\nu}} |\la\xi \ra^{\nu/2-m}{\hat h} (\xi ) |^2 d\sigma d\xi d\xi_*\,.
\end{align*}
Since it follows from $3 + \gamma + \nu <3/q$ and H\"older inequality that
\[
\int \frac{
|\hat \Phi_c (\xi_* - \xi^-) | |\hat f (\xi_* )|  }
{\la \xi_* -\xi^- \ra^{\nu}} d \xi_* d\sigma \lesssim 
\|f\|_{L^q}\,,
\]
we have the desired estimates for $A_{2,1,p}^{(2)}$.

On the set $A_{2,1,p}^{(3)}$ we have $\la \xi \ra \sim \la \xi - \xi_*\ra$. Hence
\begin{align*}
|A_{2,1,p}^{(3)} | ^2  \lesssim& \int b\, {\bf 1}_{ | \xi^- | \ge {\frac 1 2} \la \xi_*\ra }\frac{
|\hat \Phi_c (\xi_* - \xi^-) | |\hat f (\xi_* )|  }
{\la \xi\ra^{\nu}}|{\la\xi-\xi_* \ra^{\nu/2+m}} \hat g(\xi -\xi_*)|^2 d \sigma d\xi d\xi_* \\
& \times \int b\,\, {\bf 1}_{ | \xi^- | \ge {\frac 1 2} \la \xi_*\ra }\frac{
|\hat \Phi_c (\xi_* - \xi^-) | |\hat f (\xi_* )|  }
{\la \xi \ra^{\nu}} |\la\xi \ra^{\nu/2-m}{\hat h} (\xi ) |^2 d\sigma d\xi d\xi_*\,.
\end{align*}
We use the change of variables in $\xi_*$, $u= \xi_* -\xi^-$.
Note that $| \xi ^-| \ge {\frac 1 2} \la u +\xi^-\ra $ implies  $|\xi^-| \geq \la u\ra/\sqrt {10}$.
If $q=1$ then $\gamma + \nu>0$ and  we have
\begin{align*}
&\int b\,\, {\bf 1}_{ | \xi^- | \ge {\frac 1 2} \la \xi_*\ra }\frac{
|\hat \Phi_c (\xi_* - \xi^-) | |\hat f (\xi_* )|  }
{\la \xi\ra^{\nu}} d \sigma d\xi_*\\
&\lesssim \|\hat f\|_{L^\infty} \int \left (\frac{|\xi|}{\la u\ra}\right)^{\nu}
\la u \ra^{-(3+\gamma)} \la \xi\ra^{-\nu} du\lesssim \|f\|_{L^1}\,.
\end{align*}
On the other hand, if $q >1$ 
then this integral is upper bounded  by
\begin{align*}
&\int b\,\, {\bf 1}_{ | \xi^- | \ge {\frac 1 2} \la \xi_*\ra }\frac{
|\hat \Phi_c (\xi_* - \xi^-) |}{\la \xi \ra^{\nu/q} \la \xi_* -\xi^-\ra^{\nu/q'}}\frac{\la \xi_*\ra^{\nu/q'} |\hat f (\xi_* )|  }
{\la \xi \ra^{\nu/q'}}
 d \sigma d\xi_*
\\
&\leq \left( \int b\, {\bf 1}_{ | \xi^- | \ge {\frac 1 2} \la \xi_* \ra }\frac{
|\hat \Phi_c (\xi_* - \xi^-) |^q}{\la \xi \ra^{\nu} \la \xi_* -\xi^-\ra^{\nu q/q'}} d\sigma d \xi_*\right)^{1/q}\\
&\qquad\qquad\times
\left(\int b\,   {\bf 1}_{ | \xi^- | \ge {\frac 1 2} \la \xi_*\ra }\frac{\la \xi_*\ra^{\nu} |\hat f (\xi_* )|^{q'}}
{\la \xi \ra^{\nu}} d \sigma d\xi_* \right)^{1/q'}\\
&\leq \left( \int b\, {\bf 1}_{ | \xi^- |  \gtrsim \la u \ra }\frac{
|\hat \Phi_c (u) |^q}{\la \xi \ra^{\nu} \la u \ra^{\nu q/q'}} d\sigma d u\right)^{1/q}
\|\hat f\|_{L^{q'}} \\
&\lesssim \int \frac{du}{\la u \ra^{q (3+\gamma+\nu)}} \|f\|_{L^q}\,,
\end{align*}
where $1/q +1/q' =1$. 
Hence
we also obtain the desired estimates for $A_{2,1,p}^{(3)}$. The proof of the lemma is complete.
\end{proof}

Proposition \ref{upper-c} is then a direct consequence of Lemmas
\ref{A-1} and \ref{A-2}.

The following lemma is a variant of \cite[Lemma 4.5]{amuxy4-3}, 
where the roles of $g$ and $h$ are exchanged.
\begin{lem}\label{differ-Gam-Q}
Let $0<\nu<2$ and $\gamma > \max \{-3, -\nu -3/2\}$. Then for any $\alpha \ge 0$ and  any $\beta, \beta' \in \RR$ we have
\begin{align}\label{diff-G-Q}\nonumber 
&\Big|\Big(\Gamma ( f, g)\,  , \,h \Big)_{L^2} -\Big( Q(\mu^{1/2} f, g), h\Big)_{L^2}\Big|\\
&\qquad \lesssim 
\|\mu^{1/10} f\|_{L^2} ^{1/2}
\|g\|_{L^2_{(\nu+\gamma)/2-\beta}}
\Big(\cD(\mu^{1/4} { \,|f|} ,  \la v \ra^\beta h) \Big)^{1/2} \notag \\
&\notag \qquad \qquad + \|f\|_{L^2_{-\alpha} }
\|g\|_{L^2_{(\nu +\gamma)/2+\alpha - \beta'}}
\|h \|_{L^2_{(\nu+\gamma)/2+\beta' }} \\
& \qquad \qquad \qquad + \|\mu^{1/10} f\|_{L^2}\|\mu^{1/10} g\|_{L^2}\|\mu^{1/10}h\|_{H^{\nu/2}}\,,
\end{align}
where
\begin{align*}
\mathcal{D}(f,g)=\iiint_{\mathbb{R}^3_v\times\mathbb{R}^3_{v_*}\times\mathbb{S}^2} B(v-v_*,\sigma)f_*(g-g')^2dvdv_*d\sigma.
\end{align*}
\end{lem}
\begin{rema}
If $\gamma > -5/2$ then the last term of the right hand side of \eqref{diff-G-Q} disappears.
\end{rema}
Since it follows from Lemma 2.12 of \cite{amuxy4-3} and its proof that 
\[
\cD(\mu^{1/4} { \,|f|} ,  \la v \ra^\beta h) \lesssim \|\mu^{1/10} f\|_{L^2} |\!|\!|\la v\ra^\beta h|\!|\!|^2
\lesssim \|\mu^{1/10} f\|_{L^2} \| h \|^2_{H^{\nu/2}_{(\nu+\gamma)/2 + \beta} },
\]
we have the following;

\begin{cor}\label{convenient-123}
Let $0<\nu<2$ and $\gamma > \max \{-3, -\nu -3/2\}$. Then for any $\alpha \ge 0$ and  any $\beta, \beta' \in \RR$ we have
\begin{align}\label{diff-G-Q-09}\notag
&\Big|\Big(\Gamma ( f, g)\,  , \,h \Big)_{L^2} -\Big( Q(\mu^{1/2} f, g), h\Big)_{L^2}\Big|\\
& \quad \quad  \notag \lesssim
 \|\mu^{1/10} f\|_{L^2}\|g\|_{L^2_{(\gamma+\nu)/2 -\beta}}\|h\|_{H^{\nu/2}_{(\gamma+\nu)/2 +\beta}}\\
&\quad \qquad \qquad 
+ \|f\|_{L^2_{-\alpha} }
\|g\|_{L^2_{(\nu +\gamma)/2+\alpha - \beta'}}
\|h \|_{L^2_{(\nu+\gamma)/2+\beta' }} \,.
\end{align}
\end{cor}

\begin{proof}[Proof of Lemma \ref{differ-Gam-Q}]
We write
\begin{align*}
\Big(\Gamma( f, g)\, & , \,h \Big)_{L^2} -\Big( Q(\mu^{1/2} f, g), h\Big)_{L^2}
= \int B\, \Big( {\mu'_{*}}^{1/2} -\mu_{*}^{1/2} \Big) \big( f_{*} \big)
g h' d\sigma dv_*dv \\
& = 2
\int B\, \Big( (\mu'_{*})^{1/4} -\mu_{*}^{1/4} \Big) \big( \mu_{*}^{1/4}
f_{*} \big)
gh d\sigma dv_*dv \\
&+
\int B\, \Big( (\mu'_{*})^{1/4} -\mu_{*}^{1/4} \Big) ^2 f_{*}
g  h'
d\sigma dv_*dv  \\
&+
2
\int B\, \Big( (\mu'_{*})^{1/4} -\mu_{*}^{1/4} \Big) \big( \mu_{*}^{1/4}
f_{*} \big)
g(h'-h) d\sigma dv_*dv \\
&= D_1 + D_2 + D_3\,.
\end{align*}
Note that
\begin{align*}
&\Big((\mu'_{*})^{1/4} -\mu_{*}^{1/4} \Big)^2\le 2
\Big ((\mu'_{*})^{1/8} -\mu_{*}^{1/8} \Big)^2\Big(  (\mu'_{*})^{1/4} +\mu_{*}^{1/4} \Big)\\
&\lesssim        \min(|v-v_*|\theta,1)  \min(|v'-v'_*|\theta,1)  (\mu'_{*})^{1/4}
+  \Big(\min(|v'-v_*|\theta,1) \Big)^2\mu_{*}^{1/4} \,.
\end{align*}
By this decomposition we estimate 
\[
|D_2| \lesssim  D^{(1)}_2 + D^{(2)}_2
\]
Since $|v-v_*| \sim |v- v'_*|$ on $\mbox{supp}~b$, we have 
\begin{align*}
\la v_*\ra &\lesssim \la v-v_*\ra + \la v \ra \lesssim \la v-v'_*\ra \left(1 + \frac{\la v \ra}{\la v-v'_*\ra}
\right)\\
&\lesssim \la v-v'_*\ra \la v'_*\ra \lesssim \la v-v_*\ra \la v'_*\ra \enskip \mbox{on supp}~b\,,
\end{align*}
and hence 
\[
(\mu'_{*})^{1/4}| f_* | \lesssim \la {v'_*} \ra^\alpha (\mu'_{*})^{1/4}| \la v _* \ra^{-\alpha} f_* |\la v- v_*\ra^\alpha
\enskip \mbox{for any} \enskip \alpha \ge 0.
\]
Noting $\la v\ra^{\beta'} \lesssim \la v'-v'_* \ra^{\beta'} \la v'_* \ra^{|\beta'|}$ on supp$\,\, b$ for any
$\beta' \in \RR$, by the Cauchy-Schwarz inequality  we  have 
\begin{align*}
(D^{(1)}_2)^2 &\lesssim \int B |v-v_*|^{2\alpha}{\bf 1}_{|v-v_*|\ge 1} \min(|v-v_*|^2 \theta^2 ,1)
\Big( \frac{f_* g}{\la v _* \ra^{\alpha} \la v\ra^{\beta'}}
\Big)^2d\sigma dv_*dv \\
&\qquad \times  \int B |v-v_*|^{2\beta'}{\bf 1}_{|v-v_*|\ge 1} \min(|v-v_*|^2 \theta^2 ,1)
( \mu_{*}^{1/8} h)^2d\sigma dv_*dv \\
&+   \int B |v-v_*|^{2}{\bf 1}_{|v-v_*| < 1}\theta^2 
( \mu_{*}^{1/100} \mu^{1/100}h)^2d\sigma dv_*dv \\
&\qquad \qquad \times \int B |v-v_*|^{2}{\bf 1}_{|v-v_*| <1} \theta^2
( \mu_*^{1/100} f_* \mu^{1/100}g)^2d\sigma dv_*dv \\
& \lesssim
\|f\|_{L^2_{-\alpha}}^2\|g\|_{L^2_{(\nu+\gamma)/2  +\alpha-  \beta'}}^2 \|h\|^2_{L^2_{(\nu +\gamma)/2+\beta'}}\,,
\end{align*}
where we used the fact that $\la v'_* \ra \sim \la v' \ra \sim \la v \ra \sim \la v_* \ra$
on supp$\,\, b \cap {\bf 1}_{|v-v_*| <1}$, and $2 \gamma +4 >-3$. 
As for $D^{(2)}_2$, we have 
\[
(D_2^{(2)})^2 \lesssim 
\|\mu^{1/10} f\|_{L^2}^2\|g\|_{L^2_{(\nu+\gamma)/2  -  \beta'}}^2 \|h\|^2_{L^2_{(\nu +\gamma)/2+\beta'}}\,,
\]
thanks to the factor $\mu_*^{1/4}$ instead of ${\mu'_*}^{1/4}$.

By the Cauchy-Schwarz inequality we have for any $\beta \in \RR$
\begin{align*}
|D_3| &\lesssim  \Big(
\int B \Big( (\mu'_{*})^{1/4} -\mu_{*}^{1/4} \Big)^2 { |\mu_*^{1/4} f_{*}|}
{\big(\la v \ra^{-\beta} g\big)}^2 d\sigma dv_*dv \Big)^{1/2}\\
&\quad \times
 \Big( \int B    \, {\mu_{*}^{1/4}|f_*| } \la v \ra^{2\beta}
\Big (  h'- h
\Big)^2 d\sigma dv_*dv \Big)^{1/2}\\
& = \Big( \widetilde D_{3}(f,\la v \ra^{\beta}g) \Big)^{1/2}\Big(\cD_\beta(\mu^{1/4} f, h)\Big)^{1/2}\,.
\end{align*}
We have
\begin{align*}
\cD_\beta(\mu^{1/4} f, h)
&\le 2 \Big( \cD(\, { |\mu^{1/4} f|},  \la v \ra^{\beta} h) + \int B \, { |\mu_*^{1/4}f_*| } \Big(\la v \ra^{\beta} -
\la v' \ra^{\beta}\Big)^2h^2 dv dv_*d\sigma \Big)\\
&\lesssim \cD(\,{ | \mu^{1/4}f|},  \la v \ra^\beta h) + \|\mu^{1/8}f \|_{L^2} \|h\|^2_{L^2_{\beta+\gamma/2}}\,,
\end{align*}
because it follows from the same arguments in the proof of Lemma 2.12 in \cite{amuxy4-3} that
\begin{equation}\label{wet}
\Big|\la v \ra^\beta -
\la v' \ra^\beta\Big| \lesssim \sin \frac{\theta}{2} \Big(\la v \ra^\beta \la v_*\ra^{ 2|\beta|+1}{\bf 1}_{|v-v_*| > 1}+
\la v \ra^{\beta-1} |v-v_*| {\bf 1}_{|v-v_*| \leq 1}\Big)\,.
\end{equation}
The similar method as for $D_2^{(2)}$  shows 
\[
\widetilde D_3 \lesssim 
\|\mu^{1/10} f\|_{L^2}\|g\|_{L^2_{(\nu+\gamma)/2  -  \beta}}^2\,.
\]

To estimate $D_1$ we use the Taylor formula
\begin{align*}
(\mu'_{*})^{1/4} -\mu_{*}^{1/4}
&= \big(\nabla \mu^{1/4}\big) (v_*)\cdot (v'_*-v_*) \\
& + 
\int_0^1(1-\tau)  \big( \nabla^2 \mu^{1/4}\big)(v_*+ \tau(v'_*-v_*) ) (v'_*-v_*)^2  d\tau\,.
\end{align*}
By writing
\[
v'_* - v_* = \frac{|v-v_*|}{2}\{(\sigma\cdot \vk)\vk -\sigma\} + \frac{v-v_*}{2}(1-\vk \cdot \sigma),
\enskip \vk = \frac{v-v_*}{|v-v_*|},
\]
we see that the integral corresponding the first term on the integral of $D_1$ vanishes becasue of the 
symmetry on $\SS^2$. Therefore, we have 
\begin{align*}
|D_1| &\lesssim 
\int B  \min(|v-v_*|\theta^2, 1) \mu_*^{1/4}
f _*  gh  d\sigma dv_*dv \\
&\qquad \qquad  +  \int B    \min(|v-v_*|^2\theta^2, 1) \mu_*^{1/4}
f_*  gh  d\sigma dv_*dv \\
&\lesssim \int B  {\bf 1}_{|v-v_*| \ge 1} \min(|v-v_*|^2\theta^2, 1) \mu_*^{1/4}
f _*  gh  d\sigma dv_*dv \\
&\qquad \qquad  +  \int   {\bf 1}_{|v-v_*| < 1}  |v-v_*|^{\gamma+1+\nu/2} \mu_*^{1/10}
f_* (\mu^{1/10} g)\frac{\mu^{1/10}  h}{|v-v_*|^{\nu/2}}  dv_*dv \\
&\lesssim
\|\mu^{1/10} f\|_{L^2}\|g\|_{L^2_{(\nu+\gamma)/2  -  \beta'}} \|h\|_{L^2_{(\nu +\gamma)/2+\beta'}}\\
&\qquad \qquad+ \|\mu^{1/10} f\|_{L^2}\|\mu^{1/10} g\|_{L^2}\|\mu^{1/10}h\|_{H^{\nu/2}}
\end{align*}
in view of $2\gamma +\nu+2 >-3$.
\end{proof}

Since we may replace $B  {\bf 1}_{|v-v_*| \ge 1}$ by $B_{\bar c}$ in the 
proof of Lemma \ref{differ-Gam-Q},
under the condition $\gamma >-3$, it follows that for any $\beta, \beta' \in \RR$ and for any $\alpha \ge 0$ we have
\begin{align}\label{diff-G-Q-56789}\notag 
&\Big|\Big(\Gamma_{\bar c} ( f, g)\,  , \,h \Big)_{L^2} -\Big( Q_{\bar c}(\mu^{1/2} f, g), h\Big)_{L^2}\Big|\\
&\quad \lesssim \notag
\|\mu^{1/10} f\|_{L^2} 
\|\la v \ra^{-\beta} g\|_{L^2_{(\nu+\gamma)/2}}|\!|\!|\la v \ra^{\beta}
 h|\!|\!|\\
& \qquad \qquad + \|f\|_{L^2_{-\alpha} }
\|g\|_{L^2_{(\nu+\gamma)/2+\alpha - \beta'}}
\|h \|_{L^2_{(\nu+\gamma)/2+\beta'}}. 
\end{align}
By means of Lemma 3.2 of \cite{AMUXY3}, we get, for any $\alpha \ge 0$ and any $\beta' \in \RR$,
\begin{align}\label{cut-part-123}\notag
\Big|\Big(\Gamma_{\bar c} ( f, g)\,  , \,h \Big)_{L^2}\Big|
&\lesssim \|\mu^{1/10}f\|_{L^2} |\!|\!|g|\!|\!| |\!|\!|h |\!|\!|\\
& \quad + \|f\|_{L^2_{-\alpha} }
\|g\|_{L^2_{(\nu +\gamma)/2+\alpha - \beta'}}
\|h \|_{L^2_{(\nu+\gamma)/2+\beta' }}. 
\end{align}

Now we write
\begin{align*}
\Big(\Gamma_{c} ( f, g)\,  , \,h \Big)_{L^2}  &= \int B_c \mu_*^{1/2}(f'_* g' -f_* g) h dv dv_* d\sigma\\
& =  \int B_c \Big ( {\mu'}_*^{1/4} - {\mu_*}^{1/4} \Big )  {(\mu^{1/4} f)}_* g h dv dv_* d\sigma\\
&\quad +  \int B_c \Big ( \mu_*^{1/4} - {\mu'_*}^{1/4} \Big ) ^2 f_* g h' dv dv_* d\sigma\\
&\quad + \int B_c \Big ( {\mu'}_*^{1/4} - {\mu_*}^{1/4} \Big )  {(\mu^{1/4} f)}_* g (h'-h) dv dv_* d\sigma\\
&\quad + 
\int B_c \mu_*^{1/4}\Big ( (\mu^{1/4} f)'_* g' -(\mu^{1/4}f)_* g\Big ) h dv dv_* d\sigma\\
& = I_1 + I_2 + I_3 + I_4\,.
\end{align*}
The estimation for $I_4$ is just the same as in the arguments starting from the line 14 at the page 1021 of \cite{AMUXY3},
by replacing $f$ by $\mu^{1/4}f$. Then we have
$ |I_4| \lesssim \|\mu^{1/10}f\|_{L^2} |\!|\!|g|\!|\!| |\!|\!|h |\!|\!|$.
On the other hand, the estimations for $I_1$, $I_2$, $I_3$ are quite the same as those for $D_1$, $D_2$, $D_3$,
respectively, in the proof of Lemma \ref{differ-Gam-Q}.
 Therefore, $\Big(\Gamma_{c} ( f, g)\,  , \,h \Big)_{L^2}$ has the same bound as the right hand side of
\eqref{cut-part-123}.
Thus the proof of Lemma \ref{MS-general} is complete.

It remains to prove Proposition \ref{upper-recent}.  To this end we state a variant of \cite[Proposition 2.5]{amuxy4-3} where the roles of $g,h$  are exchanged.

\begin{prop}\label{IV-prop-2.5}
Let $0<\nu<2$ and $\gamma>\max\{-3, -\nu -3/2\}$.  For any $ \ell, \beta, \delta \in \RR$ and any small $\varepsilon >0$
$$
\Big|\Big(\la v \ra^\ell \,Q_c(f, g)-Q_c
(f, \la v \ra^\ell g),\, h\Big)\Big|\lesssim
\|f\|_{L^2_{\ell-1-\beta-\delta}}
\|g\|_{L^2_\beta}\Vert h\Vert
_{H^{(\nu-1+\epsilon)^+}_\delta}.
$$
\end{prop}
\begin{proof}
It suffices to write 
\begin{align*}
\Big(\la v \ra^\ell \,Q_c(f, g)- Q_c
(f, \la v \ra^\ell g),\, h\Big)
&=  \int B_c \Big(\la v' \ra^\ell- \la v \ra^\ell
\Big)  f_* gh dv dv_* d\sigma \\
&+ \int B_c\Big(\la v' \ra^\ell- \la v \ra^\ell
\Big) f_* g \Big(h'-h\Big) dv dv_* d\sigma \notag \\
&= J_1 + J_2\,. \notag
\end{align*}
The estimation for $J_2$ is  quite the  same as in the proof of \cite[Proposition 2.5]{amuxy4-3}
As for $J_1$, we use the Taylor expansion, with $ v'_\tau = v+  \tau(v'-v)$, 
\begin{align}\label{one-use}
\la v' \ra^\ell - \la v \ra^\ell = \nabla \Big( \la v \ra^\ell\Big)\cdot (v'-v)
+  \int_0^1 (1-\tau)\nabla^2 \Big( \la v'_\tau \ra^\ell \Big)d\tau(v'-v)^2 \,.
\end{align}
Since we have 
\[ 
v' - v= \frac{|v-v_*|}{2} (\sigma - (\vk \cdot \sigma) \vk ) +
\frac{v-v_*}{2}(1 - \vk \cdot \sigma) 
\]
and the integral corresponding to the first term vanishes because of  the symmetry on $\SS^2$, 
it follows that
\begin{align*}
\Big|J_1\Big|&\lesssim  \left|\int b(\cos \theta)\Phi_c \sin^2( \theta/2)|v-v_*|\Big(
|\nabla \Big(\la v \ra^\ell \Big)|  \right. \\
& \left.\qquad \qquad \qquad \qquad \qquad \qquad \qquad \qquad + \int_0^1|\nabla^2 \Big( \la v'_\tau\ra^\ell\Big)|d\tau \Big)| f_* gh |dv dv_* d\sigma \right|\\
&\lesssim   \int_{|v-v_*|\lesssim 1} |v-v_*|^{\gamma+1+(\nu -1 +\varepsilon)^+} 
|\la v_* \ra^{\ell-1-\beta-\delta } f_*|\\
&\qquad \qquad \qquad \times 
\, |(\la v \ra^{\beta} g)| \frac{|(\la v \ra^{\delta}h)| }{ |v-v_*|^{(\nu -1 +\varepsilon)^+}}
 dv dv_* \,\\
&\lesssim \|f\|_{L^2_{\ell-1-\beta-\delta} } \|g\|^2 _{L^2_\beta}\|h\|^2 _{H^{(\nu -1 +\varepsilon)^+}_\delta}\,,
\end{align*}
which, together with the estimate for $J_2$, gives the desired estimate.

\end{proof}

\begin{proof}[Proof of Proposition \ref{upper-recent}] Here we only prove the case $m \in [-\nu/2, 0]$
because the other case was essentially given in \cite[Lemma 4.7]{amuxy4-3}. Note
\begin{align*}
\Big(Q_c(f,g),h\Big) &= \Big(Q_c(f, \la v \ra^{\ell} g), \la v \ra^{-\ell} h\Big) \\
&+ \Big( \la v \ra^{\ell}Q_c(f, \la v \ra^{\ell} g)  - Q_c(f, \la v \ra^{\ell} g)    , \la v \ra^{-\ell}h \Big) .
\end{align*}
It follows from Propositions \ref{upper-c} and \ref{IV-prop-2.5} with $\beta =\ell$, $\delta =0$ that 
\begin{align*}
\Big|\Big(Q_c(f,g),h\Big)\Big|  \lesssim \|f\|_{L^2}\|g \|_{H^{\nu/2+m}_\ell} \|h\|_{H^{\nu/2-m}_{-\ell}}\,,
\end{align*}
where we have used $m \le 0$.  By means of \cite[Theorem 2.1]{AMUXY2010} (see also (2.1) of 
\cite{amuxy4-3}) we have for any $m, \ell\in\RR$,
$$
\Big|\Big(Q_{\bar c} (f,  g),\, h \Big)\Big|  \lesssim \Vert f\Vert_{L^1_{\ell^++(\gamma+\nu)^+}} \Vert g\Vert_{H^{\nu/2 +m}_{(\ell+\gamma+\nu)^+}} \|h\|_{H^{\nu/2-m}_{-\ell}}\,.
$$
Above two estimates and Corollary \ref{convenient-123} concludes \eqref{different-g-h}
when $-\nu/2 \le m \le 0$. 
\end{proof}

\subsection{Commutator estimates with moments}\label{ap-3}

Let $W_\delta(v) = \la \delta v \ra^{-N}$ for $0< \delta <1$ and $N \ge 1$. 

\begin{align*}
&\Big(\mathcal{L}_1(g), W_\delta^2 g\Big)_{L^2(\RR^3_v) }-      \Big(\mathcal{L}_1(W_\delta g), W_\delta g \Big)_{L^2(\RR^3_v)}\\
&= \int B \mu_*^{1/2}  {\mu'}_*^{1/2}(W_\delta -W'_\delta)g' W_\delta g dv dv_* d\sigma\\
&=\frac{1}{2} \int B \mu_*^{1/2}  {\mu'}_*^{1/2}(W_\delta -W'_\delta)^2g' g dv dv_* d\sigma\\
&\le \frac{1}{2} \int B \mu_* (W'_\delta -W_\delta)^2g^2 dv dv_* d\sigma\,,
\end{align*}
where we used the change of variables $(v', v'_*, \sigma) \rightarrow (v, v_*, \vk)$ and the Cauchy-Schwarz inequality.
Note that we have 
\begin{align}\label{W-est}\notag
|W_\delta(v') -W_\delta(v) | &\le \int_0^1 |\nabla W_\delta(v'_\tau)|d\tau |v'-v|, \enskip \enskip v'_\tau = v+  \tau(v'-v) \\ \notag
&\lesssim \delta \int_0^1 W_\delta(v'_\tau)\la \delta v'_\tau\ra^{-1} d\tau |v-v_*|\sin \frac{\theta}{2}\\
\notag
&\lesssim \delta^{\nu/2} |v-v_*|^{\nu/2} \theta W_\delta(v)\la v_* \ra^{N+2-\nu/2}\\ 
&\lesssim \delta^{\nu/2} \la v\ra^{\nu/2} \theta W_\delta(v)\la v_* \ra^{N+2}\,,
\end{align}
because $\la \delta v'_\tau\ra^{-1}  \lesssim \la \delta v \ra^{-1} \la v_*\ra$ and $\delta |v-v_*| \le \la \delta v \ra \la v_* \ra$.
Since it follows from Lemma 2.5 of  \cite{AMUXY2} that
$$\int |v-v_*|^{\gamma +\nu} \mu_*^{1/2} d v_* \lesssim \la v \ra^{\gamma+\nu}, 
$$ 
we have
\begin{align}\label{linear-comm-weight}
\left|\Big(\mathcal{L}_1(g), W_\delta^2 g\Big)_{L^2(\RR^3_v) }-      \Big(\mathcal{L}_1(W_\delta g), W_\delta g \Big)_{L^2(\RR^3_v)}\right|
\lesssim \delta^{\nu} \|g\|^2_{L^2_{(\nu+\gamma)/2}}\,.
\end{align}

\vskip1cm

\begin{prop}\label{weight-commutator}
Let $\gamma > \max \{-3, -\nu -3/2\}$. Then we have 
\begin{align*}
\left| \Big(\Gamma(f,g), \right. &\left. W_\delta h\Big)_{L^2(\RR^3)}-
\Big(\Gamma(f, W_\delta g),  h\Big)_{L^2(\RR^3)} \right| \\
&\lesssim \delta^{\nu/2} \|f\|_{L^2}  \|W_\delta g\|_{L^2_{(\nu+\gamma)/2}} 
\Big(\|h\|_{L^2_{(\nu+\gamma)/2}} + \|h\|_{H^{\nu'/2}_{-N_1}}\Big)\\
& \qquad + \delta^{\nu/2} \|f\|_{L^2}^{1/2}  \|W_\delta g\|_{L^2_{(\nu+\gamma)/2}} 
\cD(\mu^{1/2}f, h)^{1/2}\,,
\end{align*}
for any $N_1 >0$ and any $0 \le \nu' \le \nu$ satisfying 
$\gamma + \nu' >-3/2$.
\end{prop}

\begin{proof}
Note that 
\begin{align*}
\Big(\Gamma(f,g), W_\delta h\Big)_{L^2(\RR^3)}-&
\Big(\Gamma(f, W_\delta g),  h\Big)_{L^2(\RR^3)}\\
&= \int B {\mu'}_*^{1/2}  (W'_\delta -W_\delta){f}_* g h' dv dv_* d\sigma\\
&=  \int B \mu_*^{1/2}  (W'_\delta -W_\delta){f}_* gh dv dv_* d\sigma\\
&+ \int B \Big({\mu'}_*^{1/2} - \mu_*^{1/2}\Big) f_* (W'_\delta -W_\delta)gh' dv dv_* d\sigma\\
&+\int B  \mu_*^{1/2} f_* (W'_\delta -W_\delta)g (h'-h) dv dv_* d\sigma\\
& =  A_1 + A_2 + A_3. 
\end{align*}

Note that
\begin{align*}
&|v-v_*|^\gamma \lesssim \la v\ra^\gamma \la v_* \ra^{|\gamma|}
\enskip \mbox{or} \enskip \la v\ra^\gamma \la v'_* \ra^{|\gamma|}
\enskip \mbox{if } \enskip |v-v_*| \ge 1\,,\\
& \la v \ra \sim \la v_* \ra \sim \la v' \ra \sim \la v_*'\ra 
\enskip \mbox{if } \enskip |v-v_*| < 1\,.
\end{align*}
We divide 
\begin{align*}
A_2 = \int_{|v-v_*| < 1} \cdots dvdv_*d\sigma + \int_{|v-v_*| \ge 1} \cdots dvdv_*d\sigma= A_{2,1} + A_{2,2}.
\end{align*}
Using  that 
\begin{align*}
|{\mu'_*}^{1/2} - \mu_*^{1/2}| \lesssim \min(|v-v_*|\theta, 1)\mu_*^{1/4}
+ \min(|v-v'_*|\theta, 1){\mu'}_*^{1/4}
\end{align*}
we estimate  as follows:  $|A_{2,1}| \lesssim |A_{2,1}^{*}| + |A^{*'}_{2,1}|$.  Then  for any $N_1 >0$  we have
\begin{align*}
|A_{2,1}^{*'}| &\lesssim \delta^{\nu/2} \int_{|v-v_*| <1}
|v-v_*|^{\gamma + 1 +\nu/2} b \theta^2\left| f_* \frac{W_\delta g} {\la v \ra^{N_1} }
\frac{h'}{\la v' \ra^{N_1}}\right|
dv dv_* d\sigma\\
&\lesssim 
\delta^{\nu/2} \left(\int \Big(\int_{|v-v_*| <1} |v-v_*|^{2\gamma + 2 +\nu} dv_* \Big)\left|\frac{W_\delta g} {\la v \ra^{N_1} }\right|^2dv\right)^{1/2} \\
&\qquad \qquad \qquad 
 \times \left( \int b \theta^2 |f_*|^2 \left|\frac{h'}{\la v' \ra^{N_1}}\right|^2 dv dv_* d\sigma\right ) ^{1/2}\\
&\lesssim \delta^{\nu/2} \|f\|_{L^2} \|W_\delta g \|_{L^2_{-N_1}} \|h\|_{L^2_{-N_1}}\,.
\end{align*}
Similarly we have the same upper bound for $A_{2,1}^{*}$.
As for $A_{2,2}$,  it follows from the Cauchy -Schwarz inequality that 
\begin{align*}
A_{2,2}^2 &\le 2\int_{|v-v_*|\ge 1} B f_*^2 (W_\delta - W_\delta')^2g^2 ({\mu'_*}^{1/2} + \mu^{1/2}_*)dv dv_* d\sigma \\
& \qquad \qquad \qquad \times \int B h^2 \Big ({\mu_*}^{1/4} - {\mu'}^{1/4}_*\Big)^2 dv dv_* d\sigma\\
&\lesssim \delta^{\nu}\int \la v\ra^{\gamma+\nu} b \theta^2
|f_* W_\delta g |^2dv dv_*d \sigma \|h\|^2_{L^2_{(\nu+\gamma)/2}}\\
&\lesssim \delta^{\nu}\|f\|^2_{L^2} \|W_\delta g\|^2_{L^2_{(\nu+\gamma)/2}} \|h\|^2_{L^2_{(\nu+\gamma)/2}}\,,
\end{align*}
where we have used the fact that, if  $|v-v_*| \ge 1$,
\[
|v-v_*|^{\gamma} (W_\delta - W_\delta')^2 \lesssim \delta^{\nu} \la v\ra^{\gamma+\nu} \theta^2 W_\delta^2
\min \{ \la v_* \ra^{|\gamma|+2N+4}\,,\, \la v'_* \ra^{|\gamma|+2N+4}\},
\]
because of \eqref{W-est} and it with $v_*$ replaced by $v_*'$. 
By means of the Cauchy-Schwarz inequality, we have 
\begin{align*}
A_3^2 & \le \int B  \mu_*^{1/2} f_* (W'_\delta -W_\delta)^2g^2  dv dv_* d\sigma\times 
\cD(\mu^{1/2}| f|, h).
\end{align*}
The first factor on the right hand side is estimated above from $C \delta^{\nu}$ times 
\begin{align*}
&\|\mu^{1/4} f \|_{L^1}\|W_\delta g\|^2_{L^2_{(\nu+\gamma)/2}}\\
&+\int \Big( \int_{|v-v_*|<1} |v-v_*|^{2(\gamma+\nu)}dv_*\Big)^{1/2} \|\mu^{1/4} f\|_{L^2}\|W_\delta g\|^2_{L^2_{-N_1}}
\end{align*}
because it follows again from \eqref{W-est} that
\[
|v-v_*|^{\gamma} (W_\delta - W_\delta')^2 \mu^{1/4}_* 
\lesssim \delta^{\nu}  \theta^2 W_\delta^2
\Big\{ \la v\ra^{\gamma+\nu} {\bf 1}_{|v-v_*|\ge 1} + \frac{|v-v_*|^{\gamma+\nu}}{\la v \ra^{N_1}}
{\bf 1}_{|v-v_*|<1} \Big\}.
\]
Therefore we also have the desired bound for $A_3$.  

In order to estimate $A_1$, we use the following Taylor expansion of the second order; with
$ v'_\tau = v+  \tau(v'-v)$,  
\begin{align}\label{W-est-sec}
W_\delta(v') -W_\delta(v)  &= \nabla W_\delta (v)\cdot (v'-v)\notag\\
& \qquad \qquad  + 
\int_0^1 (1-\tau)\nabla^2 W_\delta(v'_\tau) d\tau (v'-v)^2.
\end{align}
Similar as in \eqref{one-use},
we can estimate the factor $W'_\delta -W_\delta$ by
\begin{align}\label{second-w}
&\notag |\nabla W_\delta(v)| |v-v_*| \theta^2  + |\nabla^2 W_\delta(v'_\tau) | |v-v_*|^2\theta^2\\
&\notag \lesssim \delta W_\delta(v) \la \delta v \ra^{-1} |v-v_*| \theta^2 + \delta^2 W_\delta(v'_\tau)
\la \delta v'_\tau \ra^{-2} |v-v_*|^2 \theta^2\\
&\lesssim \delta^{\nu/2} W_\delta(v) |v-v_*|^{\nu} \theta^2 \la v_*\ra^{1-\nu/2} + \delta^{\nu}W_\delta(v)
|v-v_*|^{\nu}\theta^2  \la v_*\ra^{N+4 -\nu} .
\end{align}
Consequently, for any $0 \le \nu' \le \nu$ satisfying $\gamma + \nu' >-3/2$,  we have
\begin{align*}
|A_1|&\lesssim \delta^{\nu/2} \iint \la v \ra^{\gamma+\nu} {\bf 1}_{|v-v_*| \ge 1} \mu_*^{1/4}| f_* W_\delta g h |dv dv_*\\
& + \delta^{\nu/2} \iint |v-v_*|^{\gamma+(\nu+\nu')/2} {\bf 1}_{|v-v_*| < 1}\left| f_* 
\frac{W_\delta g }{\la v \ra^{N_1}} \frac{h}{|v-v_*|^{\nu'/2} \la v \ra^{N_1}} \right |dv dv_*\\
& \lesssim \delta^{\nu/2} \|\mu^{1/4}f\|_{L^1} \|W_\delta g\|_{L^2_{(\nu+\gamma)/2}}\|h\|_{L^2_{(\nu+\gamma)/2}}\\
& + \delta^{\nu/2} \Big(\int |f_*|^2 \int \left|\frac{h} {|v-v_*|^{\nu'/2} \la v \ra^{N_1}} \right |^2dv\Big) dv_*\Big)^{1/2}\\
&\quad \qquad \times 
\Big(\int \left|\frac{W_\delta g }{\la v \ra^{N_1}}\right|^2 \int_{|v-v_*|<1}
  |v-v_*|^{2\gamma+\nu + \nu'} dv_* \Big) dv\Big)^{1/2}\\
&\lesssim \delta^{\nu/2} \|f\|_{L^2}  \Big(\|W_\delta g\|_{L^2_{(\nu+\gamma)/2}}\|h\|_{L^2_{(\nu+\gamma)/2}}
+ \|W_\delta g\|_{L^2_{-N_1}} \|h\|_{H^{\nu'/2}_{-N_1}}\Big).
\end{align*}

\end{proof}
The following corollary is a variant of 
Lemma 4.9 in \cite{amuxy4-3}, which is used to prove the non-negativity of solutions. 
\begin{cor}\label{cor-weight-commutator}
Let $\gamma > \max \{-3, -\nu -3/2\}$ and let $\varphi(v,x) = (1+|v|^2+|x|^2)^{\alpha/2}$ for $\alpha  >3/2$. 
If we put $W_{\varphi} = \varphi(v,x)^{-1}$, then we have 
\begin{align*}
\left| \Big(\Gamma(f,g), \right. &\left. W_{\varphi} h\Big)_{L^2(\RR^3)}-
\Big(\Gamma(f, W_\varphi g),  h\Big)_{L^2(\RR^3)} \right| \\
&\lesssim \|f\|_{L^2}  \|W_\varphi g\|_{L^2_{\gamma/2}} 
\Big(\|h\|_{L^2_{(\nu+\gamma)/2}} + \|h\|_{H^{\nu'/2}_{-N_1}}\Big)\\
& \qquad +  \|f\|_{L^2}^{1/2}  \|W_\varphi g\|_{L^2_{\gamma/2}} 
\cD(\mu^{1/2}f, h)^{1/2}\,,
\end{align*}
for any $N_1 >0$ and any $0 \le \nu' \le \nu$ satisfying 
$\gamma + \nu' >-3/2$.
\end{cor}
\begin{proof}
Instead of \eqref{W-est} and \eqref{second-w}, respectively, 
it suffices to use 
\begin{align*}
|W_\varphi(v') -W_\varphi(v) | &\le \int_0^1 |\nabla W_\varphi(v'_\tau)|d\tau |v'-v|, \enskip \enskip v'_\tau = v+  \tau(v'-v) \\ \notag
&\lesssim \int_0^1 W_\varphi(v'_\tau)\la v'_\tau\ra^{-1} d\tau |v-v_*|\sin \frac{\theta}{2}\\
\notag
&\lesssim  \theta W_\varphi(v)\la v_* \ra^{\alpha+2}\Big( {\bf 1}_{|v-v_*|\ge1} + |v-v_*|{\bf1}_{|v-v_*|<1}\Big)
\end{align*}
and 
\begin{align*}
&|\nabla W_\varphi(v)| |v-v_*| \theta^2  + |\nabla^2 W_\varphi(v'_\tau) | |v-v_*|^2\theta^2\\
&\notag \lesssim W_\varphi(v) \la  v \ra^{-1} |v-v_*| \theta^2 +  W_\varphi(v'_\tau)
\la v'_\tau \ra^{-2} |v-v_*|^2 \theta^2\\
&\notag \lesssim \theta^2 W_\varphi(v) \la v_*\ra^{\alpha +4}\Big(
 {\bf 1}_{|v-v_*|\ge1}  +  |v-v_*|{\bf1}_{|v-v_*|<1}\Big).
\end{align*}
\end{proof}

\subsection{Commutator with $v$ derivative mollifier}\label{ap-4}
Since the kinetic factor of the collision cross section is singular, we must confine ourselves 
to a special order of  the mollifier. 
Let $1 \ge N_0 \ge  \nu/2$ and 
let $M^\delta(\xi) = \displaystyle  \frac{1}{(1+ \delta \la \xi \ra)^{N_0} }$ for $0 < \delta \le 1$.
\begin{prop}\label{IV-coro-2.15}
Assume that  $0 < \nu <2$ and $\gamma>\max\{-3, -\frac 32 -\nu\} $.
Then  for any $\nu' >0$ satisfying $\nu-1 \le \nu' <\nu$ and $\gamma +\nu' >-3/2$ we have
\begin{align*}
\Big 
| \Big (M^\delta(D_v)\,& Q_c (f,  g) - Q_c ( f,  M^\delta(D_v) \, g)  ,\, h \Big ) \Big |\\
&\qquad \qquad 
\lesssim  \| f\|_{L^2} \|M^\delta (D_v) g\Vert_{H^{\nu'/2}}  \, \Vert h\Vert_{H^{\nu'/2}}\,.
\end{align*}

\end{prop}

\begin{proof}
We shall follow similar arguments {}of Proposition 3.4 from \cite{AMUXY-Kyoto}. By using the formula {}from
the Appendix of \cite{advw}, we have
\begin{align*}
( Q_c(f, g), h ) =& \iiint_{ \RR^3\times\RR^3\times\SS^2} b \Big({\frac{\xi}{ | \xi |}} \cdot \sigma \Big) [ \hat\Phi_c (\xi_* - \xi^- ) - \hat \Phi_c (\xi_* ) ] \\
& \qquad \qquad  \times \hat f (\xi_* ) \hat g(\xi - \xi_* ) \overline{{\hat h} (\xi )} d\xi d\xi_*d\sigma \,,
\end{align*}
where $\xi^-=\frac 1 2 (\xi-|\xi|\sigma)$.
Therefore
\begin{align*}
 \Big (M^\delta(D) \, Q_c (f,  g)& -  Q_c (f,  M^\delta(D)\, g)  , h \Big ) \\
= &\iiint b \Big({\frac{\xi}{ | \xi |}} \cdot \sigma \Big) [ \hat\Phi_c (\xi_* - \xi^- ) - \hat \Phi_c (\xi_* ) ] \\
&\quad \times
\Big(M^\delta (\xi) - M^\delta (\xi-\xi_*)\Big)
\hat f (\xi_* ) \hat g(\xi - \xi_* ) \overline{{\hat h} (\xi )} d\xi d\xi_*d\sigma \\
= & \iiint_{ | \xi^- | \leq { \frac 1 2} \la \xi_*\ra }  \cdots\,\, d\xi d\xi_*d\sigma
+ \iiint_{ | \xi^- | \geq {\frac 1 2} \la \xi_*\ra } \cdots\,\, d\xi d\xi_*d\sigma \,\\
=& \cA_1(f,g,h)  +  \cA_2(f,g,h) \,\,.
\end{align*}
Then, we write $\cA_2(f,g,h)$ as
\begin{align*}
\cA_2 &=  \iiint b \Big(\frac{\xi}{ | \xi |} \cdot \sigma \Big) {\bf 1}_{ | \xi^- | \ge {\frac1 2}\la \xi_*\ra }
\hat\Phi_c (\xi_* - \xi^- ) \cdots 
d\xi d\xi_*d\sigma \\
&\quad - \iiint b \Big(\frac{\xi}{ | \xi |} \cdot
 \sigma \Big){\bf 1}_{ | \xi^- | \ge {\frac12}\la \xi_*\ra } \hat \Phi_c (\xi_* )
\cdots 
 d\xi d\xi_*d\sigma \\
&= \cA_{2,1}(f,g,h) - \cA_{2,2}(f,g,h)\,.
\end{align*}
On the other hand,  for $\cA_1$ we use the Taylor expansion of $\hat \Phi_c$ of
order $2$ to have
$$
\cA_1 = \cA_{1,1} (f,g,h) +\cA_{1,2} (f,g,h),
$$
where
\begin{align*}
\cA_{1,1} &= \iiint b\,\, \xi^-\cdot (\nabla\hat\Phi_c)( \xi_*)
{\bf 1}_{ | \xi^- | \leq \frac{1}{2} \la \xi_*\ra }
\Big(M^\delta (\xi) - M^\delta (\xi-\xi_*)\Big)\\
& \qquad \times
\hat f (\xi_* ) \hat g(\xi - \xi_* ) \bar{\hat h} (\xi ) d\xi d\xi_*d\sigma,
\end{align*}
and $\cA_{1,2} (f,g,h)$ is the remaining term corresponding to the second order term in the Taylor expansion of $\hat\Phi_c$.

We first consider $\cA_{1,1}$.
By writing
\[
\xi^- = \frac{|\xi|}{2}\left(\Big(\frac{\xi}{|\xi|}\cdot \sigma\Big)\frac{\xi }{|\xi|}-\sigma\right)
+ \left(1- \Big(\frac{\xi}{|\xi|}\cdot \sigma\Big)\right)\frac{\xi}{2},
\]
we see that the integral corresponding to the first term on the right hand side vanishes because of the symmetry
on $\SS^2$.
Hence, we have
\[
\cA_{1,1}= \iint_{\RR^6} K(\xi, \xi_*) \Big(M^\delta (\xi) - M^\delta (\xi-\xi_*)\Big)
\hat f (\xi_* ) \hat g(\xi - \xi_* ) \bar{\hat h} (\xi ) d\xi d\xi_* \,,
\]
where
\[
K(\xi,\xi_*) = \int_{\SS^2}
 b \Big({ \frac{\xi}{ | \xi |}} \cdot \sigma \Big)
\left(1- \Big(\frac{\xi}{|\xi|}\cdot \sigma\Big)\right)\frac{\xi}{2}\cdot
(\nabla\hat\Phi_c)( \xi_*)
{\bf 1}_{ | \xi^- | \leq {\frac1 2} \la \xi_*\ra } d \sigma \,.
\]
Note that $| \nabla \hat \Phi_c (\xi_*) | \lesssim \frac{1}{\la
\xi_*\ra^{3+\gamma +1}}$, {}from the Appendix of \cite{AMUXY2}. If $\sqrt 2
|\xi| \leq \la \xi_* \ra$, then $\sin(\theta/2) $ $| \xi|= |\xi^-| \leq \la \xi_* \ra/2$
because $0 \leq \theta \leq \pi/2$, and we have
\begin{align*}
|K(\xi,\xi_*)| &\lesssim \int_0^{\pi/2} \theta^{1-\nu} d \theta\frac{ \la \xi\ra}{\la \xi_*\ra^{3+\gamma +1}}
\lesssim \frac{1  }{\la \xi_*\ra^{3+\gamma}}\left(
\frac{\la \xi \ra}{\la \xi_*\ra}\right) \,.
\end{align*}
On the other hand, if $\sqrt 2 |\xi| \geq \la \xi_* \ra$, then
\begin{align*}|K(\xi,\xi_*)| &\lesssim \int_0^{\pi\la \xi_*\ra /(2|\xi|)} \theta^{1-\nu} d \theta\frac{ \la \xi\ra}{\la \xi_*\ra^{3+\gamma +1}}
\lesssim \frac{1  }{\la \xi_*\ra^{3+\gamma}}\left(
\frac{\la \xi \ra}{\la \xi_*\ra}\right)^{\nu-1}\,.
\end{align*}
Hence we obtain
\begin{align}\label{later-use3}
|K(\xi,\xi_*)| &\lesssim \frac{1  }{\la \xi_*\ra^{3+\gamma}}\left\{
\left( \frac{\la \xi \ra}{\la \xi_*\ra}\right){\bf 1}_{ \la \xi_*
\ra \geq
\sqrt 2 |\xi| } \right.\notag\\
&\qquad\left.+{\bf 1}_{ \sqrt 2 |\xi| \geq  \la \xi_* \ra \geq
|\xi|/2} + \left( \frac{\la \xi \ra}{\la \xi_*\ra}\right)^{\nu-1}
{\bf 1}_{ |\xi|/2 \ge \la \xi_* \ra }\right\}\,. 
\end{align}
Similarly to $\cA_{1,1}$,
we  can also write
\[
\cA_{1,2}= \iint_{\RR^6} \tilde K(\xi, \xi_*) \Big(M^\delta (\xi) - M^\delta (\xi-\xi_*)\Big)
\hat f (\xi_* ) \hat g(\xi - \xi_* ) \bar{\hat h} (\xi ) d\xi d\xi_* \,,
\]
where
\[
\tilde K(\xi,\xi_*) = \int_{\SS^2}
 b \Big({\frac{\xi}{ | \xi |}} \cdot \sigma \Big)
\int^1_0(1-\tau)  (\nabla^2\hat \Phi_c) (\xi_* -\tau\xi^- ) \cdot\xi^- \cdot\xi^-
{\bf 1}_{ | \xi^- | \leq {\frac 12} \la \xi_*\ra } d\tau  d \sigma\,.
\]
Again {}from the Appendix of \cite{AMUXY2}, we have
$$
| (\nabla^2\hat \Phi_c) (\xi_* -\tau\xi^- ) | \lesssim {\frac{ 1}{\la  \xi_* -\tau \xi^-\ra^{3+\gamma +2}}}
\lesssim
 {\frac{1}{\la \xi_*\ra^{3+\gamma +2}}},
$$
because $|\xi^-| \leq \la \xi_*\ra/2$, which leads to
\begin{align}\label{later-use4}\notag
|\tilde K(\xi,\xi_*)| &\lesssim \frac{1  }{\la
\xi_*\ra^{3+\gamma}}\left\{ \left( \frac{\la \xi \ra}{\la
\xi_*\ra}\right)^2 {\bf 1}_{ \la \xi_* \ra \geq \sqrt 2 |\xi|
}\right.\\
&\qquad\left. +{\bf 1}_{ \sqrt 2 |\xi| \geq  \la \xi_* \ra \geq
|\xi|/2} + \left( \frac{\la \xi \ra}{\la \xi_*\ra}\right)^{\nu} {\bf
1}_{ |\xi|/2 \ge \la \xi_* \ra }\right\}\,. 
\end{align}
We employ (3.4)
of Lemma 3.1 in \cite{AMUXY-Kyoto} with $p=1$ and $\lambda =0$, that is, 
\begin{align}\label{simple-form}
&\left|M^\delta (\xi) - M^\delta (\xi-\xi_*)\right|\leq C M^\delta (\xi-\xi_*)\left\{ \Big(
\frac{\la \xi_*\ra}{\la \xi \ra }\Big) {\bf 1}_{\la \xi_* \ra \ge \sqrt 2 |\xi|} \right.\notag\\
&+ \left( {M^\delta (\xi_*)\big(1 + \delta \la \xi - \xi_*\ra\big)^{N_0}}
+1 \right) {\bf 1}_{   \sqrt 2 |\xi| >\la \xi_* \ra \ge   |\xi|/2  }
+\left.  \frac{\la \xi_*\ra}{\la \xi \ra }{\bf 1}_{|\xi|/2>\la \xi_* \ra }\right\}\,.  
\end{align}
It follows from 
\eqref{later-use3} and \eqref{later-use4} that
we have 
$$
|\cA_{1}|\lesssim  |\cA_{1,1}| + |\cA_{1,2}|
\lesssim A_1+ A_2 + A_3,
$$
with
\begin{equation}\label{A_1}
A_1=\iint_{\RR^6}
\left |\frac{\hat f(\xi_*)}{\la \xi_*\ra^{3+\gamma}}\right | 
\left |M^\delta (\xi-\xi_*)\hat g(\xi -\xi_*)\right | |\hat h(\xi)|
{\bf 1}_{ \la \xi_* \ra \geq \sqrt 2 |\xi| }d\xi_* d\xi\, ,
\end{equation}
and
\begin{align*}
A_2=&
\iint_{\RR^6}
\left |\frac{\hat f(\xi_*)}{\la \xi_*\ra^{3+\gamma}}\right | \left |M^\delta (\xi-\xi_*)\hat g(\xi -\xi_*)\right | |\hat h(\xi)|\\
&\times\left( {M^\delta (\xi_*)\big(1 + (\delta \la \xi - \xi_*\ra)^{N_0}\big)}
+1 \right) {\bf 1}_{   \sqrt 2 |\xi| >\la \xi_* \ra \ge   |\xi|/2  }d\xi_*d\xi\, ;\\
A_3=&\iint_{\RR^6}
\left |\frac{\hat f(\xi_*)}{\la \xi_*\ra^{3+\gamma}}\right | \left |M^\delta (\xi-\xi_*)\hat g(\xi -\xi_*)\right | |\hat h(\xi)|
\left(  \frac{\la \xi\ra}{\la \xi_* \ra } \right)^{\nu-1}  {\bf 1}_{|\xi|/2>\la \xi_* \ra } d\xi_* d\xi\, .
\end{align*}
Setting $\hat G(\xi) = \la \xi \ra^{\nu'/2} M^\delta (\xi) \hat g(\xi)$ and $\hat H(\xi) = \la \xi \ra^{\nu'/2}\hat h(\xi)$,  we get
\begin{align*}
& A_1 \le \int_{\RR^3} \frac{|\hat H(\xi)|}{\la \xi \ra^{3/2 + \varepsilon}}
\left(\int_{\RR^3} \left( \frac{\la \xi\ra}{\la \xi_* \ra} \right)^{3/2 + \varepsilon-\nu'/2} {\bf 1}_{ \la \xi_* \ra \geq \sqrt 2 |\xi| }
 \frac{|\hat f(\xi_*) |\, |\hat G ( \xi -\xi_*)|}{\la \xi_* \ra^{3/2+\gamma+\nu' -\varepsilon}
}
d\xi_* \right) d\xi \\
& \qquad
\lesssim \|h\|_{H^{\nu'/2}}
\|f\|_{L^2} \|M^\delta g\|_{H^{\nu'/2}} \,, \notag
\end{align*}
because  $3/2 + \varepsilon-\nu'/2\ge 0$ and $3/2+\gamma+\nu' -\varepsilon \ge 0$ for a sufficiently small $\varepsilon >0$. 
Here we have used the fact that $\la \xi_* \ra \sim \la \xi-\xi_*\ra$ if $\la \xi_* \ra \geq \sqrt 2 |\xi|$.
Noticing the third formula of \eqref{equivalence-relation}, we get
\begin{align*}
&\left| A_2 \right|^2 \lesssim
\left \{\int_{\RR^3}\frac{|\hat f(\xi_*)|^2 d\xi_*}{
\la \xi_*\ra^{6+2\gamma +\nu'}}
\int_{\la \xi -\xi_*\ra \lesssim \la \xi_* \ra} \left(
  \frac{\la \xi_* \ra^{-2N_0}}{\la \xi- \xi_* \ra^{\nu'-2N_0}}        +
\frac{1}{\la \xi -\xi_* \ra^{\nu'}} \right) d\xi \right\} \notag \\
&\quad \qquad \times \left( \iint_{\RR^6}|\hat G ( \xi -\xi_*)|^2 |\hat H(\xi)|^2 d\xi d\xi_* \right) \,.\notag\\
&\quad \lesssim
\int_{\RR^3}\frac{|\hat f(\xi_*)|^2} {
\la \xi_*\ra^{3+2(\gamma +\nu')}}
d\xi_* \|M_\lambda^\delta g\|_{H^{\nu'/2}}^2 \|h\|_{H^{\nu'/2}}^2 \lesssim
\|f\|_{L^2}^2 \|M_\lambda^\delta g\|_{H^{\nu'/2}}^2 \|h\|_{H^{\nu'/2}}^2\,,
\end{align*}
because $3+2(\gamma +\nu')>0$.
Since  $\nu' \ge \nu-1$ and $6 +2(\gamma + \nu')>3$, we have
\begin{align*}
&\left| A_3 \right|^2 \lesssim 
\left(\int_{\RR^3} |\hat f(\xi)|^2{d\xi_*}
\int_{\RR^3} |\hat H(\xi)|^2d\xi\right)\\
& \times \left(\int_{\RR^3}  \frac{d\xi_*}{ \la \xi_* \ra^{6+ 2(\gamma +\nu')}}
\int_{\RR^3} \left(\frac{\la \xi_*\ra}{\la \xi \ra}\right)^{2\{2s'-(\nu-1)\}}
{\bf 1}_{ |\xi|/2 \ge \la \xi_* \ra } |\hat G ( \xi -\xi_*)|^2 d\xi
\right)\notag \\
& \qquad
\lesssim
\|f\|_{L^2}^2 \|M_\lambda^\delta g\|_{H^{\nu'/2}}^2 \|h\|_{H^{\nu'/2}}^2\, . \notag
\end{align*}
The above three estimates yield the desired estimate for $\cA_1(f,g,h)$.

Next consider $\cA_2(f,g,h) = \cA_{2,1}(f,g,h) -
\cA_{2,2}(f,g,h)$. Since $\theta \in [0,\pi/2]$ and $|\xi^-|= |\xi| \sin(\theta/2)$ $\geq$
$\la \xi_*\ra/2$, we have $\sqrt 2 |\xi| \geq
\la \xi_*\ra$. Write
\[
\cA_{2,j}= \iint_{\RR^6} K_j(\xi, \xi_*) \Big(M_\lambda^\delta (\xi) - M_\lambda^\delta (\xi-\xi_*)\Big)
\hat f (\xi_* ) \hat g(\xi - \xi_* ) \bar{\hat h} (\xi ) d\xi d\xi_* \,.
\]
Then we have
\begin{align*}
&|K_2(\xi, \xi_*)| = \left|\int  b \Big({\frac{\xi}{ | \xi |}} \cdot \sigma \Big)\hat \Phi_c(\xi_*) {\bf 1}_{ | \xi^- | \ge {\frac12}\la \xi_*\ra } d\sigma\right|\\
& \lesssim  {\frac{1}{\la \xi_* \ra^{3+\gamma }}} \frac{\la  \xi\ra^{\nu} }{\la \xi_*\ra^{\nu}}{\bf 1}_{\sqrt 2 |\xi| \geq \la \xi_* \ra} \notag \\
 &    \lesssim
\frac{1  }{\la \xi_*\ra^{3+\gamma}}\left\{
{\bf 1}_{ \sqrt 2 |\xi| \geq  \la \xi_* \ra \geq |\xi|/2}
+
\left(
\frac{\la \xi \ra}{\la \xi_*\ra}\right)^{\nu}
{\bf 1}_{ |\xi|/2 \ge \la \xi_* \ra }\right\}  \notag \,,
\end{align*}
which shows the desired estimate for $\cA_{2,2}$, by  exactly the same way as
the estimation on $A_2$ and $A_3$.

As for $\cA_{2,1}$, it suffices to work under the condition
 $|\xi_* \cdot \xi^-| \ge \frac1 2 |\xi^-|^2$.
In fact, on the complement of this
set, we have
 $|\xi_* -\xi^-| > | \xi_*|$, and $\hat \Phi_c(\xi_*-\xi^-)$ is
the same as $\hat \Phi_c(\xi_*)$.
Therefore, we consider $\cA_{2,1,p}$ which is defined by replacing $K_1(\xi, \xi_*)$ by
\[
K_{1,p}(\xi,\xi_*) = \int_{\SS^2}
 b \Big({\frac{\xi}{ | \xi |}} \cdot \sigma \Big)
\hat \Phi_c ( \xi_*-\xi^-)
{\bf 1}_{ | \xi^- | \geq {\frac 12} \la \xi_*\ra }{\bf 1}_{| \xi_* \,\cdot\,\xi^-| \ge {\frac1 2} | \xi^-|^2} d \sigma \,.
\]
By writing
\[
{\bf 1}= {\bf 1}_{\la \xi_* \ra \geq |\xi|/2} {\bf 1}_{\la\xi-\xi_* \ra \leq{2}\la \xi_* - \xi^- \ra}
+  {\bf 1}_{\la \xi_* \ra \geq |\xi|/2} {\bf 1}_{\la\xi-\xi_* \ra > {2}\la \xi_* - \xi^-\ra}
+  {\bf 1}_{\la \xi_* \ra < |\xi|/2},
\]
we decompose respectively
\begin{align*}
\cA_{2,1,p}
=
B_1+ B_2 +B_3\,.
\end{align*}
On the sets corresponding to the above integrals, we have $\la \xi_* -\xi^- \ra \lesssim \,
\la \xi_* \ra$, because $| \xi^- | \lesssim | \xi_*|$
that follows {}from  $| \xi^-|^2 \le 2 | \xi_* \cdot\xi ^-| \lesssim |\xi^-|\, | \xi_*|$.
Furthermore, on the sets for $B_1$ and $B_2$  we have $\la \xi \ra \sim \la \xi_* \ra$,
so that
$\la \xi_* -\xi^- \ra \lesssim \ \la \xi \ra$ and $b\,\, {\bf 1}_{ | \xi^- | \ge {\frac12} \la \xi_*\ra } {\bf 1}_{\la \xi_* \ra \geq |\xi|/2}$
is bounded.
Putting again
 $\hat G(\xi) = \la \xi \ra^{\nu'/2} M^\delta (\xi) \hat g(\xi)$ and $\hat H(\xi) = \la \xi \ra^{\nu'/2}\hat h(\xi)$,
by means of \eqref{simple-form}  we have
\begin{align*}
|B_1| ^2
\lesssim& \left[\iiint
\left
|\frac{\hat \Phi_c (\xi_* - \xi^-)}{\la \xi_* - \xi^-\ra^{s'}} \right|^2  | \hat f (\xi_* )|^2 b\,\, {\bf 1}_{ | \xi^- | \ge {\frac1 2} \la \xi_*\ra } {\bf 1}_{\la \xi_* \ra \geq |\xi|/2} \right.\\
&\quad \times
\left\{M^\delta (\xi_*)^2
\left( \frac{{\bf 1}_{ \la \xi -\xi *\ra \lesssim \la \xi_{ *}-\xi^-  \ra}}{
\la \xi-\xi_*\ra^{\nu' } }   +
 \frac{ \delta^{2N_0} {\bf 1}_{ \la \xi -\xi *\ra \lesssim \la \xi_{ *}-\xi^-  \ra}}{
\la \xi-\xi_*\ra^{\nu'-2N_0 } } \right) \right.\\
&\quad \left. + \left.
\frac{{\bf 1}_{ \la \xi -\xi *\ra \lesssim \la \xi_{ *} -\xi^- \ra}}{
\la \xi-\xi_*\ra^{\nu'} } \right \}
d\xi d\xi_* d \sigma \right] \left(\iiint  |    \hat G(\xi - \xi_* )|^2 |{\hat H} (\xi ) |^2 d\sigma d\xi d\xi_*\right) \,.
\end{align*}
Putting $u = \xi_* - \xi^-$, we have 
 $\la u \ra \lesssim \la \xi_*\ra$,
and 
$
M^\delta (\xi_*)^2 \lesssim {(1+ \delta \la u \ra)^{-2N_0}}$.
Therefore,
\begin{align*}
&|B_1 | ^2
\lesssim  \int |\hat f(\xi_*)|^2 \left\{ \sup_u 
{\la u \ra^{-( 6 +2\gamma+\nu')}}   \int_{\la \xi^+ -u \ra \leq \la u \ra} 
 \Big ( \frac{ \delta^{2N_0}{(1+ \delta \la u
\ra)^{-2N_0}} }{
\la \xi^+-u\ra^{\nu'-2N_0 } } + \right. \\
& \qquad \qquad \left. + \frac{1}{ \la
\xi^+-u\ra^{\nu'} } \Big) d\xi^+ \right \}d\xi_* \,\,
 \|M^\delta(D)g\|^2_{H^{\nu'/2}} \|h\|_{H^{\nu'/2}}^2\\
& \qquad \lesssim \|f\|^2_{L^2}  \|M^\delta (D)g\|^2_{H^{\nu'/2}}
\|h\|_{H^{\nu'/2}}^2 \sup_u \frac{1}{\la u \ra^{ 3 +2(\gamma+ \nu')} }\,.
\end{align*}
Here we have used the change of variables
$\xi \rightarrow \xi^+$ whose Jacobian is
\begin{align*}
&\Big|\frac{\partial \xi^+}{\partial \xi} \Big|=\frac{ \Big|I+ \frac{\xi}{|\xi|}\otimes
\sigma\Big|} {8}\\
& =\frac{|1+ \frac{\xi}{|\xi|}\cdot\sigma|}{8}=\frac{\cos^2
(\theta/2)}{4}\ge \frac{1}{8}, \qquad \theta\in [0,\frac{\pi}{2}].  \notag
\end{align*}
As for $B_2$, we first note that, on the set of the  integration,
$\xi^+ = \xi-\xi_* +u$ implies
\[ \frac{\la \xi -\xi_*\ra}{2} \le  \la \xi -\xi_*  \ra - |u| \le \la \xi^+ \ra
\le \la \xi - \xi_*\ra +|u| \lesssim   \la \xi - \xi_*\ra \,,\]
so that
\[(
\enskip M^\delta(\xi) \sim \enskip ) \enskip M^\delta(\xi^+) \sim
M^\delta (\xi-\xi_*)\,,
\]
 and hence we have by the Cauchy-Schwarz inequality
\begin{align*}
|B_2| ^2  \lesssim&  \iiint |\hat f(\xi_*)|^2\, |\hat G(\xi -\xi_*)|^2 d \sigma d\xi d\xi_* \\
& \qquad \times \iiint \frac{
|\hat \Phi_c (\xi_* - \xi^-) |^2  }
{\la \xi_* -\xi^- \ra^{2\nu'}} |{ \hat H} (\xi ) |^2 d\sigma d\xi d\xi_*\\
\lesssim &
\|f\|^2_{L^2}  \|M^\delta(D)g\|^2_{H^{\nu'/2}} \|h\|_{H^{\nu'/2}}^2\,,
\end{align*}
because   $6 +2(\gamma + \nu')>3$.
On the  set of the  integration for  $B_3$ we recall  $\la \xi \ra \sim \la \xi - \xi_*\ra$ and
\[
|M^\delta(\xi) -M^\delta(\xi-\xi_*) |\lesssim \frac{\la \xi_*\ra}{\la \xi \ra} M^\delta(\xi -\xi_*)\,,
\]
so that
\begin{align*}
|B_3 | ^2  \lesssim& \iiint b\,\, {\bf 1}_{ | \xi^- | \ge {\frac12} \la \xi_*\ra }
\left(\frac{\la \xi_* \ra}{\la \xi \ra}\right)^{\nu}
|\hat f(\xi_*)|^2 |\hat G(\xi -\xi_*)|^2 d \sigma d\xi d\xi_* \\
& \times \iiint b\,\, {\bf 1}_{ | \xi^- | \ge {\frac1 2} \la \xi_*\ra }
\left(\frac{\la \xi_* \ra}{\la \xi \ra}\right)^{2-\nu}
\frac{
|\hat \Phi_c (\xi_* - \xi^-) |^2}
{\la \xi\ra^{2\nu'} } |{\hat H} (\xi ) |^2 d\sigma d\xi d\xi_*\,.
\end{align*}
We use the change of variables  $\xi_*  \rightarrow u= \xi_* -\xi^-$.
Note that $| \xi ^-| \ge {\frac1 2} \la u +\xi^-\ra $ implies  $|\xi^-| \geq \la u\ra/\sqrt {10}$,
and that
\[
\la \xi_* \ra \lesssim \la \xi_* - \xi^- \ra + |\xi| \sin \theta/2\,,
\]
which yields
\[
\left(\frac{\la \xi_* \ra}{\la \xi \ra}\right)^{2-\nu}
\lesssim 
\left(\frac{\la u \ra}{\la \xi \ra}\right)^{2-\nu} + \theta^{2-\nu}.
\]
Then we have
\begin{align*}
&
\iint  b\,\, {\bf 1}_{ | \xi^- | \ge {\frac12} \la \xi_*\ra }  
\left(\frac{\la \xi_* \ra}{\la \xi \ra}\right)^{2-\nu}
\frac{
|\hat \Phi_c (\xi_* - \xi^-) |^2}
{\la \xi\ra^{2\nu'}} d\sigma d\xi_*
\lesssim
\int
\frac{{\bf 1}_{\la u\ra \lesssim |\xi|}}{\la u \ra^{6+2(\gamma +\nu')}}\left( \frac {\la u \ra}
{\la \xi \ra}\right)^{2\nu'} \\
&\qquad \qquad \times
\Big( \int  b\, {\bf 1}_{ | \xi^- |  \gtrsim \la u \ra } \left(\frac{\la u \ra}
{\la \xi \ra}\right)^{2-\nu} d\sigma +
\int  b \theta^{2-\nu}  {\bf 1}_{ | \xi^- |  \gtrsim \la u \ra } d\sigma
  \Big)du\\
&\qquad \qquad \lesssim \int \frac{du} {\la u \ra^{6+2(\gamma +\nu')}} < \infty\,,
\end{align*}
because
\begin{align*}
\int  b \theta^{2-\nu}  {\bf 1}_{ | \xi^- |  \gtrsim \la u \ra } d\sigma
\left \{
\begin{array}{ll}
\lesssim \left( \frac {\la u \ra}
{\la \xi \ra}\right)^{2-2\nu} \enskip &\mbox{if $\nu >1$}\\
\lesssim \log \frac {\la \xi\ra} {\la u \ra}&\mbox{if $\nu = 1$}\\
< \infty \enskip &\mbox{if $\nu <1$}.
\end{array}
\right.
\end{align*}
Thus we have the same bound for $B_3$. 
The proof of the proposition is then completed.
\end{proof}
Let us recall Proposition 2.9 {}from \cite{AMUXY2010}.

\begin{prop}\label{prop2.9_amuxy3} 
Let $M(\xi)$ be a positive symbol in $S^{0}_ {1,0}$ in the form of
$M(\xi) = \tilde M(|\xi|^2)$.  Assume that there exist
constants $c, C>0$ such that
 for any $s, \tau>0$
$$
c^{-1}\leq \frac{s}{\tau}\leq c \,\,\,\,\,\,\mbox{implies}
\,\,\,\,\,\,\,\,C^{-1}\leq \frac{\tilde M(s)}{ \tilde M(\tau)}\leq C,
$$
and  $M(\xi)$ satisfies
$$
|M^{(\alpha)}(\xi)| = |\partial_\xi^\alpha M(\xi)| \leq C_{\alpha}
M(\xi) \la \xi \ra^{-|\alpha|}\, ,
$$
for any $\alpha\in\NN^3$.  Then,  if $0<\nu<1$, for any $N >0$ there exists a $C_N >0$ such that
\begin{align}\label{10.8-2}
&\left|( M(D_v) Q_{\bar c}(f,\, g)-  Q_{\bar c}(f,\, M(D_v) g) ,\,\, h)_{L^2}\right | \hskip4cm
\notag \\
&\qquad \qquad \leq C_N \|f\|_{L^1_{\gamma^+}} \Big(\|M(D_v)\,
g\|_{L^2_{\gamma^+}} + \| g \|_{H^{-N}_{\gamma^+}} \Big)
\|h\|_{L^2}.
\end{align}
Furthermore, if $1 < \nu <2$,  for any $N>0$ and any $\varepsilon >0$ , there exists
a  $C_{N, \varepsilon}>0 $ such that
\begin{align}\label{10.8-3}
\left |(M(D_v) Q_{\bar c}(f,\, g)-Q_{\bar c}(f,\, M(D_v) g),\,\, h)_{L^2} \right |  \hskip4cm \notag \\
\leq C_{N, \varepsilon} \|f\|_{L^1_{(\nu+ \gamma-1)^+}} \Big(
\|M(D_v) g\|_{H^{\nu-1+\varepsilon} _{(\nu+ \gamma-1)^+}}  +
\|g\|_{H^{-N}_{\gamma^+}}\Big)
 \|h\|_{L^2}\, . 
\end{align}
When $\nu = 1$ we have the same  estimate  as (\ref{10.8-3}) with
$(\nu+ \gamma-1)$ replaced by $(\gamma+ \kappa)$ for any small
$\kappa >0$.
\end{prop}
\subsection{Boundedness of $M^\delta(D_v)$ on the triple norm}\label{ap-5}

Instead of $M^\delta(\xi)$,  we consider a little more general symbol $M(\xi) \in 
S^0_{1,0}$ satisfying conditions in Proposition \ref{prop2.9_amuxy3}.

\begin{lem}\label{M-bounded-triple}
Let $\gamma >-3$ and $0 < \nu <2$. Then we have 
\begin{align*}
|\!|\!|M(D_v) g|\!|\!|^2 \lesssim |\!|\!|g|\!|\!|^2.
\end{align*}
\end{lem}
\begin{proof}
The proof is based on arguments in the subsection 2.3 of \cite{AMUXY2}.
Since $J_2^{\Phi_\gamma}(g) \sim \|g\|^2_{L^2_{(\nu+\gamma)/2}}$
it suffices to show 
\begin{align}\label{bound-M-tri}
 J_1^{\Phi_\gamma}(M g) \lesssim J_1^{\Phi_\gamma}(g) + \|g\|^2_{L^2_{(\nu+\gamma)/2}}
+ \|g\|^2_{H^s_{\gamma/2}}\,,
\end{align}
in view of Proposition 2.2 of \cite{AMUXY2}, that is,
there exist two generic constants $C_1, C_2>0$ such that
\[
C_1 \left\{\left\|
g\right\|^2_{H^s_{\gamma/2}(\RR^3_v)}+\left\|
g\right\|^2_{L^2_{(\nu+\gamma)/2}(\RR^3_v)}\right\}
\leq |\!|\!| g |\!|\!|^2 \leq  C_2 \left\|
g\right\|^2_{H^s_{(\nu+\gamma)/2}(\RR^3_v)}\,.
\]
Since $J_1^{\Phi_\gamma} (g)  \sim J_1^{\Phi_0} (\la v \ra^{\gamma/2} g) $ modulo
$ \|g\|^2_{L^2_{(\nu+\gamma)/2}}$, and we have
\begin{align*}
 J_1^{\Phi_0} (\la v \ra^{\gamma/2}M g) &\le 2 \Big(J_1^{\Phi_0} (M \la v \ra^{\gamma/2} g)
+ J_1^{\Phi_0} ([\la v \ra^{\gamma/2}, M] g )\Big)\\
&\lesssim J_1^{\Phi_0} (M \la v \ra^{\gamma/2} g) + \|([\la v \ra^{\gamma/2}, M] g\|^2_{H^s_s}\,,
\end{align*}
it suffices to show \eqref{bound-M-tri} with $\gamma =0$. We recall (2.21) of \cite{AMUXY2};
\begin{align}\label{part1-intermid}
J_1^{\Phi_0}(g)&=\iiint b(\cos\theta)  {\mu} _\ast
 ( g'-  g)^2dv_*d\sigma dv \notag \\
&=\frac{1}{(2\pi)^3}\iint b \Big(\frac{\xi}{|\xi|}\cdot \sigma \Big)
\Big(\widehat {\mu} (0) | \widehat {g} (\xi ) -
\widehat { g} (\xi^+) |^2 \\
& \qquad \qquad \qquad + 2 \textrm{Re}\,
\Big(\widehat{ \mu }(0) - \widehat { \mu} (\xi^-) \Big)
 \widehat {g} (\xi^+ )
\overline{\widehat {g}} (\xi ) \Big)d\xi d\sigma . \notag
\end{align}
Since $\widehat { \mu}(\xi)$
is real-valued,
it follows that
\begin{align*}
\textrm{Re}\, \Big(\widehat { \mu} (0)
- \widehat { \mu} (\xi^-) \Big) \widehat { g} (\xi^+ )
\overline{\widehat { g} } (\xi )
&= \Big (\int
\big (1- \cos (v\cdot \xi^-)\big)
\mu(v) dv\Big)\,
\textrm{Re}\,\widehat { g} (\xi^+ )
\overline{\widehat { g} } (\xi )\\
&\lesssim \min \{ \la \xi \ra^2 \theta^2 \,, \, 1 \} |\widehat { g} (\xi^+ )
\overline{\widehat { g} } (\xi )|.
\end{align*}
Therefore the second term of the right hand side of \eqref{part1-intermid} is estimated by
$\|g\|^2_{H^{\nu/2}}$ with a constant factor.  Similarly we have 
\begin{align*}
J_1^{\Phi_0}( M g)&=\iiint b(\cos\theta)  {\mu} _\ast
 ( (Mg)'-  Mg)^2dv_*d\sigma dv\\
&\lesssim \iint b \Big(\frac{\xi}{|\xi|}\cdot \sigma \Big)
\widehat {\mu} (0) | M(\xi) \widehat {g} (\xi ) -M(\xi^+)
\widehat { g} (\xi^+) |^2 d\xi d\sigma +\|g\|^2_{H^{\nu/2}} \\
&\lesssim J_1^{\Phi_0}(g) +  \iint b \Big(\frac{\xi}{|\xi|}\cdot \sigma \Big)
 | M(\xi)  -M(\xi^+)|^2|\widehat {g} (\xi )|^2d\xi d\sigma +\|g\|^2_{H^{\nu/2}}\\
&\lesssim J_1^{\Phi_0}(g) +\|g\|^2_{H^{\nu/2}}\,,
\end{align*}
because $ | M(\xi)  -M(\xi^+)|\lesssim M(\xi) \theta^2$ ( see (2.3.5) of \cite{AMUXY2010},
for example).
\end{proof}

\subsection{Mollifier with respect to $x$ variable}\label{ap-6}

\begin{lem}\label{x-moll}
Let $S \in C_0^\infty(\RR)$ satisfy $0 \le S \le 1$ and 
\[
S(\tau) = 1, \enskip |\tau| \le 1; \enskip S(\tau) = 0, \enskip |\tau| \ge 2.
\]
Put  $S_\delta(D_x) = S(\delta D_x)$ for $\delta >0$.  Then
\begin{align}\label{1st}
&\left|\int_0^T \right. 
\Big(S_\delta(D_x)\left. \Gamma(f, g) - \Gamma(f, S_\delta(D_x)g), \, \, h\Big)_{L^2(\RR^6_{x,v})}dt \right|
\notag \\
& \lesssim 
\|\nabla f \|_{L^\infty([0,T]\times \RR_x^3; L^2_v )}
\|\la v \ra^{|\gamma|/2 + \nu} g\|_{L^2([0,T]\times \RR^6_{x,v})}\notag\\
&\qquad \qquad \qquad \qquad \times 
\|\delta h\|_{L^2([0,T]\times \RR^3_x; 
H^{\nu}_{\gamma/2}(\RR^3_v))}.
\end{align}

\end{lem}
\begin{proof} The proof is similar as the one of Lemma 3.4 in \cite{AMUXY2010}.
If $$K_\delta(z) = \displaystyle   -\frac{z}{\delta^4}
\cF^{-1}(S)\Big(\frac{z}{\delta}\Big)$$ then we have 
\begin{align*}
&\left|\int_0^T 
\Big(S_\delta(D_x)\Gamma(f, g) - \Gamma(f, S_\delta(D_x)g), \, \, h\Big)_{L^2(\RR^6)}dt \right|\\
&=\left| \int_0^1 \left\{
\int_{\RR_x^3 \times \RR_y^3} K_\delta (x-y) \right.\right. \\
& \left. \left. \times \int_0^T \Big(\Gamma \big (\nabla f(t, x + \tau(y-x)), \cdot), \delta g(t, y, \cdot)\big), \, 
h(t, x, \cdot)\Big)_{L^2(\RR^3_v)} dt dx dy\right \}d\tau\right|\\
&\lesssim \delta \|\nabla f \|_{L^\infty([0,T]\times \RR_x^3; L^2_v)}\int_0^T\int_{\RR^3_x}
\Big(| K_\delta | * \|g(t,\cdot) \|_{L^2_{|\gamma|/2 +\nu}}  \Big)(x)  \, \|h(t,x) \|_{H^\nu_{\gamma/2}} dx dt\,
\end{align*}
where we have used 
\eqref{different-g-h}.  We get \eqref{1st} since $\|K_\delta\|_{L^1_x} = \|K_1\|_{L^1_x} $. 
\end{proof}

\bigskip

\noindent
{\bf Acknowledgements:} The research of the first author was supported in part
by  Grant-in-Aid for Scientific Research No.25400160,
Japan Society for the Promotion of Science.

\end{document}